\documentclass[11pt]{article}

\usepackage{amssymb,amsmath, amsthm}
\usepackage{setspace}
\usepackage{multirow}

\usepackage{amssymb}
\usepackage{framed}
\usepackage{mathrsfs}
\usepackage{enumerate}
\usepackage{tikz}
\usetikzlibrary{matrix}
\usepackage[all]{xy}
\usetikzlibrary{shapes}
\usepackage{hyperref}
\usepackage{setspace}

\usepackage{bm}
\usepackage[algo2e,ruled,linesnumbered]{algorithm2e}
\let\oldnl\nl 
\newcommand{\nonl}{\renewcommand{\nl}{\let\nl\oldnl}}
\usepackage{algpseudocode}
\usepackage{cite} 

\newcommand{\cO}{\mathcal{O}}
\newcommand{\cT}{\mathcal{T}}
\newcommand{\cG}{\mathcal{G}}
\newcommand{\cA}{\mathcal{A}}
\newcommand{\cL}{\mathcal{L}}
\newcommand{\RR}{\mathbb{R}}

\newcommand{\argmin}{\operatornamewithlimits{argmin}}

\newtheorem{theorem}{Theorem}[section]
\newtheorem{condition}{Condition}[section]
\newtheorem{corollary}[theorem]{Corollary}
\newtheorem{lemma}[theorem]{Lemma}

\newtheorem{definition}[theorem]{Definition}
\newtheorem{example}{Example}[section]
\newtheorem{assumption}[theorem]{Assumption}
\newtheorem{remark}[theorem]{Remark}

\usepackage{mathtools}

\title{First-Order Algorithms Without Lipschitz Gradient:\\ A Sequential Local Optimization Approach}
\author{Junyu Zhang$^{\dag\dag}$ and Mingyi Hong$^\dag$ \thanks{$^{\dag\dag}$ Department of Industrial Systems Engineering and Management, National University of Singapore, junyuz@nus.edu.sg \\
	    $^\dag$ Department of Electrical and Computer Engineering, University of Minnesota, Twin Cities,  mhong@umn.edu }} 

\begin{document}
	\date{}
	\maketitle
	
\begin{abstract}
	Most first-order methods rely on the global Lipschitz continuity of the objective gradient, which fails to hold in many problems. This paper develops a sequential local optimization (SLO) framework for first-order algorithms to optimize problems without Lipschitz gradient. Operating on the assumption that the gradient is {\it locally} Lipschitz continuous over any compact set, SLO develops a careful scheme to control the distance between successive iterates. The proposed framework can easily adapt to the existing first-order methods such as projected gradient descent (PGD), truncated gradient descent (TGD), and a parameter-free variant of Armijo line search. We show that SLO requires $\mathcal{O}(\epsilon^{-1} \mathcal{L}_1(Y))$ gradient evaluations to find an $\epsilon$-stationary point, where $Y$ is certain compact set with  $\cO(\epsilon^{-1/2})$ radius, and $\mathcal{L}_i(Y)$ denotes the Lipschitz constant of the $i$-th order derivatives in $Y$. It is worth noting that our analysis provides the first non-asymptotic convergence rate for the (slight variant of) Armijo line search algorithm without globally Lipschitz continuous gradient or convexity.  As a generic framework, we also show that SLO can incorporate more complicated subroutines such as a variant of the accelerated gradient descent (AGD) method that can harness the problem's second-order smoothness without Hessian computation, which achieves an improved  $\tilde{\mathcal{O}}\big(\epsilon^{-7/8}\mathcal{L}_1^{1/2}(Y)\mathcal{L}_2^{1/4}(Y)\big)$ complexity. 
\end{abstract}

\section{Introduction}
In this paper, we consider the nonconvex and non-gradient-Lipschitz (No-Grad-Lip) optimization problem $\min\{f(x): x\in\mathbb{R}^n\},$ where  $f(\cdot)$ is nonconvex, continuously differentiable and globally lower bounded, yet its gradient $\nabla f(\cdot)$ may not be globally Lipschitz continuous. 
Many practical examples exhibit  such challenging No-Grad-Lip property, which are listed below. 
\begin{example}[Tensor Decomposition]
	\label{example:SymTD}
	Let $\cT\in\RR^{d\times \cdots\times d}$ be a $k$-dimensional tensor. Then the symmetric tensor decomposition in CP format, see e.g. \cite{kolda2009tensor,comon2008symmetric}, solves the following problem
	\begin{equation}
		\label{prob:SymTD}
		\mathop{\mathrm{minimize}}_{x_1,\gamma_1,...,x_m,\gamma_m} \Big\|\cT - \sum_{i=1}^m\gamma_i\cdot \underbrace{x_i\otimes\cdots\otimes x_i}_{k\,\,x_i's}\Big\|^2,
	\end{equation}
	where $x_i\in\RR^d$ are vectors, $\gamma_i$ are scalars, and $\otimes$ denotes the Kronecker product.
\end{example} 
The Hessian of the objective function grows unbounded as the $\|x_i\|$'s grow, which prevents a globally Lipschitz gradient. Similar issues also exist in non-symmetric tensor/matrix factorization. 

\begin{example}[Unsupervised Autoencoder Training]
	\label{example:autoencoder}
	Let $X\in\RR^{d\times N}$ be the data matrix and let $W_i$, $1\leq i \leq m$ be $m$ weight matrices. Then the autoencoder training \cite{hinton2006reducing} can be formulated as
	\begin{equation}
		\label{prob:autoencoder}
		\mathop{\mathrm{minimize}}_{W_1,...,W_m}\,\, \Big\|X - \sigma\big(W_m\cdots\sigma\big(W_2\,\sigma\big(W_1 X\big)\big)\cdots\big)\Big\|^2,
	\end{equation}
	where $\sigma(\cdot)$ is an activation function that applies elementwisely to any matrix input.
\end{example}
\begin{example}[Supervised Neural Network Training]
	\label{example:supervised-dnn}
	Let $X\in\RR^{d\times N}$ be the data matrix, let $Y\in\RR^{N}$  be the labels vector of $X$, and let $W_i$, $1\leq i \leq m$ be $m$ weight matrices. Then the supervised training of neural network can be written as 
	\begin{equation} 
		\label{prob:supervised-dnn}
		\mathop{\mathrm{minimize}}_{W_1,...,W_m}\,\, \Big\|Y - \sigma\big(W_m\cdots\sigma\big(W_2\,\sigma\big(W_1 X\big)\big)\cdots\big)\Big\|^2
	\end{equation} 
\end{example}
For both Examples \ref{example:autoencoder} and \ref{example:supervised-dnn}, the Hessian explodes as the weights grow unbounded. Consider a single data point toy example of $\min_{w_1,w_2} (\sigma(w_2\sigma(w_1x))-y)^2$ with $(x,y)=(1,0)$, $w_1,w_2\in\RR$, and a smoothed Relu activation $\sigma(s):=\log(1+e^s)$. Direct computation shows that the gradient Lipschitz constants are globally unbounded and vary drastically across different regions. And this phenomenon is also recently verified numerically in training general neural nets, see \cite[Sec. H]{zhang2019gradient-clipping}. \vspace{0.2cm}

\noindent\textbf{Existing Works.} $\,\,$ In convex optimization, there have been a radial reformulation that transforms a convex No-Lip-Grad function to an equivalent convex function with Lipschitz gradient \cite{grimmer2018radial,grimmer2023radial}, and Frank-Wolfe approaches relying on self-concordance condition \cite{carderera2021simple,dvurechensky2022generalized,zhao2020analysis}. It is worth noting that \cite{grimmer2023radial} also extends the radial reformulation method to nonconvex objective functions that satisfy certain strict upper radial property. Though the resulting algorithm does not require the global gradient Lipschitz constant $L$, its convergence analysis does rely on the existence such a constant. For nonconvex No-Grad-Lip problems, several works \cite[etc.]{Bauschke17,Bolte18,Lu18Relative} proposed the notion of $L$-smooth adaptable ($L$-smad) functions, based on which a generalized descent lemma was proved and a Bregman proximal gradient (BPG) method is proposed. Since then, many extensions of BPG have been developed to handle the No-Grad-Lip issue \cite{Reem19,radualex2019optimal,Teboulle18,li2019provable,bompaire2018dual,Hien19, hanzely2018accelerated,davis18high-order}. For example, an inertial variant of BPG was proposed in \cite{hanzely2018accelerated,Hien19} and a stochastic  BPG was developed in \cite{davis18high-order}. Though most existing nonconvex No-Lip-Grad works  focus on BPG type methods, they suffer a few limitations:\vspace{0.2cm}

\noindent\textbf{(1).} BPG methods critically depend on identifying the $L$-smad function, which is only known for a few function classes such as the log-determinant of Fisher information matrix \cite{hanzely2018accelerated} and functions with polynomially growing Hessian \cite{Bolte18,gao2019leveraging,ch2019bregman}, etc. Applications beyond these classes are limited. It is meaningful to derive alternative frameworks that does not rely on the $L$-smad functions. \vspace{0.2cm}

\noindent\textbf{(2).} BPG methods only provide an $\cO(1/T)$ convergence of the Bregman divergence $D_h(x^{k+1},x^k)$, while how $D_h(x^{k+1},x^k)$ relates to $\|\nabla f(x^k)\|^2$ has yet to be discussed. Consider $d$-th order polynomial problems such as Example \ref{example:autoencoder} and \ref{example:supervised-dnn} with $\sigma(x)\equiv x$, \cite{Lu18Relative,li2019provable} suggest setting $h(x) = \frac{\alpha}{2}\|x\|^d_2 + \frac{\sigma}{2}\|x\|^2_2,$ for some $\alpha, \sigma>0$. In this case, we show that $\frac{\|\nabla f(x^k)\|^2}{\|x^k\|^{d-2}}\big/D_h(x^{k+1},x^k) = \Theta(1)$ under appropriate conditions. If BPG can show all its iterations until finding an $\epsilon$-solution are bounded by some $R_\epsilon>0$ that may potentially depend on $\epsilon$, then BPG will result in an $\cO\left(\epsilon^{-1}R_\epsilon^{d-2}\right)$ complexity for finding $\epsilon$-solution. Unfortunately, existing BPG analysis does not provide such an $R_\epsilon$. In fact, for order-$d$ polynomial optimization, if $\|x^k\|^{d-2}$ grows faster than $\Omega(k)$, then the $\cO(1/k)$ convergence of $D_h(x^{k+1},x^k)$ does not imply the convergence of $\|\nabla f(x^k)\|^2$ without additional assumptions.   
In this paper, we design an algorithm and the corresponding analysis techniques that guarantees such a  bounded $R_\epsilon$ for the iterates until finding an $\epsilon$-solution.    \vspace{0.2cm}

\noindent\textbf{Contributions.}$\,\,$ We list the main contribution of our paper as follows. 
\begin{itemize}
	\item We develop a sequential local optimization (SLO)  framework to solve the No-Lip-Grad problems. It can easily incorporate existing first-order methods like PGD and TGD, providing an alternative to BPG method when an appropriate $L$-smad function is not available.
	\item We show that algorithms covered by the SLO framework only require certain {\it local} gradient Lipschitz constants to determine the algorithmic parameters. In case such local knowledge is not accessible, we design a completely parameter-free method that combines the Armijo line search (see e.g. \cite{armijo1966minimization,bertsekas99,nocedal2006numerical}) with a novel normalized search upper bound. To our best knowledge, this result is the first finite iteration convergence rate analysis for the Armijo type line search algorithm without global Lipschitz gradient or convexity conditions. 
	\item We perform rigorous iteration complexity analysis to the  algorithms covered by SLO. To find an $\bar x$ s.t. $\|\nabla f(\bar x)\|^2\leq \epsilon$, most of the proposed methods need $\mathcal{O}(\epsilon^{-1} \mathcal{L}_1(Y))$ gradient evaluations, where $Y$ is some compact set with $\cO(\epsilon^{-1/2})$ radius that contains all the iterations until finding $\bar{x}$. Moreover, by constructing a hard instance for which all the $\epsilon$-solutions are $\Omega(\epsilon^{-1/2})$ away from the initial point, we prove the tightness of our radius upper bound of the set $Y$. 
	\item We show that our SLO framework is generic enough to incorporate complicated subroutines such as an AGD variant enhanced with the ``convex-until-proven-guilty'' technique, which further improves the complexity dependency on $\epsilon$ and $\mathcal{L}_1(Y)$ to $\epsilon^{-7/8}$ and $\left(\mathcal{L}_1(Y)\right)^{1/2}$, respectively.
\end{itemize} 

\section{The Sequential Local Optimization framework}
In this section, we  formally setup the basic assumptions, the Sequential Local Optimization (SLO) framework, and a convergence analysis of the general SLO framework.
\subsection{The basic assumptions}
To formalize the discussion, we make the following assumptions to the objective function $f(\cdot)$. 
\begin{assumption}
	\label{assumption:LowerBound}
	{\bf (1).} $f(\cdot)$ is lower bounded: $f^*:=\inf_x f(x)>-\infty$. {\bf (2).} $f(\cdot)$ is continuously differentiable, and $\nabla f(x)$ is locally Lipschitz continuous everywhere but not globally Lipschitz continuous. Specifically, $\nabla f(\cdot)$ is $\cL_1(C)$-Lipschitz continuous within any \textbf{compact} set $C\subsetneq \mathbb{R}^n$; w.l.o.g.,  we assume $\cL_1(C)\ge 1, \forall~ C\subsetneq \mathbb{R}^n$;
	{\bf (3).} For any ball area $B(x,r)$ with moderate radius $r = \cO(1)$,  we assume $\cL_1(B(x,r))$ is either explicitly known or can be estimated efficiently. 
\end{assumption}
For designing accelerated method, we also assume the second-order local Lipschitz continuity. 
\begin{assumption}
	\label{assumption:growth_fun_2}
	$f(\cdot)$ is second-order continuously differentiable. For any compact set $C\subsetneq\mathbb{R}^n$, the Hessian $\nabla^2 f(\cdot)$ is $\cL_2(C)$-Lipschitz continuous in $C$: $\|\nabla^2 f(x) - \nabla^2 f(y)\|\le \cL_2(C)\|x-y\|, \forall~x, y\in C.$
	W.l.o.g., we assume $\cL_2(\cdot)\geq1$. For any ball area $B(x,r)$ with moderate radius $r = \cO(1)$, we assume $\cL_2(B(x,r))$ is either explicitly known or can be estimated efficiently.
\end{assumption}

These assumptions are very general and are satisfied by many problems including Examples \ref{example:SymTD}, \ref{example:autoencoder} and \ref{example:supervised-dnn}. It is also easy to see that $\cL_1(C_1)\leq \cL_1(C_2),$ $\cL_2(C_1)\leq \cL_2(C_2)$, for $\forall~C_1\subseteq C_2\subsetneq \RR^n$. Finally, it is worth noting that estimating $\cL_1(C)$ or $\cL_2(C)$ for arbitrary $C$ is not easy in general. Yet in this paper, we only require its accessibility in small ball area with $O(1)$ radius, making the heuristic estimation of $\cL_1(C)$ or $\cL_2(C)$ much easier. Our experiments also verify this point. Moreover, in many problems, these constants can actually be  computed analytically.

\subsection{The SLO framework}
Before describing the SLO framework, let us illustrate why the classical GD analysis fails in No-Grad-Lip problems. By Assumption \ref{assumption:LowerBound}, $\nabla f(\cdot)$ is $\cL_1\big(B(x_0,r)\big)$-Lipschitz continuous within any ball $B(x_0,r)$. Consider the GD update $x_{t+1} = x_t - \frac{\nabla f(x_t)}{\cL_1\!(B(x_0,r))}.$
If we set $R_T \!\geq\! \max_{0\leq t\leq T}\{\|x_t-x_0\|\}$, then $\{x_t\}_{t=0}^T\subset B(x_0,R_T)$. Setting  $r = R_T$ in the GD update and applying classical analysis yields 
$$f(x_{t+1}) - f(x_t) \leq - \frac{1}{2\cL_1\big(B(x_0,R_T)\big)}\|\nabla f(x_t)\|^2\quad\mbox{for}\quad 0\leq t\leq T-1,$$
see e.g. \cite{bertsekas99}. Averaging the above inequality for $t = 0,1,...,T-1$ gives
\begin{align}\label{eqn:GD-bound}
	\min_{0\leq t\leq T-1} \left\{\|\nabla f(x_t)\|^2\right\} \leq \frac{2\cL_1\big(B(x_0,R_T)\big)\cdot (f(x_0) - f(x^*))}{T}.
\end{align}
This derivation recovers the $\cO(1/T)$ convergence of GD \cite{bertsekas99} if $\cL_1(\cdot)$ is constant. However, if $\cL_1\big(B(x_0,R_T)\big)$ depends on the size of  $B(x_0,R_T)$ and the level set of $f(\cdot)$ is unbounded, the above argument no longer holds. This is because $R_T$ and $\cL_1\big(B(x_0,R_T)\big)$ may grow unbounded as $T\to\infty$. 

To resolve this issue, we propose the following {\it sequential local optimization} (SLO) framework (Algorithm \ref{alg:Meta}), which is an episodic meta scheme that runs a certain subroutine $\cA$ to perform optimization within a predetermined compact region where one can efficiently estimate the Lipschitz constants to initialize $\cA$. The framework is carefully designed so that SLO can largely preserve the structure of each subroutine $\cA$, which enables SLO to incorporate subroutines with complicated mechanism and analyze them under a unified framework.  

\begin{algorithm2e}
	\DontPrintSemicolon
	\caption{A Sequential Local Optimization (SLO) Framework}
	\label{alg:Meta}
	\textbf{input:} \hspace{0.1cm}
	initial point $x_0^1$; tolerance $\epsilon>0$; constants $D>0$, $d\geq0$; subroutine $\cA$.=\\
	\textbf{default:} $L_1^\tau = L_2^\tau = \textbf{null}$, $\tau = 1,2,3,\cdots$\\	
	\For{$\tau = 1,2,3,\cdots$}{ 
		\textbf{if} $\,\,L_1^\tau$ is required by $\cA\,\,$ \textbf{then} $\,\,\text{Estimate}\,\, L_1^\tau = \cL_1\big(B(x^\tau_0,D)\big)$ \,\,\textbf{end}\\
		\textbf{if} $\,\,L_2^\tau$ is required by $\cA\,\,$ \textbf{then} $\,\,\text{Estimate}\,\, L_2^\tau = \cL_2\big(B(x^\tau_0,D)\big)$ \,\,\textbf{end}\\
		\For{$k = 0,1,2,\cdots$}{ 
			\textbf{if}$\,\,\|\nabla\! f(x^\tau_k)\|\!<\!\sqrt{\epsilon}\,\,$  \textbf{then}\,\, \textbf{return}($x^\tau_k$)\,\, \textbf{end} \,\,\,\,\,\,/*Terminate SLO if $\|\nabla\! f(x^\tau_k)\|$ is small*/\\ 
			Generate $x^{\tau}_{k+1} = \cA\left(f, x^\tau_{k}; L_1^{\tau},L_2^{\tau}; x^\tau_0,D,d\right),$ 
			which satisfies 
			\begin{equation}\label{eq:truncated}
				\|x^{\tau}_{0} - \cA\left(f, x^\tau_{k}; L_1^{\tau},L_2^{\tau}; x^\tau_0,D,d)\right\|\le D. 
			\end{equation}
			
			\textbf{if}$\,\,\|x^\tau_{k+1} \!-  x^\tau_0\|\!\in\![D-d, D]\,\,$ \textbf{then}\,\, Set $x^{\tau+1}_0 = x^{\tau}_{k+1}$ and \textbf{break} the for loop\,\, \textbf{end}\\
			/****End the $\tau$-th epoch, if $x^\tau_{k+1}$ enters a margin of $B(x^\tau_0,D)$****/
		}
	} 
\end{algorithm2e}  

In details, Algorithm \ref{alg:Meta} is divided into epochs based on the distance from $x^\tau_0$, the first point of each epoch $\tau$.  In each epoch, we fix a region $B(x^\tau_0, D)$ and estimate Lipschitz constants $L_1^{\tau}$ (and $L_2^{\tau}$ if needed) for the subroutines. Then the classical first-order methods remain valid in this region, except that, they may need to be truncated in order not move far away from $x^\tau_0$ (cf. \eqref{eq:truncated}). If $x^\tau_k$ is not $\epsilon$-stationary, then it either remains in $B(x^\tau_0,D-d)$ or it enters $\{x: \|x-x^\tau_0\|\in[D-d, D]\}$. For the latter case, we  end the epoch and set it as the first point of the next epoch. Next, let us provide  a generic convergence analysis for Algorithm \ref{alg:Meta} and analyze several first-order subroutines.

\subsection{Convergence analysis of the SLO framework.}
To analyze SLO, we require the subroutine $\cA$ to satisfy the following generic descent condition. 
\begin{condition}
	\label{condition:sufficient} {\it \bf (Sufficient Descent)}
	Let $L_1^{\tau}, L_2^{\tau}$ and $\{x^\tau_k\}^{K_\tau}_{k=0}$ be generated by a subroutine $\cA$ during the $\tau$-th epoch of Algorithm \ref{alg:Meta}. Suppose $x^\tau_{K_\tau}$ is the last iterate of its epoch. Then there exist functions $C_1\left(\cdot\right)>0$ and $C_2\left(\cdot\right)>0$ such that 
	\begin{equation}
		\label{cond-1}
		f(x^\tau_k) - f(x^\tau_{k-1}) \leq - C_1\left(L_1^{\tau}, L_2^{\tau}, \epsilon\right)\cdot\|x^\tau_k - x^\tau_{k-1}\|^2, \quad\mbox{for}\quad k = 1,2,\cdots,K_\tau, 
	\end{equation}
	\begin{equation}
		\label{cond-2}
		f(x^\tau_k) - f(x^\tau_{k-1})\leq-C_2\left(L_1^{\tau}, L_2^{\tau}, \epsilon\right),\quad\mbox{for}\quad k = 1,2,\cdots,K_\tau-1. 
	\end{equation}
\end{condition}
\noindent For notational simplicity, we denote  
$C_1^\tau:=C_1\left(L_1^{\tau}, L_2^{\tau}, \epsilon\right)$ and $C_2^\tau:= C_2\left(L_1^{\tau}, L_2^{\tau}, \epsilon\right)$ throughout the following sections. Based on Condition \ref{condition:sufficient}, we next show that a descent of $\frac{\sqrt{C_1^\tau  C_2^\tau}}{2} D$ is guaranteed if the iterates of epoch $\tau$ approaches the boundary of the local region, which always happens if an $\epsilon$-stationary point is not discovered by Algorithm \ref{alg:Meta} during this epoch.
\begin{lemma}[Per-epoch descent]
	\label{lemma:meta-epoch}
	Let $\{x^\tau_k\}^{K_\tau}_{k=0}$ be   generated by Algorithm \ref{alg:Meta} as epoch $\tau$, with  $\|x^\tau_{\!K_\tau}\!-\!x^\tau_{0}\|\!\in\!\left[D-d,\! D\right]$. 
	If Condition \ref{condition:sufficient} holds and $D\!\geq\!\sqrt{\frac{C_2^\tau}{4C_1^\tau}} \!+\! 2d$, then  
	$f(x^\tau_{K_\tau}) \!-\! f(x^\tau_0) \!\leq \!-\! \frac{\sqrt{C_1^\tau  C_2^\tau}}{2}D.$
\end{lemma}
\noindent\proof{Proof.}
For $k = 1,2,\cdots,K_\tau-1$, summing up \eqref{cond-1} and \eqref{cond-2} respectively yields 
$$f(x^\tau_{K_\tau-1}) - f(x^\tau_0)   \leq    -C^\tau_1\sum_{k=1}^{K_\tau-1}\big\|x^\tau_k-x^\tau_{k-1}\big\|^2
\stackrel{(i)}\leq  -\frac{C^\tau_1}{K_\tau-1}\big\|x^\tau_{K_\tau-1}-x^\tau_{0}\big\|^2.$$
$$f(x^\tau_{K_\tau-1}) - f(x^\tau_0)  \leq -(K_\tau-1)
C^\tau_2.$$  
where $(i)$ uses Jensen's inequality. 
Combining the above two inequalities yields
$$f(x^\tau_{K_\tau-1}) - f(x^\tau_0) \leq -\frac{1}{2}\left(\frac{C^\tau_1}{K_\tau-1}\big\|x^\tau_{K_\tau-1}-x^\tau_{0}\big\|^2 + (K_\tau-1) 
C_2^\tau\right)\overset{(i)}{\leq} - \sqrt{C_1^\tau  C_2^\tau}\big\|x^\tau_{K_\tau-1}-x^\tau_{0}\big\|,$$
where (i) is due to $\frac{1}{2}(a+b)\geq \sqrt{ab}$, $\forall a,b\geq0$.
For the last iteration, \eqref{cond-2} of Condition \ref{condition:sufficient} implies
$f(x^\tau_{K_\tau}) - f(x^\tau_{K_\tau-1})\leq -C_1^\tau \big\|x^\tau_{K_\tau} - x^\tau_{K_\tau-1}\big\|^2.$
Consequently, we have
\begin{eqnarray} 
	f(x^\tau_{K_\tau}) - f(x^\tau_0) & \leq & -\sqrt{C_1^\tau  C_2^\tau}\big\|x^\tau_{K_\tau-1}-x^\tau_{0}\big\| -C_1^\tau\big\|x^\tau_{K_\tau} - x^\tau_{K_\tau-1}\big\|^2\nonumber\\
	& \leq & -\sqrt{C_1^\tau C_2^\tau}\big \|x^\tau_{K_\tau-1}-x^\tau_0\big\| -C_1^\tau \Big(\big\|x^\tau_{K_\tau} - x^\tau_{0}\big\|-\big\|x^\tau_{K_\tau-1}-x^\tau_0\big\|\Big)^2\nonumber\\
	& \leq & - \min_{\omega\geq0}\left\{\sqrt{C_1^\tau  C_2^\tau}\cdot\omega + C_1^\tau \Big(\big\|x^\tau_{K_\tau} - x^\tau_{0}\big\|-\omega\Big)^2 \right\}.\nonumber
\end{eqnarray}
Because $\|x^\tau_{\!K _\tau}\!- x^\tau_{0}\|\geq \sqrt{\frac{C^{\tau}_2}{4C^{\tau}_1}}$, the minimum is achieved at $\omega^* = \|x^\tau_{K_\tau} - x^\tau_{0}\|-\sqrt{\frac{C^{\tau}_2}{4C^{\tau}_1}}$.Therefore, 
\begin{eqnarray*}
	\label{lm:meta-epoch-1}
	f(x^\tau_{K_\tau}) - f(x^\tau_0) \leq  -\sqrt{C_1^\tau  C_2^\tau}\cdot\omega^* - C_1^\tau \Big(\big\|x^\tau_{K_\tau} - x^\tau_{0}\big\|-\omega^*\Big)^2 = - \sqrt{C_1^\tau  C_2^\tau} \big\|x^\tau_{K_\tau} - x^\tau_{0}\big\| + \frac{C_2^\tau}{4}.
\end{eqnarray*}
Note that $\|x^\tau_{K_\tau} - x^\tau_{0}\| \geq D - d,$ and  $D\geq\frac{1}{2}\sqrt{\frac{C_2^\tau}{C_1^\tau}} + 2d$, the above inequality  reduces to
$$\!\!\!\!\!f(x^\tau_{\!K_{\!\tau}}\!)  -\! f(x^\tau_0) \leq\! -\sqrt{C_1^\tau  C_2^\tau} \left(\!\|x^\tau_{\!K_{\!\tau}} \!\!- x^\tau_{0}\| \!- \frac{1}{4}\sqrt{\frac{C_2^\tau}{C_1^\tau}}\,\right)\!\!\leq\! -\sqrt{C_1^\tau C_2^\tau}\left(\!D  - \frac{1}{4}\sqrt{\frac{C_2^\tau}{C_1^\tau}}- d\right)\!\!\leq\! -\frac{\sqrt{C_1^\tau C_2^\tau}}{2} D.$$
This completes the lemma.  $\qquad\qquad\qquad\qquad\qquad$ $\qquad\qquad\qquad\qquad\qquad$ $\qquad\qquad\qquad\qquad$ 
\endproof

Let  $N_\cA^\tau$ denote the number of gradient evaluations in each execution of $\cA$ in epoch $\tau$. For example, if $\cA$ is a single gradient step, then $N_\cA^\tau = 1$; if $\cA$ compute the output by solving a subproblem, then $N_\cA^\tau$ is the number of gradient evaluations used to solve this subproblem. With this notation, we provide the following theorem. 
\begin{theorem}
	\label{theorem:meta-complexity}
	Let Assumption \ref{assumption:LowerBound} hold. If $L_2^\tau$ is required by $\cA$ in addition, we also require Assumption \ref{assumption:growth_fun_2} to hold. Let $\big\{x^\tau_k\big\}_{k=0,...,K_\tau}^{\tau = 1,2,\cdots}$ be generated by the Algorithm \ref{alg:Meta}. If the algorithmic parameters and the subroutine $\cA$ can simultaneously guarantee Condition \ref{condition:sufficient} and the inequality
	\begin{equation}
		\label{thm:meta-complexity-0}
		\sum_{\tau=1}^\infty \frac{\sqrt{C_1^\tau C_2^\tau}}{2} D = +\infty,\qquad\mbox{with}\qquad D\geq \frac{1}{2}\sqrt{\frac{C_2^\tau}{C_1^\tau}} + 2d,\quad \forall \tau,
	\end{equation}
	then it takes at most $T$ epochs to output a point $x^\tau_k$ s.t. $\|\nabla f(x^\tau_k)\|^2\leq\epsilon$, where $T$ satisfies
	\begin{equation}
		\label{thm:meta-complexity-1}
		T \leq \inf \left\{\tau+1: \sum_{k=1}^\tau \frac{\sqrt{C_1^k C_2^k}}{2} D\geq f(x^1_0) - f^*\right\} <+\infty.
	\end{equation}
	The total number of gradient evaluations of Algorithm \ref{alg:Meta} is at most  
	\begin{equation}
		\label{thm:meta-complexity-2}
		\sum_{\tau=1}^T K_\tau N_\cA^\tau \leq \max_{1\leq \tau\leq T}\left\{\frac{N_\cA^\tau}{C_2^\tau}\right\}\cdot \left(f(x^1_0)-f^*\right) + \sum_{\tau=1}^{T}N_\cA^\tau.
	\end{equation}
	Its dependence on $\epsilon$ and $\cL_1(\cdot)$ (and $\cL_2(\cdot)$ if used) is characterized by $T$, $C_1^\tau$ and $C_2^\tau$.
\end{theorem}
\proof{Proof.}
By the termination criterion of Algorithm \ref{alg:Meta}, we know that $\|\nabla f(x^\tau_k)\|\geq\sqrt{\epsilon}$ for $0\leq k\leq K_\tau$ and $1\leq \tau\leq T-1$. Therefore, for the first $T-1$ epochs during which the algorithm doesn't terminate, Lemma \ref{lemma:meta-epoch} indicates that 
$$f(x^\tau_{K _\tau}) - f(x^\tau_0) \leq -\sqrt{C_1^\tau C_2^\tau}D/2\qquad\mbox{for} 1\leq \tau\leq T-1.$$ 
Because $x^{\tau+1}_0:=x^\tau_{K_\tau}$ in Algorithm \ref{alg:Meta}, summing up this inequality yields 
$$\sum_{\tau=1}^{T-1} \frac{\sqrt{C_1^\tau C_2^\tau}}{2} D \leq f(x^1_0) - f(x^{T-1}_{K_{T-1}})\leq f(x^1_0) - f^*.$$
By the first inequality in \eqref{thm:meta-complexity-0}, $T$ must be finite and inequality \eqref{thm:meta-complexity-1} holds true. Note that except for the last iterate of each epoch, all the other iterates satisfy \eqref{cond-2} of Condition \ref{condition:sufficient}.  Therefore, the maximum number of iterations of the $\tau$-th epoch satisfies
$$K_\tau \leq \frac{x^\tau_{K_\tau} - f(x^\tau_0)}{C_2^\tau} + 1 = \frac{f(x^{\tau+1}_0) - f(x^\tau_0)}{C_2^\tau} + 1.$$
The total iteration complexity satisfies the following relation
\begin{eqnarray*}
	\sum_{\tau=1}^T \left(\frac{f(x^\tau_0)-f(x^{\tau}_{K_\tau})}{C_2^\tau} + 1\right) N_\cA^\tau & \leq & \max_{1\leq \tau\leq T}\left\{\frac{N_\cA^\tau}{C_2^\tau}\right\}\cdot\sum_{\tau=1}^{T}\left(f(x^\tau_0)-f(x^{\tau}_{K_\tau})\right) + \sum_{\tau=1}^{T}N_\cA^\tau\\
	& \leq & \max_{1\leq \tau\leq T}\left\{\frac{N_\cA^\tau}{C_2^\tau}\right\}\cdot \left(f(x^1_0)-f^*\right) + \sum_{\tau=1}^{T}N_\cA^\tau,
\end{eqnarray*} 
which completes the proof. $\qquad\qquad\qquad\qquad\qquad$ $\qquad\qquad\qquad\qquad\qquad$ $\qquad\qquad\qquad\qquad$  
\endproof
The above result relies on the satisfaction of Condition \ref{condition:sufficient} and \eqref{thm:meta-complexity-0}. So far it contains generic parameters $C^{\tau}_1$, $C^{\tau}_2$, $N^\tau_{\cA}$ that are difficult to interpret. Next, we verify that Condition \ref{condition:sufficient} for several variants of the GD algorithm, and specialize the above generic result. 

\subsection{The gradient projection subroutine}
The first specific subroutine $\cA$ is the gradient projection algorithm, which is defined as follows. 
\begin{definition}[Gradient projection subroutine]
	For a differentiable function $f$, constants $\eta,r >0$, and a point $x\in B(\bar x,r)$, we define the gradient projection subroutine as 
	\begin{eqnarray}\label{eq:A:GP}
		\cA_{GP}(f,x,\bar x,\eta,r) :=  \begin{cases}
			x - \eta\nabla f(x), & \mbox{ if } \quad x - \eta\nabla f(x) \in B(\bar x,r),\\
			\bar x + \frac{x - \eta\nabla f(x) - \bar x}{\|x - \eta\nabla f(x) - \bar x\|}\cdot r, & \mbox{otherwise,}
		\end{cases}
	\end{eqnarray}
	which is the analytical form of projecting a gradient descent step to a Euclidean ball $B(\bar x, r)$.
\end{definition}
Note that the nature of the gradient projection  does not require any second-order information of $f$, thus $L_2^\tau$ remains its default value in Algorithm \ref{alg:Meta}, i.e., $L_2^\tau = \textbf{null}$. We also set $d= 0 $ and let the subroutine $\cA$ in Algorithm \ref{alg:Meta} be the gradient projection subroutine:
\begin{equation}
	\label{eqn:meta-GP}
	\cA\left(f,x^\tau_k;L^{\tau}_1,L_2^\tau;x^\tau_0, D,d\right): = \cA_{GP}\left(f,x^\tau_k,x^\tau_0,\left(L^{\tau}_1\right)^{-1}, D\right).
\end{equation}
The finite travel condition \eqref{eq:truncated} naturally holds due to the projection. It remains to verify Condition \ref{condition:sufficient} and to calculate the functions $C_1(\cdot)$ and $C_2(\cdot)$, which is provided as follows. 
\begin{lemma}
	\label{lemma:condition-GP}
	Let $\big\{x^\tau_k\big\}_{k=0}^{K_\tau}$ be generated by Algorithm \ref{alg:Meta} in the $\tau$-th epoch, with $d = 0$ and $L_2^\tau = \emph{\textbf{null}}$. Suppose Assumption \ref{assumption:LowerBound} holds and the subroutine $\cA$ is set according to \eqref{eqn:meta-GP}.  If $\|\nabla f(x^\tau_k)\|>\sqrt{\epsilon}$ for $0\leq k\leq K_\tau-1,$ then Condition \ref{condition:sufficient} holds with 
	$C_1\left(L^{\tau}_1,L_2^\tau,\epsilon\right) = \frac{L_1^\tau}{2}$ and $C_2(L^{\tau}_1,L_2^\tau,\epsilon) = \frac{\epsilon}{2L_1^\tau},$
	where $L_1^\tau : = \cL_1\big(B(x^\tau_0,D)\big)$ is the local gradient Lipschitz constant.
\end{lemma} 
\proof{Proof.}
First, let us prove \eqref{cond-1} for $k = 1,2,\cdots ,K_\tau$. Note that the gradient projection $\cA_{GP}(\cdot)$ can be equivalently written as the following optimization problem:
\begin{equation}
	\label{lm:condition-GP-0}
	x^\tau_{k+1} = \argmin_{\|x-x^\tau_0\|\leq D} f(x^\tau_k) \!+\! \left\langle\nabla f(x^\tau_k), x- x^\tau_k\right\rangle \!+\! \frac{L_1^\tau}{2}\|x - x^\tau_k\|^2.
\end{equation} 
Note that the function $f$ has $L_1^\tau$-Lipschitz continuous gradient in $B(x^\tau_0,D)$, and the subroutine $\cA_{GP}(\cdot)$ guarantees that the iterates generated in $\tau$-th epoch are restricted in $B(x^\tau_0, D)$. It follows that the Lipschitz continuity property within the set $B(x^\tau_0,D)$ indicates 
\begin{eqnarray}
	\label{lm:condition-GP-1}
	f(x^\tau_{k+1}) \leq f(x^\tau_{k}) + \langle \nabla f(x^\tau_k),x^\tau_{k+1}-x^\tau_k\rangle + \frac{L_1^\tau}{2}\|x^\tau_{k+1} - x^{\tau}_k\|^2.
\end{eqnarray}
By the $L_1^\tau$-strong convexity of the optimization problem in \eqref{lm:condition-GP-0}, we also have 
\begin{eqnarray}
	\label{lm:condition-GP-2}
	f(x^\tau_{k}) + \langle \nabla f(x^\tau_k),x^\tau_{k+1}-x^\tau_k\rangle + \frac{L_1^\tau}{2}\|x^\tau_{k+1} - x^{\tau}_k\|^2 \leq  f(x^\tau_{k}) -\frac{L_1^\tau}{2}\|x^\tau_{k+1} - x^\tau_k\|^2. 
\end{eqnarray}
Combining \eqref{lm:condition-GP-1} and \eqref{lm:condition-GP-2} proves \eqref{cond-1} with $C_1\left(L^{\tau}_1,L_2^\tau,\epsilon\right)  = \frac{L_1^\tau}{2}$. 

Next, we prove \eqref{cond-2} for $k \leq K_\tau-1$. Since   epoch $\tau$ ends at the iteration $K_\tau$, by the termination rule of each epoch, we know $\|x^\tau_{k}-x^\tau_0\|< D-d = D$ for $k\leq K_\tau-1$. This indicates that 
$$x^\tau_{k} = x^\tau_{k-1} - \frac{1}{L_1^\tau}\nabla f(x^\tau_{k-1})\in \mathrm{int}\big(B(x^\tau_0, D)\big), \quad k = 1,2,\cdots, K_\tau-1.$$
This is because if $x^\tau_{k'-1} - \frac{1}{L_1^\tau}\nabla f(x^\tau_{k'-1})\notin \mathrm{int}\big(B(x^\tau_0, D)\big)$ for some $k'\leq K_\tau-1$, the $x^\tau_{k'-1}$ will be on the boundary of $B(x^\tau_0,D)$ due to the projection step, leading to a contradiction. That is, the gradient projection subroutine just reduces to the gradient descent step for $k\leq K_\tau-1$. On the other hand, because $\|f(x^\tau_{k})\|>\sqrt{\epsilon}$ for $k\leq K_\tau-1$, we have 
\begin{eqnarray}
	\label{lm:condition-GP-3}
	f(x^\tau_k) & = & f\left(x^\tau_{k-1} - \frac{1}{L_1^\tau}\nabla f(x^\tau_{k-1})\right)\\
	& \leq & f\left(x^\tau_{k-1} \right) + \left\langle\nabla f(x^\tau_{k-1}), - \frac{1}{L_1^\tau}\nabla f(x^\tau_{k-1})\right\rangle + \frac{L^\tau_1}{2}\left\|\frac{1}{L_1^\tau}\nabla f(x^\tau_{k-1})\right\|^2\nonumber\\
	& = & f(x^\tau_{k-1}) - \frac{1}{2L_1^\tau}\left\|\nabla f(x^\tau_{k-1})\right\|^2 \leq   f(x^\tau_{k-1}) - \frac{\epsilon}{2L_1^\tau}\,\,\,.\nonumber
\end{eqnarray}
In conclusion, we prove \eqref{cond-2} with $C_2(L_1^\tau,L_2^\tau,\epsilon) = \frac{\epsilon}{2L_1^\tau}$.  $\qquad\qquad\qquad\qquad\qquad\qquad\qquad\qquad\quad$ 
\endproof
\noindent Substituting $C_1^\tau = \frac{L_1^\tau}{2}$, $C_2^\tau = \frac{\epsilon}{2L_1^\tau}$ and $d = 0$ into Lemma \ref{lemma:meta-epoch} yields the following corollary.

\begin{corollary}
	\label{corollary:GP-epoch-descent}
	Under the settings of Lemma \ref{lemma:condition-GP}. If $\|\nabla f(x^\tau_k)\|>\sqrt{\epsilon}$ for $0\leq k\leq K_\tau-1,$  then $f(x^\tau_{K_\tau}) - f(x^\tau_0) \leq - \frac{\sqrt{\epsilon}D}{4}$
	as long as we choose $D\geq \frac{\sqrt{\epsilon}}{2}$.
\end{corollary}

\noindent Denote $\Delta_f := f(x^1_0) - f^*$, and $R_T:=\frac{4\Delta_f}{\sqrt{\epsilon}}+D.$ We summarize the complexity result as follows.
\begin{theorem}
	\label{theorem:meta-complexity-GP}
	Suppose Assumption \ref{assumption:LowerBound} holds. Consider Algorithm \ref{alg:Meta} with the gradient projection subroutine $\cA_{GP}$ given in \eqref{eq:A:GP}. If we choose $d = 0$, $D\geq\frac{\sqrt{\epsilon}}{2}$, and specify the parameters by \eqref{eqn:meta-GP}, then  it takes at most $T = \left\lceil\frac{4\Delta_f}{\sqrt{\epsilon} D}+1\right\rceil$ epochs to output a point $x^\tau_k$ s.t. $\|\nabla f(x^\tau_k)\|^2\leq\epsilon$. Moreover, the total number of gradient evaluations is at most  
	$\sum_{\tau=1}^T K_\tau \leq \cO\left(\frac{\Delta_f}{\epsilon}\cdot\cL_1\left(B(x^1_0,R_T)\right)\right).$
\end{theorem} 

\proof{Proof.}
To prove the above result, we only need to specialize Theorem \ref{theorem:meta-complexity} using the constants we obtained for the gradient projection subroutine $\cA_{GP}$. By using the values of $C^\tau_1$ and $C^\tau_2$ in Lemma \ref{lemma:condition-GP}, we know
$\sum_{\tau=1}^\infty \frac{\sqrt{C_1^\tau C_2^\tau}}{2}D = \sum_{\tau=1}^\infty \frac{\sqrt{\epsilon}D}{4} = +\infty.$
Therefore, the first relation \eqref{thm:meta-complexity-0} is satisfied. Also our choice of $D$ satisfies the second inequality in \eqref{thm:meta-complexity-0}. It then follows from \eqref{thm:meta-complexity-1}  that the upper bound for the number of epochs is 
\begin{align}
	T = \left\lceil\frac{2 (f(x^1_0)- f(x^*))}{D \sqrt{C^{\tau}_1 C^{\tau}_2}}+ 1\right\rceil = \left\lceil \frac{4 (f(x^1_0)- f(x^*))}{\sqrt{\epsilon}\cdot D} + 1\right\rceil.
\end{align}
Second, because $\|x^\tau_k - x^\tau_0\| \leq D$ and $x^{\tau+1}_0 = x^\tau_{K_\tau}$, the triangle inequality indicates that 
$$\|x^\tau_0-x^1_0\|\leq\sum_{t=1}^{\tau-1}\|x^{t}_{K_t}-x^t_0\|\leq (\tau-1) D\leq (T-1)D = \frac{4\Delta_f}{\sqrt{\epsilon}}, \quad\quad \forall \tau.$$  
That is, $B(x^\tau_0,D)\subseteq B\big(x^1_0,\frac{4\Delta_f}{\sqrt{\epsilon}} + D\big) = B(x^1_0,R_T)$, we have
$$\max_{1\leq \tau\leq T}\left\{\frac{N_\cA^\tau}{C_2^\tau}\right\} \overset{(i)}{=} \max_{1\leq \tau\leq T}\left\{\frac{2L_1^\tau}{\epsilon}\right\} \leq \frac{2}{\epsilon}\cdot\cL_1\big(B\big(x^1_0,R_T\big)\big),$$
where (i) utilizes the expression of $C_2^\tau$ in Lemma \ref{lemma:condition-GP}, and $N^{\tau}_{\cA}=1$ because  $\cA_{GP}$ only evaluates the gradient once.
Substituting the above inequality into Theorem \ref{theorem:meta-complexity} proves the current result. $\,\,$ 
\endproof
Though any $D\geq\frac{\sqrt{\epsilon}}{2}$ is sufficient for convergence, a small $D = \cO(\sqrt{\epsilon})$ may cause conservative progress and hence too many epochs. This further leads to too many local Lipschitz constant estimations to initialize the epochs.  It will be better to set $D = \Omega(1)$ for numerical efficiency. 

\subsection{The truncated gradient descent (TGD) subroutine}
In addition to the gradient projection subroutine, we also propose a truncated gradient descent (TGD) subroutine to solve the problem, which is formally described as follows. 
\begin{definition}[Truncated gradient descent subroutine]
	For any differentiable function $f$, constant $L,d >0$, and a point $x$. We define the truncated gradient descent subroutine as 
	\begin{eqnarray}
		\cA_{TG}(f,x,L,d) :=  \begin{cases}
			x - \frac{1}{L}\cdot\nabla f(x), & \mbox{ if } \quad\|\nabla f(x)\|\leq L\cdot d,\\
			x - d\cdot\frac{\nabla f(x)}{\|\nabla f(x)\|}, & \mbox{otherwise}.
		\end{cases}
	\end{eqnarray}
\end{definition}
\noindent The following notation is used to denote the TGD subroutine:
\begin{equation}
	\label{eqn:meta-TG}
	\cA\left(f,x^\tau_k;L^{\tau}_1,L_2^\tau; x^\tau_0, D,d\right): = \cA_{TG}\left(f,x^\tau_k,L^{\tau}_1,d\right).
\end{equation}
Compared to the standard normalized gradient descent (NGD) designed for convex No-Lip-Grad problem \cite{grimmer2019convergence}: $x_{t+1}=x_t-d\frac{\nabla f(x)}{\|\nabla f(x)\|}$, which moves a distance of size $d$ in each iteration, the TGD subroutine \eqref{eqn:meta-TG} only truncates the update if a pure gradient descent with stepsize $1/L^{\tau}_1$ exceeds the moving distance limit of $d$. This adaptation over the standard NGD enables TGD to guarantee \eqref{eq:truncated} and Condition \ref{condition:sufficient} simultaneous, which facilitates the analysis of nonconvex No-Lip-Grad problem.  
\begin{lemma}
	\label{lemma:condition-TG}
	Let $\{x^\tau_k\}_{\!k=0}^{\!K_\tau}$ be generated by Algorithm \ref{alg:Meta} as the $\tau$-th epoch, with $d \!\geq\! \sqrt{\epsilon}$, $L_2^\tau \!=\! \emph{\textbf{null}}$. Under Assumption \ref{assumption:LowerBound} and the TGD subroutine \eqref{eqn:meta-TG}, if $\|\nabla f(x^\tau_k)\|>\sqrt{\epsilon}$ for $0\leq k\leq K_\tau-1,$ then Condition \ref{condition:sufficient} holds with $C_1\left(L^{\tau}_1,L_2^\tau,\epsilon\right) = \frac{L_1^\tau}{2}$ and $C_2(L^{\tau}_1,L_2^\tau,\epsilon) = \frac{\epsilon}{2L_1^\tau}.$
\end{lemma} 
\proof{Proof.} We divide the proof  in the following two cases.\\	
\textbf{Case 1.} $\|\nabla f(x^\tau_{k-1})\|\leq L^\tau_1 d$. In this case, $x^\tau_{k}-x^\tau_{k-1} = -\frac{1}{L_1^\tau}\nabla f(x^\tau_{k-1})$. Because all the iterates of this epoch remains in $B(x^\tau_0, D)$, the $L_1^\tau$-Lipschitz continuity of $\nabla f(\cdot)$ in $B(x^\tau_0,D)$ indicates 
\begin{eqnarray*}
	f(x^\tau_{k})  &\leq &  f(x^\tau_{k-1})- \frac{1}{2L_1^\tau}\|\nabla f(x^\tau_{k-1})\|^2  =  f(x^\tau_{k-1})-\frac{L_1^\tau}{2}\|x^\tau_{k}-x^\tau_{k-1}\|^2.
\end{eqnarray*}
Since the $\|\nabla f(x^\tau_{k-1})\|>\sqrt{\epsilon}$, inequality \eqref{lm:condition-GP-3} is still true. That is, $f(x^\tau_{k}) \leq f(x^\tau_{k-1}) -\frac{\epsilon}{2L_1^\tau}.$ 
Combining the above results proves the result under Case 1.\\
\textbf{Case 2.} $\|\nabla f(x^\tau_{k-1})\|>L_1^\tau d$. Under this case we have 
\begin{align}\label{eq:normalized:step}
	x^\tau_k - x^\tau_{k-1} = -d\cdot\frac{\nabla f(x^\tau_{k-1})}{\|\nabla f(x^\tau_{k-1})\|}, \quad \|x^\tau_k - x^\tau_{k-1}\| = d.
\end{align}
It follows that 
\begin{eqnarray*} 
	f(x^\tau_k) & \leq & f\left(x^\tau_{k-1} \right) + \left\langle\nabla f(x^\tau_{k-1}),x^\tau_k - x^\tau_{k-1} \right\rangle + \frac{L^\tau_1}{2}\left\|x^\tau_k - x^\tau_{k-1}\right\|^2\nonumber\\
	& = & f(x^\tau_{k-1}) - d\|\nabla f(x^\tau_{k-1})\| + \frac{L^\tau_1}{2}\cdot d^2\nonumber\\
	& \leq & f(x^\tau_{k-1}) - \frac{L^\tau_1d^2}{2}   \stackrel{\eqref{eq:normalized:step}}=    f(x^\tau_{k-1}) - \frac{L^\tau_1}{2}\|x^\tau_k - x^\tau_{k-1}\|^2.
\end{eqnarray*}
Since we assumed $L_1^\tau\geq1$,  then $d\geq \sqrt{\epsilon}\geq \sqrt{\epsilon}/L_1^\tau$,  we also have $f(x^{\tau}_{k}) -f(x^\tau_{k-1})  \leq - \frac{L_1^\tau d^2}{2} \leq  - \frac{\epsilon}{2L_1^\tau}.$
Combining the above  inequalities proves the lemma under Case 2. $\qquad\qquad\qquad\qquad\quad\,\,\,\,\quad\,\,\,\,$ 
\endproof

\noindent Because $C_1(\cdot)$ and $C_2(\cdot)$ for the $\cA_{TG}$ subroutine are identical to those of the $\cA_{GP}$ subroutine. Therefore, Theorem \ref{theorem:meta-complexity-GP} still holds true  with slightly different choices of parameters $d$ and $D$. 
\begin{theorem}
	\label{theorem:meta-complexity-TG}
	Suppose the Assumptions \ref{assumption:LowerBound} holds true. Consider Algorithm \ref{alg:Meta} with the truncated gradient subroutine $\cA_{TG}$  specified by \eqref{eqn:meta-TG}. If we choose $d \geq \sqrt{\epsilon}$, $D\geq\frac{\sqrt{\epsilon}}{2} + 2d$, then  it takes at most $T = \left\lceil\frac{4(f(x^1_0) - f^*)}{\sqrt{\epsilon} D}+1\right\rceil$ epochs to output a point $x^\tau_k$ s.t. $\|\nabla f(x^\tau_k)\|^2\leq\epsilon$. Moreover, the total computational complexity can be upper bounded by 
	$\sum_{\tau=1}^T K_\tau \leq \cO\left(\frac{\Delta_f}{\epsilon}\cdot\cL_1\left(B(x^1_0,R_T)\right)\right),$
	where we denote $\Delta_f := f(x^1_0) - f^*$, and $R_T:=\frac{4\Delta_f}{\sqrt{\epsilon}}+D.$
\end{theorem}
Like our comment of Theorem \ref{theorem:meta-complexity-GP}, we suggest a constant scale $D$ for better numerical efficiency.
\begin{remark}
	Let $B(x_0^1,R_T^*)$ be the smallest ball containing all the iterates, then the complexities of the above methods are actually $\cO(\epsilon^{-1}\cL_1(B(x_0^1,R_T^*)))$. In this paper we simply adopt the  pessimistic upper bound of $R_T^*$ as $R_T \!=\! \sum_{\tau=1}^T\|x^\tau_{K_\tau}\!-\!x^\tau_{0}\| \!\leq\! \cO\big(\frac{\Delta_f}{\sqrt{\epsilon}}\big)$. In practice, we expect $R_T^*\ll R_T$.  
\end{remark}

In fact, one can implement the TGD subroutine step-by-step without meta outer loops: $x_{k+1} = \cA_{TG}\big(f,x_k,\cL_{1}(B(x_k,\delta)),\delta\big)$, and the same complexity can be derived by adopting some key techniques from the generic SLO analysis. It seems that the No-Grad-Lip issue can be easily tackled by a carefully truncated stepsize rule that controls the moving distance while guaranteeing the sufficient descent, which poses a question on the necessity of such a framework. However, as discussed at the beginning of the SLO framework enjoys much better flexibility because it largely preserves the original iteration structure of the subroutines, which enables us to design and analyze subroutines with complicated mechanism, including the subroutine with Convex Until Proven Guilty (CUPG) acceleration, which will be discussed in later sections.

\section{A parameter-free line search subroutine}
\label{sec:lsearch}

Previous subroutines all rely on good estimations of $\cL_1^\tau$. In this section, we introduce a completely {\it parameter-free} subroutine that equips the classical Armijo linesearch with \emph{normalized} search region. Suppose we have a linesearch direction $v$ that is   \emph{gradient related} for some $\alpha\in(0,1]$, that is,
\begin{equation}
	\label{eqn:lsearch-grc}
	\langle\nabla f(x),v\rangle \leq -\alpha \|\nabla f(x)\|\cdot\|v\|.
\end{equation}
The the basic idea is to find some $\delta>0$ s.t. the Armijo  condition (see e.g. \cite{armijo1966minimization,bertsekas99,nocedal2006numerical}) hold:
\begin{equation}
	\label{eqn:lsearch-rule}
	f(x + \delta\cdot v) \leq f(x) + \sigma\delta\langle\nabla f(x),v\rangle,  \quad \sigma\in(0,1)
\end{equation}
If $\nabla f(\cdot)$ is globally Lipschitz continuous, the $\cO(\epsilon^{-1})$ complexity of Armijo linesearch is well-known. However, to our best knowledge, only asymptotic convergence without explicit rate can be derived for classical Armijo linesearch with a fixed search interval $\delta\in[0,\bar \delta]$, for nonconvex and No-Lip-Grad problems, see e.g. \cite{bertsekas99,nocedal2006numerical}. In this paper, we propose a simple but novel normalized search interval $\big[0, \bar{\delta}/\|v_t\|\big]$, where $v_t$ is the search direction at  iteration $t$, see details in Algorithm \ref{alg:lsearch}. Adopting the techniques of the SLO analysis, we provide the first finite-time non-asymptotic convergence result for the Armijo linesearch in nonconvex and No-Grad-Lip problems. 

\begin{algorithm2e}
	\DontPrintSemicolon  
	\caption{A parameter-free Normalized Armijo subroutine: $\cA_{LS}(f,x,\bar{\delta})$}
	\label{alg:lsearch}
	\textbf{default:} Constants $\theta\in(0,1)$, $\sigma\in(0,1)$, $\bar\delta>0$.\\
	\textbf{input:}  An objective function $f$ and a point $x$. \\
	Find some update direction $v\neq 0 $ such that \eqref{eqn:lsearch-grc} hold and initialize $\delta=\frac{\bar{\delta}}{\|v\|}$.\\
	\While{$\mathbf{true}$}{
		\textbf{if}\,\, $f(x + \delta\cdot v)> f(x) + \sigma\delta\langle\nabla f(x),v\rangle$\,\, \textbf{then} \,\, set $\delta = \delta\cdot\theta$\\
		\textbf{else}\,\, \textbf{return}$(x + \delta\cdot v)$\,\, \textbf{end} 
	} 
\end{algorithm2e}

It can be observed that the intuition behind the $\cA_{LS}$ subroutine is similar to that of the $\cA_{TG}$ subroutine, i.e.,  the distance between two successive iterates is at most $\bar{\delta}$ given the normalized search upper bound. Therefore, some {\it local} Lipschitz constants can be utilized in the complexity analysis. Again, for notational simplicity, we rewrite the line search subroutine as below:
\begin{equation}
	\cA(f,x^\tau_k;L_1^\tau ,L_2^\tau;x^\tau_0,D ,d) = \cA_{LS}(f,x^\tau_k,\bar \delta) \quad\mbox{with}\quad L_1^\tau=L_2^\tau=\textbf{null}, D = d = \bar \delta.
\end{equation}
With $D = d = \bar \delta$, Algorithm \ref{alg:Meta} implies that each epoch ends immediately after one single iteration.  We can remove the epoch superscripts and simply write Algorithm \ref{alg:Meta} as $x_{t+1} = \cA_{LS}(f,x_t,\bar \delta).$ Due to the simplicity of this scheme, we provide an independent analysis that shares the same spirit of Condition \ref{condition:sufficient}, but with a much simpler form. 
\begin{lemma}
	\label{lemma:lsearch} Consider Algorithm \ref{alg:lsearch}. Suppose Assumption \ref{assumption:LowerBound} hold, then the Armijo condition \eqref{eqn:lsearch-rule} can be guaranteed  as long as we set
	$\delta\leq\min\bigg\{\frac{(1\!-\!\sigma)\alpha\|\nabla f(x)\|}{\cL_1\big(B(x, \bar \delta)\big)\|v\|}, \; \frac{\bar \delta}{\|v\|}\bigg\}$.
\end{lemma} 
\proof{Proof.} 
First, because $\delta\leq \frac{\bar \delta}{\|v\|}$, all points along the search path belong to the 
following set:
$$\big\{x':x' = x + \delta v, \; 0\leq\delta\leq \bar \delta/\|v\|\big\}\subset B(x,\bar \delta).$$ 
Thus $\nabla f(\cdot)$ is $\cL_1 (B(x,\bar \delta) )$-Lipschitz continuous in $B(x,\bar \delta)$. Denote $L_1 = \cL_1 (B(x,\bar \delta) )$, we have 
\begin{eqnarray*}
	f(x + \delta\cdot v) & \leq & f(x)  + \delta\langle\nabla f(x),v\rangle + \frac{L_1\delta^2}{2}\|v\|^2\\
	&  \overset{(ii)}{\leq} & f(x) +\sigma\delta\langle\nabla f(x),v\rangle - (1-\sigma)\alpha\delta\|\nabla f(x)\|\|v\|  + \frac{L_1\delta^2}{2}\|v\|^2\\
	& = & f(x) +\sigma \delta\langle\nabla f(x),v\rangle-\frac{\alpha\delta\|\nabla f(x)\|\|v\|}{2}\left(2(1-\sigma)  - \frac{L_1\|v\|\cdot\delta}{\alpha\|\nabla f(x)\|}\right),
\end{eqnarray*}
where (ii) is due to \eqref{eqn:lsearch-grc} and the fact that $\sigma\in (0,1)$.
If we set $\delta\leq\min\{\frac{(1-\sigma)\alpha\|\nabla f(x)\|}{L_1\cdot\|v\|},\frac{\bar \delta}{\|v\|}\}$, then $2(1-\sigma)  - \frac{L_1\|v\|\cdot\delta}{\alpha\|\nabla f(x)\|}\geq0$ and \eqref{eqn:lsearch-rule} is satisfied. $\qquad\qquad\qquad\qquad\qquad\qquad\qquad\qquad\qquad\qquad\,\quad\,\,\,\,$ 
\endproof
\begin{corollary}
	\label{corollary:lsearch}
	Suppose Assumption \ref{assumption:LowerBound} and \eqref{eqn:lsearch-grc} hold. Consider Algorithm \ref{alg:Meta} with the $\cA_{LS}$ subroutine, with $\bar{\delta}$ satisfying $\bar{\delta}>\sqrt{\epsilon}$. Suppose  $\|\nabla f(x_t)\| > \sqrt{\epsilon}$ for iteration $t$, then 
	\begin{equation}
		\label{cor:lsearch-1}
		f(x_{t+1}) \!-\! f(x_t) \!\leq\! - \sigma\alpha\sqrt{\epsilon}\|x_{t+1} \!-\! x_t\|\!\qquad\mbox{and}\qquad 
		f(x_{t+1}) \!-\! f(x_t) \leq -\frac{\theta\sigma(1\!-\!\sigma)\alpha^2\epsilon}{2\cL_1\big(B(x_t, \bar{\delta})\big)}.\,\,\,\,
	\end{equation}
\end{corollary}
\proof{Proof.}
First, let $v_t$ and $\delta_t$ be the search direction and the stepsize respectively. Then we obtain
$$f(x_{t+1}) -f(x_t)  =  f(x_t + \delta_t v_t) -f(x_t)\leq   \delta_t \sigma\langle\nabla f(x_t),  v_t\rangle$$
according to the line search termination criterion. By using the fact that $\|x_{t+1}-x_t\| = \delta_t\|v_t\|$, $\|\nabla f(x_t)\|>\sqrt{\epsilon}$ and $v_t$ is chosen to be gradient related satisfying \eqref{eqn:lsearch-grc}, we have 
\begin{equation}
	\label{eqn:cor-lsearch-1}
	f(x_{t+1}) -f(x_t)  \leq -\sigma\alpha\delta_t \|\nabla f(x_t)\|\|v_t\| \leq -\sigma\alpha\sqrt{\epsilon} \|x_{t+1}-x_t\|,
\end{equation}
which proves the first half of \eqref{cor:lsearch-1}. To show the second half, note that the back tracking line search strategy and Lemma \ref{lemma:lsearch} together indicate 
$\delta_t \geq \min\left\{\theta \cdot \frac{(1-\sigma)\alpha\|\nabla f(x_t)\|}{\cL_1(B(x_t, \bar \delta))\|v_t\|}, \theta\cdot\frac{\bar{\delta}}{\|v_t\|}\right\}.$
Then it follows that:
\begin{equation}
	\label{eqn:cor-lsearch-2}
	\|x_{t+1}-x_t\| = \delta_t\|v_t\| \geq \theta\cdot\min\left\{\frac{(1-\sigma)\alpha\sqrt{\epsilon}}{\cL_1(B(x_t,\bar \delta))},\bar\delta\right\} \overset{(i)}{=}  \frac{\theta(1-\sigma)\alpha\sqrt{\epsilon}}{\cL_1(B(x_t,\bar \delta))},
\end{equation}
where (i) is because we assume $\cL_1(\cdot)\geq1$, $0<\theta,\alpha,\sigma\leq1$ and the fact that we have chosen $\bar\delta>\sqrt{\epsilon}$. 
Combining \eqref{eqn:cor-lsearch-1} and \eqref{eqn:cor-lsearch-2} proves the second half of  \eqref{cor:lsearch-1}.$\quad\,\,\qquad\qquad\qquad\qquad\qquad\,\qquad\qquad\,\,\,\,$ 
\endproof 

\noindent Consequently, we have the following complexity result. 
\begin{theorem}
	Suppose Assumption \ref{assumption:LowerBound} holds. Let $\{x_t\}_{t\geq0}$ be generated by the line search scheme $x_{t+1} = \cA_{LS}(f,x_t,\bar \delta)$ with $\bar \delta >\sqrt{\epsilon}$. Let $x_T$ be the first iterate s.t. $\|\nabla f(x_T)\|\leq \sqrt{\epsilon}$. Suppose all search directions  satisfy \eqref{eqn:lsearch-grc} with $\alpha\in(0,1]$. Then  
	$$R_T:=\max\{\|x_t-x_0\|: 0\leq t\leq T\} \leq \frac{\Delta_f}{\alpha\sigma\sqrt{\epsilon}} \qquad\mbox{and}\qquad T\leq\cO\left(\frac{\Delta_f}{\epsilon}\cdot\cL_1\Big(B\big(x_0,R_T+\bar\delta\big)\!\Big)\!\!\right).$$
	Specifically, if we set the search directions to be $-\nabla f(x_t)$, then the result holds with $\alpha = 1$.
\end{theorem}
\proof{Proof.}
Because $x_T$ is the output, we know $\|\nabla f(x_t)\|>\sqrt{\epsilon}$ for any $0\leq t\leq T-1$. Therefore, for any $1\leq t\leq T$, Corollary \ref{corollary:lsearch} indicates 
$$\Delta_f\geq f(x_0) - f(x_t)\geq\sum_{k=1}^t\sigma\alpha\sqrt{\epsilon}\|x_k - x_{k-1}\|\geq\sigma\alpha\sqrt{\epsilon}\|x_t-x_0\|.$$
Therefore, we have $\|x_t-x_0\|\leq \frac{\Delta_f}{\sigma\alpha\sqrt{\epsilon}} = R_T, \;  \forall~t=1,\cdots, {T}$. This implies that for any $t\leq T$, $B(x_t,\bar \delta)\subset B(x_0,R_T+\bar\delta)$. That is, the line search region will be $B(x_0,R_T+\bar\delta)$, we can use $\cL_1(B(x_0,R_T+\bar\delta))$ to upper bound all the $\cL_1(B(x_t,\bar\delta))$. Therefore, Corollary \ref{corollary:lsearch} also indicates that the per-iteration descent is also lower bounded by 
$f(x_t) - f(x_{t+1}) \geq \frac{\theta\sigma(1-\sigma)\alpha^2{\epsilon}}{\cL_1(B(x_0,R_T+\bar\delta))}, \; \forall~t\le T.$
Therefore, the total iteration complexity can be upper bounded by
$$T\leq \cO\bigg(\Delta_f\Big/\frac{\theta\sigma(1-\sigma)\alpha^2{\epsilon}}{\cL_1(B(x_0,R_T+\bar\delta))}\bigg) = \cO\left(\frac{\Delta_f}{\epsilon}\cdot\cL_1\Big(B\big(x_0,R_T+\bar\delta\big)\Big)\right).$$
This proves the result. $\qquad\qquad\qquad\qquad\qquad$ $\qquad\qquad\qquad\qquad\qquad$ $\qquad\qquad\qquad\qquad\qquad$  
\endproof

\section{Accelerated gradient projection algorithm}
\label{appdx:AGD}
Up to now, we have discussed three relatively simple subroutines for SLO that achieve the same complexity of $\cO\Big(\frac{\Delta_f}{\epsilon}\!\cdot\!\cL_1\big(B(x_0,R_T)\big)\!\Big)$ where $R_T = \cO\Big(\frac{\Delta_f}{\sqrt{\epsilon}}\Big)$.  In this section, we develop a new subroutine which adapts the  acceleration technique {\it convex until proven guilty} (CUPG)  developed in \cite{carmon2017convex} recently. CUPG is a first-order algorithm which  achieves  the $\mathcal{O}\Big(\frac{\Delta_f}{\epsilon^{7/8}}\Big)$ complexity, by assuming $f$ to have global Lipschitz gradient and Hessian. Our goal is two-fold: First, we demonstrate that the SLO framework is flexible enough to include complicated algorithms as subroutines; Second, we  show that by further assuming the {\it local} Hessian Lipschitz continuity, it is possible to improve the complexity of SLO to  $\tilde\cO\Big(\frac{\Delta_f}{\epsilon^{7/8}}\cdot\cL_1^\frac{1}{2}\big(B\big(x_0^1,R_T\big)\big)\cdot\cL_2^\frac{1}{4}\big(B(x_0^1,R_T)\big)\!\Big)$. The development and analysis of the new subroutine is not trivial because CUPG is a sophisticated algorithm derived for unconstrained problems, where the iterates cannot be properly bounded. The key to this adaptation is to properly bound the iterates in each epoch so that the local Lipschitz constants can be used.  First of all, we describe the accelerated gradient projection (AGP) subroutine as Algorithm \ref{alg:AGP-subroutine}.

\begin{algorithm2e}
	\DontPrintSemicolon
	\caption{The AGP subroutine: $\cA_{AGP}\left(f, \bar x;L_1, L_2; \bar x_0, D, \epsilon\right)$}
	\label{alg:AGP-subroutine}
	Set $\alpha \!=\! 2\sqrt{L_2}\epsilon^{\frac{1}{4}}$, $X \!=\! \big\{x:\|x \!-\! \bar x_0\|\!\leq\! D\!-\!2\epsilon^{\frac{1}{4}}\big\}$  \,\,\,\,\,\,\,\,\,\,\,\,\,\,\,\,\,\,\,\,/**Algorithm \ref{alg:Meta} requires $\bar x\!\in\!\mathrm{int}(X)$**/\\ 
	Set $\hat{f}(x) := f(x) + \alpha\|x-\bar x\|^2$\\
	$(\text{Flag},p,u,v,\{\hat y_j\}_{j=1}^t) = \texttt{AGP-UPG}\big(\hat f,X, \hat  y_0\!:=\bar x,\hat \epsilon\!:=\frac{\epsilon}{100}, \hat L_1\!:=L_1 + 2\alpha, \alpha\big)$\\
	\textbf{if}\,\, Flag ==1,3,5\,\, \textbf{return}$(p)$\,\, \textbf{end}\\
	\If{$\,\,\mathrm{Flag} == 2,4$}{
		Find $b^{(1)} \!=\! \argmin_b\!\big\{ \!f(b)\!:\! b\!\in\!\{\hat y_1,...,\hat y_{t-1},u\}\!\big\}$ and $b^{(2)} \!=\! \argmin_b\big\{ \!f(b)\!:\! b\!\in\!\big\{ \!u \!\pm\! \frac{\alpha(u-v)}{L_2\|u-v\|}\big\}\!\big\}$ \\
		\textbf{if}\,\, $f(b^{(1)})\leq f(\bar x) - \frac{\alpha^3}{64L_2^2}$\,\, \textbf{then}\,\, \textbf{return}$\left(b^{(1)}\right)$\\
		\textbf{else}\,\, \textbf{return}$\left(b^{(2)}\right)$\,\, \textbf{end} 
	}
\end{algorithm2e} 

The main idea of Algorithm \ref{alg:AGP-subroutine} is that, given the previous iterate $\bar{x}$, a surrogate function $\hat{f}(x)$ is constructed and then optimized by \texttt{AGP-UPG} (\underline{A}ccelerated \underline{G}radient \underline{P}rojection \underline{U}ntil \underline{P}roven \underline{G}uilty), which will perform a series of accelerated gradient projection (AGP) steps, followed by a procedure \texttt{Find-NC-Pair} to ensure sufficient descent. Details of $\texttt{AGP-UPG}$ is given in Algorithm \ref{alg:AGP-UPG}, with sub-functions $\texttt{Certify-Progress}$ and $\texttt{Find-NC-Pair}$ are given in Algorithm \ref{alg:Certify} and \ref{alg:NC-Pair} respectively. 

\begin{algorithm2e} 
	\DontPrintSemicolon 
	\caption{$\texttt{AGP-UPG}\big(\hat f, X , \hat y_0\in\mathrm{int}(X), \hat\epsilon, \hat L_1, \alpha\big)$}
	\label{alg:AGP-UPG} 
	\textbf{Initialize:} $\hat x_0 = \hat y_0$, $\kappa = \hat L_1/\alpha$, $\omega = \frac{\sqrt{\kappa}-1}{\sqrt{\kappa} + 1}$, $t=1$. \\
	\While{$\mathbf{true}$ }{ 
		Set $\tilde y_{t} = \hat x_{t-1} - \hat L_1^{\!-1}\nabla \hat f(\hat x_{t-1})$ \\
		\textbf{if}\,\, $\tilde y_{t}\in\mathrm{int}(X)$\,\, \textbf{then}\,\, set $\hat y_{t} = \tilde y_{t}$\\
		{\nonl\small$\,\,|$}\vspace{-0.02cm}\\
		\textbf{else}\,\, \\
		{\nonl$\,\,\Bigg|$} \vspace{-0.04cm}\\
		{\nonl$\,\,\Bigg|$} \vspace{-0.65cm}\\
		{\nonl$\,\,\Bigg|$} \vspace{-2.8cm}\\ 
		$\quad$ Set $\hat y_{t} = \mathbf{Proj}_{X}\{\tilde y_{t}\}$.\,\,\,\,/**$\hat y_t\in\partial X$, \texttt{AGP-UPG} returns an output and ends.**/\\
		$\quad$ \textbf{if}\,\, $\hat f(\hat y_{t})> \hat f(\hat y_0)$\,\, \textbf{then}\,\, \\
		{\nonl$\quad\,\,\,\,\Bigg|$}\vspace{-1.15cm}\\ 
		$\qquad\,\,\,$ Set $w = \hat y_0$ and 
		$(u,v) = \texttt{Find-NC-Pair}\big(\hat f,\{\hat x_k\}_{k=0}^t, \{\hat y_k\}_{k=0}^t, w\big)$\vspace{0.1cm}\\
		$\qquad\,\,\,$ \textbf{return}$\left(\text{Flag}=2, \textbf{null},u,v, \{\hat y_k\}_{k=0}^t\right)$\vspace{0.1cm}\\
		$\quad$ \textbf{else}\,\, \textbf{return}$(\text{Flag}=1, \hat y_t,\textbf{null},\textbf{null},\textbf{null})$\,\, \textbf{end}\vspace{-0.0cm}\\
		\textbf{end}\vspace{-0.0cm}\\		
		Set $\hat x_t = \hat y_t + \omega (\hat y_t-\hat y_{t-1})$\,\, and\,\,
		$w = \texttt{Certify-Progress}\big(\hat f,X,\hat y_0,\hat y_t,\hat L_1,\alpha,\kappa\big)$
		\vspace{-0.0cm}\\
		\textbf{if}\,\, $w\neq\textbf{null}$ and $w\in\partial X$\,\, \textbf{then}\,\, \textbf{return}$\left(\text{Flag} = 3, w,\textbf{null},\textbf{null},\textbf{null}\right)$\,\, \textbf{end}\vspace{-0.0cm}\\
		\If{\,\,$w\neq \emph{\textbf{null}}$ \emph{and} $w\in\rm{int}(X)$\vspace{-0.0cm}}{    		
			$(u,v) = \texttt{Find-NC-Pair}\big(\hat f,\{\hat x_k\}_{k=0}^t, \{\hat y_k\}_{k=0}^t, w\big)$\vspace{-0.0cm}\\
			\textbf{return}$\left(\text{Flag}=4,\textbf{null}, u,v, \{\hat y_k\}_{k=0}^t\right)$\vspace{-0.0cm}
		}\vspace{-0.01cm}
		\textbf{if}\,\, $\|\nabla \hat f(\hat y_t)\|\leq\sqrt{\hat \epsilon}$\,\, \textbf{then}\,\, \textbf{return}$(\text{Flag}=5, \hat y_t,\textbf{null},\textbf{null},\textbf{null})$\,\, \textbf{end} \\
		Set $t = t+1$ \vspace{-0.1cm}
	} 
\end{algorithm2e}  

\begin{algorithm2e}
	\DontPrintSemicolon  
	\caption{$\texttt{Find-NC-Pair}\big(\hat f,\{\hat x_k\}_{k=0}^t, \{\hat y_k\}_{k=0}^t, w\big)$}
	\label{alg:NC-Pair} 
	\For{$j = 0,1,...,t-1$}{
		\textbf{for}\,\, $u =\hat y_j,w$\,\, \textbf{do} \\
		$\,\,\,\,\Big|\,\,\,\,\,\,$
		\textbf{if}\,\, $\hat f(u)< \hat f(\hat x_j) + \langle\nabla \hat f(\hat x_j),u-\hat x_j\rangle + \frac{\alpha}{2}\|u-\hat x_j\|^2$\,\, \textbf{then}\,\, \textbf{Return}$(u,\hat x_j)$\,\, \textbf{end}\vspace{-0.22cm}\\
		\textbf{end}
		\vspace{-0.4cm}
	}
\end{algorithm2e}
\begin{algorithm2e}
	\DontPrintSemicolon 
	\caption{$\texttt{Certify-Progress}\big(\hat f,X,\hat y_0,\hat y_t,\hat L_1,\alpha,\kappa\big)$}
	\label{alg:Certify}
	\textbf{if}\,\, $\hat f(\hat y_t)>\hat f(\hat y_0)$\,\, \textbf{then}\,\,  \textbf{return}($\hat y_0$)\,\, \textbf{end}\\
	Set $z = \hat y_t - \hat L_1^{\!-1}\nabla \hat f(\hat y_t)$ \\ 
	\textbf{if}\,\, $z\notin\rm{int}(X)$\,\, \textbf{then}\,\,  \textbf{return}$(\mathbf{Proj}_X\{z\})$\,\, \textbf{end}\\
	
	Set $\psi(z) = \hat f(\hat y_0) - \hat f(z) + \frac{\alpha}{2}\|z-\hat y_0\|^2$\\
	\textbf{if}\,\, $\|\nabla \hat f(\hat y_t)\|^2>2\hat L_1\psi(z)\cdot\exp\{-\frac{t}{\sqrt{\kappa}}\}$\,\, \textbf{then}\,\, \textbf{return}($z$)\\
	\textbf{else}\,\, \textbf{return}(\textbf{null})\,\, \textbf{end}
\end{algorithm2e} 

Compared to the fully unconstrained counterparts in CUPG \cite{carmon2017convex}, $\texttt{AGP-UPG}$, $\texttt{Certify-Progress}$ and $\texttt{Find-NC-Pair}$ should properly design a constraint set $X$ to utilize local Lipschitz information. However, the presence of such $X$ significantly complicates the algorithm. For example in Algorithm \ref{alg:AGP-UPG} line 13 and Algorithm \ref{alg:Certify}, line 3, we need to explicitly handle situations when the iterates do not lie in the interior of $X$. The projection step in Algorithm \ref{alg:AGP-UPG} also makes the analysis more involved as well. Finally,  to avoid confusion between the algorithmic parameters in  $\texttt{AGP-UPG}(\cdot)$ and $\cA_{AGP}(\cdot)$, we put a ``$\,\,\hat{\cdot}\,\,$'' notation for the ones used for $\texttt{AGP-UPG}(\cdot)$. We use $\partial X$ to denote the boundary of $X$. 

\subsection{Preliminary results about $\texttt{AGP-UPG}(\cdot)$ function}
To better understand $\cA_{AGP}(\cdot)$, let us first study the $\texttt{AGP-UPG}$ in Algorithm \ref{alg:AGP-UPG}, which performs the main acceleration steps. It takes the surrogate objective function $\hat{f}$, the local constraint set $X$, the accuracy $\hat{\epsilon} = \epsilon/100$, the local Lipschitz constant estimate of $\hat{f}$: $\hat L_1 = L_1 + 2\alpha$, and the penalty parameter $\alpha$ as  input. It outputs a flag that contains different status of the algorithms, as well as a few iterates to be used later. First, let us describe all possible outcomes of $\texttt{AGP-UPG}(\cdot)$ encoded in ``Flag", and then explain how to use the other outputs under different outcomes. 
\begin{remark}
	\label{proposition:Flags}
	The function $\texttt{AGP-UPG}(\cdot)$ has the following $5$ possible outcomes:
	\begin{itemize}
		\item Flag = 1: The last iterate is on the boundary of $X$ while other iterates are in the interior, i.e., $\hat y_t\in \partial X$,
		$\{\hat y_k\}_{k=0}^{t-1}\subset\text{\rm int}(X)$. The output is $(\text{Flag}=1, \hat y_t,\emph{\textbf{null}},\emph{\textbf{null}},\emph{\textbf{null}})$.
		In this case 
		$\hat f(\hat y_t)\leq \hat f(\hat y_0).$
		\item Flag = 2: Similar to  Flag = 1, $\hat y_t\in \partial X$ and $\{\hat y_k\}_{k=0}^{t-1}\subset\text{\rm int}(X)$. However, $\hat f(\hat y_t)>\hat f(\hat y_0)$. In this case, it will be proved that $\texttt{Find-NC-Pair}(\cdot)$ can find and output an NC-pair $(u,v)$ s.t.
		\begin{eqnarray}
			\label{defn:nc-pair}
			\hat f(u)<\hat f(v) + \langle\nabla \hat f(v),u-v\rangle + \frac{\alpha}{2}\|u-v\|^2.
		\end{eqnarray} 
		The output  is $\left(\text{Flag}=2, \emph{\textbf{null}},u,v, \{\hat y_k\}_{k=0}^t\right)$. The vector $u\!\in\!\{\hat y_k\}_{k=0}^{t-1}\!\subset\! X$. However, $v\!\in\!\{\hat x_k\}_{k=0}^{t-1}$ may not belong to $X$. Since $X=B(\bar x_0,D-2\epsilon^{\frac{1}{4}})$ in Algorithm \ref{alg:AGP-subroutine}, for $\forall i$, $\|\hat y_i-\bar x_0\|\leq D-2\epsilon^{\frac{1}{4}},$ and $\|\hat x_i-\bar x_0\| = \|\hat y_i + \omega(\hat y_i - \hat y_{i-1})-\bar x_0\| \leq  3D-6\epsilon^{\frac{1}{4}}$. That is, $\|u-\bar x_0\|\leq D, \|v-\bar x_0\|\leq 3D$. Therefore, $L_1, L_2$ should be chosen as  $\cL_{1}\big(B(\bar x_0,3D)\big)$ and $\cL_{2}\big(B(\bar x_0,3D)\big)$ rather than $\cL_1(X)$ and $\cL_2(X)$.
		\item Flag = 3: In this case, all the iterates are in the interior of $X$, i.e., $\{\hat y_k\}_{k=0}^t \subset\text{int}(X)$; $w\neq \emph{\textbf{null}}$, and  $w\in\partial X$. The output  is $(\text{Flag} = 3, w,\emph{\textbf{null}},\emph{\textbf{null}},\emph{\textbf{null}})$, and $\hat f(w)\leq \hat f(\hat y_0).$
		\item Flag = 4: Similar to Flag = 3, $\{\hat y_k\}_{k=0}^t \subset\text{int}(X)$; In addition, an NC-pair $(u,v)$ satisfying \eqref{defn:nc-pair} is guaranteed to be discovered in this case. The output is $(\text{Flag}=4,\emph{\textbf{null}}, u,v, \{\hat y_k\}_{k=0}^t)$, with $u\in X$. In view of the definition of $X$, we have $\|u-\bar x_0\|\leq D$ and $\|v-\bar x_0\|\leq 3D$. 
		\item Flag = 5: Similar to Flag = 3,  $\{\hat y_k\}_{k=0}^t \subset\text{int}(X)$. In addition, the last iterate satisfies that  $\|\nabla \hat f(\hat y_t)\|\leq\sqrt{\epsilon}$. The output  is $(\text{Flag}=5, \hat y_t,\emph{\textbf{null}},\emph{\textbf{null}},\emph{\textbf{null}})$.
	\end{itemize}
\end{remark}

Next, we show that when ${\rm Flag} = 2,4$ we must be able to find an NC-pair that satisfies \eqref{defn:nc-pair}. The proof is largely similar to its unconstrained counterpart in \cite{carmon2017convex}, while Lemma \ref{lemma:NC-Pair} is for \emph{projected} accelerated gradient descent updates. We relegate its proof to the Appendix.

\begin{lemma}
	\label{lemma:NC-Pair}
	Suppose $\hat L_1$ is chosen so that $\nabla \hat f(\cdot)$ is $\hat L_1$-Lipschitz continuous in $X$, and 
	$$\hat f(\hat y_{s+1})  \leq \hat f(\hat x_s) + \langle \nabla \hat f(\hat x_s), \hat y_{s+1}-\hat x_s\rangle + \frac{\hat L_1}{2}\|\hat y_{s+1}-\hat x_s\|^2$$ holds for $s=0, 1, \cdots, t-1,$ even if $\hat x_s\notin X$. 	Then if ${\rm Flag} = 2,4$,  Algorithm \ref{alg:NC-Pair} will find an NC-pair $(u,v)$ such that \eqref{defn:nc-pair} is satisfied. Moreover, 
	\begin{equation}
		\label{cor:NC-Pair-1}
		\max\{\hat f(\hat y_1),...,\hat f(\hat y_{t-1}),\hat f(u)\} \leq \hat f(\hat y_0).
	\end{equation}
\end{lemma}
In the next lemma, we provide the iteration complexity of the $\texttt{AGP-UPG}\big(\cdot)$ function.
\begin{lemma}
	\label{lemma:MaxIter}
	Let $t$ be the maximum number of iterations of $\texttt{AGP-UPG}\big(\hat f,X,\hat y_0,\hat \epsilon,\hat L_1,\alpha\big)$ before it terminates. Denote $z = \hat y_{t-1} - \hat L_1^{\!-1}\nabla \hat f(\hat y_{t-1}).$ Then $t\leq 1 + \max\Big\{0,\sqrt{\frac{\hat L_1}{\alpha}}\log\left(\frac{2\hat L_1\psi(z)}{\hat \epsilon}\right)\Big\}.$
\end{lemma}
\proof{Proof.} 
Because the first $t\!-\!1$ iterations  do not meet the termination rule for Flag = 1,2,3,4, we have 
$\hat \epsilon \leq \|\nabla \hat f(y_{t-1})\|^2 \leq 2\hat L_1\psi(z)\cdot\exp\Big\{\!\!-\!\frac{t\!-\!1}{\sqrt{\hat L_1/\alpha}}\Big\}$ in Certify-Progress.
This implies the result.
\endproof

\subsection{Iteration complexity of Algorithm \ref{alg:Meta} with $\cA_{AGP}$ subroutine}
Suppose we set the parameters of $\cA_{AGP}$ and the SLO framework as:
\begin{equation}
	\label{eqn:para-AGP-UPG}
	d = 2\epsilon^{\frac{1}{4}},\qquad D \geq 6 \epsilon^{\frac{1}{4}},\qquad L_1^\tau = \cL_1(B(x^\tau_0,3D)),\qquad L_2^\tau = \cL_2(B(x^\tau_0,3D)).
\end{equation}
Note that in the Line 4 of Algorithm \ref{alg:Meta}, we set $L_1^\tau = \cL_1(B(x^\tau_0,D))$. However, due to the special structure of the function $\texttt{AGP-UPG}(\cdot)$, as commented in Remark \ref{proposition:Flags}, we need to set the parameters to be $L_1^\tau = \cL_1(B(x^\tau_0,3D))$ and $L_2^\tau = \cL_2(B(x^\tau_0,3D))$. 
In summary, in the Line 8 of Algorithm \ref{alg:Meta}, the  $k+1$th iteration at epoch $\tau$ is generated as $
x^\tau_{k+1} =  \cA_{AGP}\left(f,x^\tau_k; L_1^\tau, L_2^\tau; x^\tau_0, D, \epsilon\right).$

To analyze the  algorithm, first, let us check the limited travel condition  \eqref{eq:truncated}. For the $\cA_{AGP}(\cdot)$ subroutine, we have $X: = B(x^\tau_0,  D - 2\epsilon^{\frac{1}{4}})$. If the output $x^\tau_{k+1}$ is returned by Line 4 of Algorithm \ref{alg:AGP-subroutine} with Flag = 1, 3, 5, then $x^\tau_{k+1}\!\in\! X$ and $\|x^\tau_{k+1}\!-\!x^\tau_0\|\leq D \!-\! 2\epsilon^{\frac{1}{4}}$, see Remark \ref{proposition:Flags}. Otherwise, if the output $x^\tau_{k+1}$ is returned by Line 5-8 with Flag = 2,4 and  $x^\tau_{k+1}\in\left\{\hat y_1,\cdots,\hat y_{t-1},u, u \pm \frac{\alpha(u-v)}{L_2^\tau\|u-v\|}\right\}$. Remark \ref{proposition:Flags} tells us  $\{\hat y_1,\cdots,\hat y_{t-1},u\}\subset X$ and hence $\max\{\|\hat y_1-x^\tau_0\|,\cdots,\|\hat y_{t-1}-x^\tau_0\|,\|u-x^\tau_0\|\}\leq D - 2\epsilon^{\frac{1}{4}}$. 
Because $\alpha = 2\sqrt{L_2^\tau}\epsilon^{\frac{1}{4}}$ and $L_2^\tau\geq 1$, we also know  
$\max\big\{\big\|u \pm \frac{\alpha(u-v)}{L_2^\tau\|u-v\|}-x^\tau_0\big\|\big\} \leq \|u-x^\tau_0\| + \frac{2\epsilon^\frac{1}{4}}{\sqrt{L_2^\tau}}\leq D.$
Therefore, in all situations, $\|x^\tau_{k+1}-x^\tau_0\|\leq D$. The subroutine $\cA_{AGP}(\cdot)$ satisfies condition  \eqref{eq:truncated}.

\begin{lemma}
	\label{lemma:condition-AGP}
	Suppose Assumptions \ref{assumption:LowerBound} and \ref{assumption:growth_fun_2} hold. Let $\{x^\tau_k\}_{k=0}^{K_\tau}$ be generated by Algorithm \ref{alg:Meta} as the $\tau$-th epoch, with the subroutine  $x^\tau_{k+1} =  \cA_{AGP}\left(f,x^\tau_k; L_1^\tau, L_2^\tau; x^\tau_0, D, \epsilon\right)$ with constants given by  \eqref{eqn:para-AGP-UPG}. If $\|\nabla f(x^\tau_k)\|>\sqrt{\epsilon}$ for $0\leq k\leq K_\tau-1,$ then Condition \ref{condition:sufficient} holds with 
	\begin{align}\label{eq:C1C2:AGP}
		C_1\left(L^{\tau}_1,L_2^\tau,\epsilon\right) = \frac{\sqrt{L_2^\tau}}{24} \epsilon^{\frac{1}{4}} \qquad\mbox{and}\qquad C_2(L^{\tau}_1,L_2^\tau,\epsilon) = \frac{\epsilon^\frac{3}{4}}{10\sqrt{L_2^\tau}}.
	\end{align}
\end{lemma}
The proof of this lemma is in the Appendix. Combined with Lemma \ref{lemma:meta-epoch}, we immediately get a per-epoch descent estimate of
$f(x^\tau_{\!K_\tau}) - f(x^\tau_0) \leq -  \frac{\sqrt{\epsilon} D}{32}$, which further implies Theorem \ref{theorem:meta-complexity-AGP}.

\begin{theorem}
	\label{theorem:meta-complexity-AGP}
	Under the assumptions and algorithmic parameters of Lemma \ref{lemma:condition-AGP}, it takes at most $T = \left\lceil\frac{32(f(x^1_0) - f^*)}{\sqrt{\epsilon} D}+1\right\rceil$ epochs to output a point $\bar x$ s.t. $\|\nabla f(\bar x)\|^2\leq\epsilon$. Each execution of $\cA_{AGP}(\cdot)$ has a complexity of $N_\cA^\tau = \cO\left(\frac{\sqrt{L_1^\tau}}{(L_2^\tau)^\frac{1}{4}}\epsilon^{-\frac{1}{8}}\log\left(1/\epsilon\right)\!\right)$, the total complexity of Algorithm \ref{alg:Meta} is at most 
	$$\sum_{\tau=1}^T K_\tau N_\cA^\tau \leq\cO\left(\frac{\Delta_f}{\epsilon^{7/8}}\cdot\cL_1^\frac{1}{2}\big(B(x^1_0,R_T)\big)\cdot\cL_2^\frac{1}{4}\big(B(x^1_0,R_T)\big)\log(1/\epsilon)\right),$$ 
	where we denote $\Delta_f := f(x^1_0) - f^*$, and $R_T : = (T+2) D =  \frac{32\Delta_f}{\sqrt{\epsilon}}+3D.$
\end{theorem}
\proof{Proof.}
The bound on $T$ is straightforward. For the per-iteration complexity of $\cA_{AGP}$, it suffices to see that the following constants are used in Lemma 4.3
$\hat L_1 = L_1^\tau + 4\sqrt{L_2^\tau}\epsilon^\frac{1}{4}$ and $\alpha = 2\sqrt{L_2^\tau}\epsilon^\frac{1}{4}.$
Thus, the total complexity for invoking the $\cA_{AGP}$ once will be 
$$N_\cA^\tau = \cO\left(\sqrt{\frac{\hat L_1^\tau}{\alpha}}\log(1/\epsilon)\right) = \cO\left(\frac{\sqrt{L_1^\tau}}{(L_2^\tau)^\frac{1}{4}}\epsilon^{-\frac{1}{8}}\cdot\log\left(1/\epsilon\right)\right).$$
Therefore, by applying the above choices of $N_{\cA}^{\tau}$ and $C^{\tau}_2$ in \eqref{eq:C1C2:AGP} to Theorem 2.4, we obtain that 
\begin{eqnarray*}
	\sum_{\tau=1}^T K_\tau N_\cA^\tau  \leq  \max_{1\leq \tau\leq T}\left\{\frac{N_\cA^\tau}{C_2^\tau}\right\}\!\Delta_f \!+\! \sum_{\tau=1}^{T}N_\cA^\tau= \cO\left(\max_{1\leq \tau\leq T}\left\{\!\!\sqrt{L_1^\tau}(L_2^\tau)^\frac{1}{4}\right\}\epsilon^{-\frac{7}{8}}\Delta_f\log(1/\epsilon)\right).
\end{eqnarray*}
Note that $\|x_k^\tau-x_k^0\|\leq D$ and $x_{K_\tau}^\tau = x_0^{\tau+1}$ for any $k,\tau$, by triangle inequality we know 
$\|x_0^\tau-x^1_0\|\leq (\tau-1)D\leq (T-1)D$. Consequently, 
$B(x_0^\tau,3D)\subseteq B(x_0^1,(T+2)D) = B(x_0^1,R_T)$ and 
$$L_1^\tau = \cL_1\big(B(x_0^\tau,3D)\big)\leq \cL_1\big(B(x_0^1,R_T)\big), \quad L_2^\tau = \cL_2\big(B(x_0^\tau,3D)\big)\leq \cL_2\big(B(x_0^1,R_T)\big),$$ 
where $R_T$ is given by $R_T : = (T+2)\cdot D \le  \frac{32\Delta_f}{\sqrt{\epsilon}}+3D.$  	
Therefore, we obtain the final complexity.  
\endproof

\section{Comparison of different algorithms.}
In this section, we compare several algorithms to illustrate the strength of our method. And we provide an example to prove the tightness of our radius estimation $R_T\leq \cO\left(\frac{\Delta_f}{\sqrt{\epsilon}}\right)$.
First, in order to compare withBPG, we provide the following lemma to relate  $D_h(x^{k+1},x^k)$ and $\|\nabla f(x^k)\|^2$. 
\begin{lemma}\label{lemma:BPG}
	Suppose $f(\cdot)$ is $L$-smooth adaptable to function
	$h(x) \!=\! \alpha\|x\|^d \!+\! \beta\|x\|^2$ for some $\alpha,\beta\!>\!0$ and $d\geq4$. Define $D_h(y,x) \!=\! h(y)\!-\!\nabla f(x)^{\!\top}\!(y\!-\!x)\!-\!h(x)$ and let $x^{k+1}\!:=\!\argmin_x \!\nabla f(x^k)^{\!\top}\!(x\!-\!x^k) \!+\! LD_h(x^k)$ be the BPG update.  Suppose $\|x^k\|\geq 1$, then if $\|\delta_k\|\ll1$ is sufficiently small, there $\exists C_1,C_2,>0$ that only depends on $\alpha,\beta,L,d$ such that $\frac{\|\nabla f(x^k)\|^2}{\|x^k\|^{d-2}}\big/D_h(x^{k+1},x^k)\in[C_1,C_2].$
\end{lemma}
Note that the assumption $\|\delta_k\|\ll 1 \leq \|x^k\|$ is very natural, because if one expects BPG to converge, then $\delta_k\to0$ as $k\to+\infty$. A detailed version of Lemma \ref{lemma:BPG} is provided in the appendix together with a proof. Next, let us consider the following series of nonconvex No-Lip-Grad instances. First, define the function $r(z)=\frac{1}{2+\ln(1+z^2)}$. For arbitrary predetermined small enough tolerance $\epsilon>0$, and constant $\Delta\gg\sqrt{\epsilon}$, denote $a_{\Delta,\epsilon} = \frac{-r'(\frac{\Delta}{\epsilon^{200}}-1)}{2}$, $b_{\Delta,\epsilon}=r(\frac{\Delta}{\epsilon^{200}}-1) + \frac{r'(\frac{\Delta}{\epsilon^{200}}-1)}{2}$ where 200 is an arbitrary exponent. Define  $f_{\Delta,\epsilon}(x_1,x_2) = \Delta\cdot g_{\Delta,\epsilon}(x_1) + (x_1^{d-2}+1)x_2^2$ with an even integer $d\geq 4$, and 
$$g_{\Delta,\epsilon}(x_1)=\begin{cases} 
	r(1) + r'(1)(x_1-1) & x_1\leq 1,\\
	r(x_1) & 1<x_1<\frac{\Delta}{\epsilon^{200}}-1,\\
	a_{\Delta,\epsilon}\big(x_1-\frac{\Delta}{\epsilon^{200}}\big)^2 +  b_{\Delta,\epsilon} & x_1\geq \frac{\Delta}{\epsilon^{200}}-1.\\
\end{cases}$$  Then it is not hard to verify that the following properties hold for $f(\cdot):=f_{\Delta,\epsilon}(\cdot)$:  
\begin{itemize}
	\item [\textbf{(i).}] \emph{$f(\cdot)$ is continuously differentiable and \textbf{coercive} with \textbf{bounded level sets}.} 
	\item [\textbf{(ii).}] \emph{The optimal solution is $(x^*_1,x_2^*)=(\Delta\epsilon^{-200},0)$. Default $x^0=0$ as initial solution, then the function gap $\Delta_f:=f(0)-f(x^*) = \Theta(\Delta)$. }
	\item [\textbf{(iii).}] \emph{If $\|\nabla f(\bar x)\|^2\leq \epsilon$, then one must have $\|\bar x\|\geq \Omega\left(\frac{\Delta_f}{\sqrt{\epsilon}\ln^2(\Delta_f/\sqrt{\epsilon})}\right)$. For example, $(x^\epsilon_1,x^\epsilon_2) = (\Delta\epsilon^{-\frac{1}{2}},0)$ is such an $\epsilon$-stationary point. This property also shows that our $R_T\leq\mathcal{O}\left(\frac{\Delta_f}{\sqrt{\epsilon}}\right)$ bound is tight up to an $\mathcal{O}(\ln^2(\Delta_f/\sqrt{\epsilon}))$ factor.}
	\item [\textbf{(iv).}] \emph{Let $\mathrm{Lev}(0) := \{x: f(x)\leq f(0)\}$ be a level set, then $\cL_1(\mathrm{Lev}(0))\geq \Delta^{d-2}\epsilon^{-200(d-2)}$. For any ball $B(0,R)$ with $R\geq \Omega(\Delta)$, we have  $\cL_1(B(0,R))\leq \cO(R^{d-2})$} 
\end{itemize}
\proof{Proof.}
\noindent\textbf{Proof of (i).} Note that $g_{\Delta,\epsilon}(\cdot)$ is a smooth concatenation of three functions. By direct computation, one can check that $g_{\Delta,\epsilon}(x_1)$ is continuously differentiable, non-negative, coercive with bounded level sets w.r.t. the variable $x_1$. Consequently, it is straightforward to show the whole function  $f_{\Delta,\epsilon}(x_1,x_2) = \Delta\cdot g_{\Delta,\epsilon}(x_1) + (x_1^{d-2}+1)x_2^2$ to be continuously differentiable, non-negative, and coercive with bounded level sets, given that $d\geq4$ is an even number. \\
\noindent\textbf{Proof of (ii).} Note that $g_{\Delta,\epsilon}(x_1)$ is monotonically decreasing for $x_1\leq \Delta\epsilon^{-200}$ and is monotonically increasing for $x_1\geq \Delta\epsilon^{-200}$.	Thus $\Delta\cdot g_{\Delta,\epsilon}(x_1)$ is minimized at $x_1^*=\Delta\epsilon^{-200}$. On the other hand, the term $(x_1^{d-2}+1)x_2^2$ is non-negative and is minimized at $x_2^*=0$ regardless of $x_1$. Thus the optimal solution of $f_{\Delta,\epsilon}$ is $(x^*_1,x_2^*)=(\Delta\epsilon^{-200},0)$. Note that $r'(x_1) = -\frac{1}{(2+\ln(1+x_1^2))^2}\cdot\frac{2x_1}{1+x_1^2}$, defaulting $x^0=0$ as the initial solution, then the function value gap satisfies 
\begin{eqnarray*}
	\Delta_f &:=& f_{\Delta,\epsilon}(0)-f_{\Delta,\epsilon}(x^*) = \Delta\left(r(1)-r'(1) - b_{\Delta,\epsilon}\right)\\
	& \,= & \Delta\left(\frac{1}{2+\ln 2}+\frac{1}{(2+\ln 2)^2}-\frac{1}{2+\ln (1+(x_1^*)^2)}+\frac{x_1^*/(1+(x_1^*)^2)}{\big(2+\ln (1+(x_1^*)^2)\big)^2}\right)\\
	& \,= & \Theta(\Delta).
\end{eqnarray*} 
\noindent\textbf{Proof of (iii).} For this property, we need more detailed computation. Suppose $\bar x$ is an $\epsilon$-stationary pint such that  $\|\nabla f_{\Delta,\epsilon}(\bar x)\|^2\leq \epsilon$. We discuss the magnitude of $\bar x$ in three cases. \\
\textbf{Case 1: $\bar x_1\geq \epsilon^{-200}\Delta-1$.} In this case, $\|\bar x\|\geq |\bar x_1|\geq \Omega(\Delta/\sqrt{\epsilon})$ is straightforward. \vspace{0.1cm}\\
\textbf{Case 2: $\bar x_1\leq 1$.} In this situation, the direct computation gives 
$$\nabla f_{\Delta,\epsilon}(\bar x) = \begin{bmatrix}
	-\frac{\Delta}{(2+\ln 2)^2} + (d-2)\bar{x}_1^{d-3}\bar{x}_2^2\\
	2(\bar{x}_1^{d-2}+1)\bar{x}_2
\end{bmatrix}$$
Since $d\geq4$ and $d$ is even, we have $\bar{x}_1^{d-2}+1\geq 1$. As a result,
\begin{equation}
	\label{0-3}
	\epsilon\geq \|\nabla f_{\Delta,\epsilon}(\bar x)\|^2 \geq 4(\bar{x}_1^{d-2}+1)^2\bar{x}_2^2\geq 4\bar{x}_2^2.
\end{equation}  
On the other hand, we also have 
\begin{equation}
	\label{0-3.5}
	\sqrt{\epsilon}\geq \|\nabla f_{\Delta,\epsilon}(\bar x)\| \geq \left|-\frac{\Delta}{(2+\ln 2)^2} + (d-2)\bar{x}_1^{d-3}\bar{x}_2^2\right|.
\end{equation} 
However, the above two inequalities cannot hold simultaneously. If $\bar x_1\leq0$, then $\bar{x}_1^{d-3}\bar{x}_2^2\leq 0$ since $d-3$ is odd. Because we assume $\Delta\gg\sqrt{\epsilon}$, this further results in the contradiction with \eqref{0-3.5} that 
$$\sqrt{\epsilon}\geq \left|-\frac{\Delta}{(2+\ln 2)^2} + (d-2)\bar{x}_1^{d-3}\bar{x}_2^2\right|\geq \frac{\Delta}{(2+\ln 2)^2}$$
If $\bar x_1\geq 0$, together with $\bar x_1\leq 1$ and \eqref{0-3}, we have $(d-2)\bar{x}_1^{d-3}\bar{x}_2^2\leq (d-2)\epsilon/4$. This again raise a contradiction between \eqref{0-3.5} and $\Delta\gg\sqrt{\epsilon}$.
Overall, we exclude the possibility that $\bar x_1\leq 1$. \vspace{0.1cm}\\
\textbf{Case 3: $1< \bar x_1 < \epsilon^{-200}\Delta-1$.} In this situation, the direct computation gives 
$$\nabla f_{\Delta,\epsilon}(\bar x) = \begin{bmatrix}
	\Delta r'(\bar{x}_1)+ (d-2)\bar{x}_1^{d-3}\bar{x}_2^2\\
	2(\bar{x}_1^{d-2}+1)\bar{x}_2
\end{bmatrix},$$
where $r'(\bar{x}_1) = -\frac{1}{(2+\ln(1+\bar{x}_1^2))^2}\cdot\frac{2\bar{x}_1}{1+\bar{x}_1^2}$. Now we assume $\Delta r'(\bar{x}_1) \leq -2\sqrt{\epsilon}$ and prove a contradiction. 
First, $|\nabla_{x_1}f_{\Delta,\epsilon}(\bar x)|\leq \sqrt{\epsilon}$ indicates that $\Delta r'(x_1)+ (d-2)\bar{x}_1^{d-3}\bar{x}_2^2\geq -\sqrt{\epsilon}$. That is, 
$$(d-2)\bar{x}_1^{d-3}\bar{x}_2^2\geq-\sqrt{\epsilon} + \frac{\Delta}{(2+\ln(1+\bar{x}_1^2))^2}\cdot\frac{2\bar{x}_1}{1+\bar{x}_1^2}\geq\frac{\Delta}{(2+\ln(1+\bar{x}_1^2))^2}\cdot\frac{\bar{x}_1}{1+\bar{x}_1^2}.$$
This further indicates that $\bar{x}_2^2\geq\frac{\Delta/(d-2)}{(2+\ln(1+\bar{x}_1^2))^2}\cdot\frac{1}{(1+\bar{x}_1^2)\bar x_1^{d-4}}$. Substitute it to $|\nabla_{x_2} f(\bar x)|^2\leq \epsilon$ yields
\begin{eqnarray*}
	\epsilon\geq 4(1+\bar x_1^{d-2})^2\bar{x}_2^2\geq\frac{\Delta/(d-2)}{(2+\ln(1+\bar{x}_1^2))^2}\cdot\frac{(1+\bar x_1^{d-2})^2}{(1+\bar{x}_1^2)\bar x_1^{d-4}}.
\end{eqnarray*}
Note that the function $\frac{\Delta/(d-2)}{(2+\ln(1+\bar{x}_1^2))^2}\cdot\frac{(1+\bar x_1^{d-2})^2}{(1+\bar{x}_1^2)\bar x_1^{d-4}}$ is monotonically increasing on $[1,+\infty)$. The above inequality further indicates that $\epsilon\geq \frac{2\Delta/(d-2)}{(2+\ln 2)^2}$,
which is impossible since we assume $\Delta\gg \sqrt{\epsilon}$. Thus we conclude $\Delta r'(\bar{x}_1) \geq -2\sqrt{\epsilon}$, namely, 
$$\frac{\Delta}{(2+\ln(1+\bar{x}_1^2))^2}\cdot\frac{2\bar{x}_1}{1+\bar{x}_1^2}\leq 2\sqrt{\epsilon}.$$
This further indicates that $\bar x_1\geq \Omega\left(\frac{\Delta_f}{\sqrt{\epsilon}\ln^2(\Delta_f/\sqrt{\epsilon})}\right)$.\vspace{0.1cm}\\
\textbf{Combine Case 1-3:} Overall, we must have $\bar x_1\geq \Omega\left(\frac{\Delta_f}{\sqrt{\epsilon}\ln^2(\Delta_f/\sqrt{\epsilon})}\right)$ and thus $\|\bar x\|\geq \Omega\left(\frac{\Delta_f}{\sqrt{\epsilon}\ln^2(\Delta_f/\sqrt{\epsilon})}\right)$ holds for all $\epsilon$-stationary point $\bar x$. \vspace{0.1cm} \\
\noindent\textbf{Proof of (iv).} First, notice that $\nabla_{x_2}^2f_{\Delta,\epsilon}(x_1,x_2) = 2(x_1^{d-2}+1)$. Since $x^* = (\Delta\epsilon^{-200},0)\in\text{Lev}(0)$, we have $\cL_1(\text{Lev}(0))\geq2(\Delta\epsilon^{-200})^{d-2}+2\geq \Delta^{d-2}\epsilon^{-200(d-2)}$. Next, to compute the Lipschitz constant in the ball $B(0,R)$, let us compute the Lipschitz constants of the two parts of $f_{\Delta,\epsilon}(\cdot)$ separately. For the term $(x_1^{d-2}+1)x_2^2$, its Lipschitz constant in $B(0,R)$ can by bounded by the Frobenius norm of its Hessian:  
$$\left\|\begin{bmatrix}
	(d-2)(d-3)x_1^{d-3}x_2^2 & 2(d-2)x_1^{d-3}x_2\\2(d-2)x_1^{d-3}x_2 & 2(x_1^{d-2}+1)
\end{bmatrix}\right\|_F \leq \cO(R^{d-2}).$$
For the other term  
$\Delta g_{\Delta,\epsilon}(x_1)$, which is a smooth concatenation of several pieces, it is straightforward to see that $\nabla(\Delta g_{\Delta,\epsilon})(\cdot)$ is $0$-Lipschitz continuous for $x_1\leq1$, and $2\Delta a_{\Delta,\epsilon}$-Lipschitz for $x_1\geq \Delta\epsilon^{-200}-1$. More specifically, we have 
$$2\Delta a_{\Delta,\epsilon} =\frac{\Delta}{(2+\ln(1+x_1^2))^2}\cdot\frac{x_1}{1+x_1^2}\ll1, \quad\text{with}\quad x_1 = \Delta\epsilon^{-200}-1.$$ 
For the last piece with $1<x_1<\Delta\epsilon^{-200}-1$, we have $$|\nabla^2(\Delta g_{\Delta,\epsilon})(x_1)|=\Delta|r''(x_1)|=\Delta\left|\frac{1}{A^2B} - \frac{4x_1^2}{A^3B}-\frac{2x_1^2}{A^2B}\right| \leq \cO(\Delta)$$
where $A = 2+\ln(1+x_1^2)$ and $B=1+x_1^2$. Overall, we know $\nabla (\Delta g_{\Delta,\epsilon})(\cdot)$ is $\cO(\Delta)$-Lipschitz continuous globally. Combined with the discussion of the other polynomial term, we know $\nabla f_{\Delta,\epsilon}(\cdot)$ is $\cO(R^{d-2})$-Lipschitz continuous in the ball $B(0,R)$. Namely, $\cL_1(B(0,R)) = \cO(R^{d-2})$. $\qquad\qquad\qquad\qquad\,\,$ \vspace{0.1cm}
\endproof

\noindent\textbf{Comparison of different methods.}
Now for any predetermined small target accuracy $\epsilon>0$, we choose to optimize the instance $f(x):= f_{\Delta,\epsilon}(x)$ and compare the following approaches. In particular, one can choose $h(x) = \alpha\|x\|^d+\beta\|x\|^2$ for BPG, \cite{li2019provable}.  

\noindent\textbf{Standard GD.} Because $\mathrm{Lev}(0)$ is bounded, one standard approach is to set a constant stepsize $\eta = \frac{1}{\cL_1(\mathrm{Lev}(0))}$ and update $x^{k+1} = x^k - \eta\nabla f(x^k)$. Standard convergence result of GD indicates that  $\cO(\Delta_f \epsilon^{-1} \cL_1(\mathrm{Lev}(0))) = \cO\left(\frac{\Delta^{d-1}_f}{\epsilon^{200(d-2)+1}}\right)$ iterations are needed for finding $\epsilon$-stationary points.

\noindent\textbf{SLO.} SLO automatically bounds $R_T = \cO\left(\frac{\Delta_f}{\sqrt{\epsilon}}\right)$, 
hence $\cL_1(B(0,R_T))=\cO\left(\frac{\Delta^{d-2}_f}{\epsilon^{\frac{d-2}{2}}}\right)$. 
Then SLO with PGD, TGD and Normalized Armijo has a complexity of $\cO\left(\epsilon^{-1}\Delta_f \cL_1(B(0,R_T))\right) = \cO\left(\frac{\Delta^{d-1}_f}{\epsilon^{d/2}}\right)$, while SLO with accelerated subroutine has a complexity of $\cO\left(\frac{\Delta_f^{3(d-1)/4}}{\epsilon^{3d/8}}\right)$ since further estimation gives $\cL_2(B(0,R_T))=\cO\left(\frac{\Delta^{d-3}_f}{\epsilon^{\frac{d-3}{2}}}\right)$, which improves over the un-accelerated subroutines by $\cO\big(\epsilon^{-d/8}\Delta_f^{(d-1)/4}\big)$ at a cost of requiring local Hessian Lipschitz constant estimation.  

\noindent\textbf{BPG.} By Lemma \ref{lemma:BPG}, we have $\|\nabla f(x^k)\|^2 = \Theta\Big(\|x^k\|^{d-2}D_h(x^{k+1},x^k)\Big)$. Because all  $\epsilon$-stationary point $\bar x$ has $\|\bar x\|\geq\tilde{\Omega}\left(\frac{\Delta_f}{\sqrt{\epsilon}}\right)$ thus $D_h(x^{k+1},x^k) = \Theta\left(\frac{\|\nabla f(x^k)\|^2}{\|x^k\|^{d-2}}\right) = \tilde{\Theta}\left(\frac{\epsilon^{d/2}}{\Delta_f^{d-2}}\right)$ if $x^k$ is close to $\bar x$. Then a \emph{necessary} (but not sufficient) condition for BPG to find an $\epsilon$-solution should be  $D_h(x^{k+1},x^k) \leq \tilde{\cO}\left(\frac{\epsilon^{d/2}}{\Delta_f^{d-2}}\right)\ll\epsilon$. \emph{Assume} BPG can bound the norms of all the iterates until finding an $\epsilon$-solution by $R_\epsilon=\cO\left(\frac{\Delta_\epsilon}{\sqrt{\epsilon}}\right)$ (an assumption which has \emph{not} been proved by existing works yet), then the complexity of BPG will be $\cO\left(\frac{L\cdot \Delta_f}{R^{d-2}\epsilon}\right) = \cO\left(\frac{\Delta^{d-1}_f}{\epsilon^{d/2}}\right)$. 

As a remark, the comparison with standard GD shows the advantage of reasonably exploiting the local Lipschitz constant. The comparison with BPG shows that even if the constant $L$ in the BPG complexity is much smaller than the $\cL_1(B(0,R_T))$ in the SLO complexity, the true complexity of BPG may not be better than SLO. This is because the $\cO(1/k)$ convergence of $D_h(x^{k+1},x^k)$ does not always imply the $\cO(1/k)$ convergence of $\|\nabla f(x^k)\|^2$ in No-Lip-Grad optimization.

\section{Numerical experiments}\label{sec:numerical}
In this section, we present some preliminary numerical demonstrations of our SLO framework to the symmetric tensor decomposition problem \eqref{prob:SymTD}, where BPG and standard Armijo linesearch (denoted as std-Armijo) are selected as benchmarks. Additional experiments on linear autoencoder training problem \eqref{prob:autoencoder} and the supervised deep linear neural network training \eqref{prob:supervised-dnn} are deferred to the Appendix. We apply SLO with subroutines $\cA_{GP}(\cdot)$, $\cA_{TG}(\cdot)$ and $\cA_{LS}(\cdot)$, which are denoted as PGD, TGD and Norm-Armijo respectively. We set the activation function as $\sigma(x):=x$ s.t. the training loss is polynomial and an $L$-smad function is available \cite{li2019provable}. 

The implementation details are listed below. For TGD and PGD, the parameter $L_1^\tau$ is estimated by randomly sampling a number of points in the ball $B(x^\tau_0,D)$, and then we set $L_1^\tau$ to be the maximum ratio $\|\nabla f(x) - \nabla f(x')\|/\|x-x'\|$ among the sampled points. The time for estimating these constants are also counted in the running time of TGD and PGD. For both Std-Armijo and Norm-Armijo, the search direction is chosen as $-\nabla f(x_t)$ and we set $\sigma \equiv 0.3$ in line search rule \eqref{eqn:lsearch-rule}. Finally, for BPG, we set the $L$-smad function as $h(x):= \frac{1}{n}\|x\|_2^n + \frac{1}{2}\|x\|_2^2$ according to \cite{li2019provable},
where $n$ is the degree of the polynomial. The implementation of the BPG subproblem is detailed in the Appendix. The running time of solving a BPG subproblem to $\epsilon_0$-accuracy is only $\cO(\log(1/\epsilon_0))$. In all the experiments, we set $\epsilon_0 = 10^{-16}$ as the machine accuracy. 

\begin{figure*}[h]
	\centering 
	\includegraphics[width=0.25\linewidth]{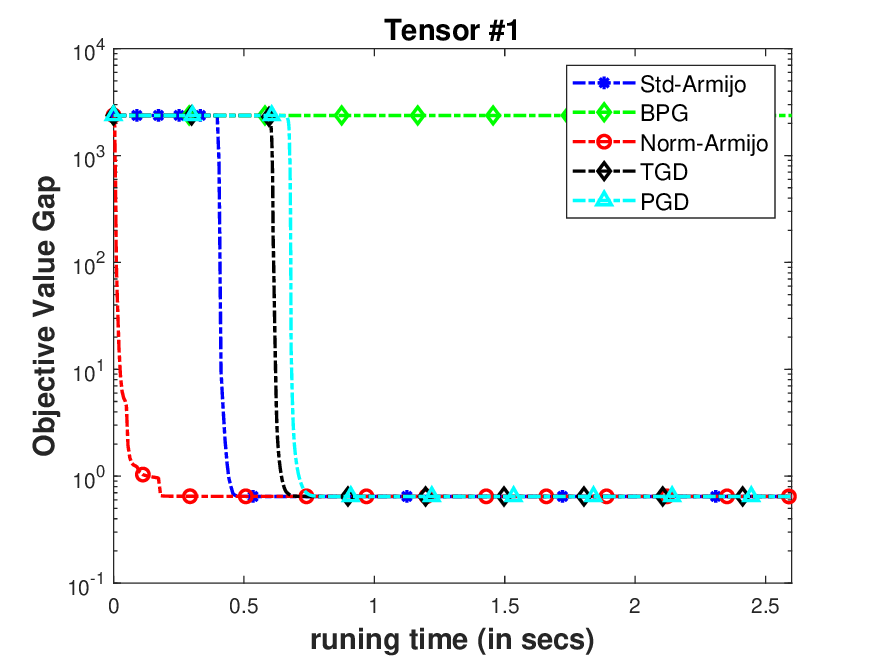}\hspace{-0.18cm}
	\includegraphics[width=0.25\linewidth]{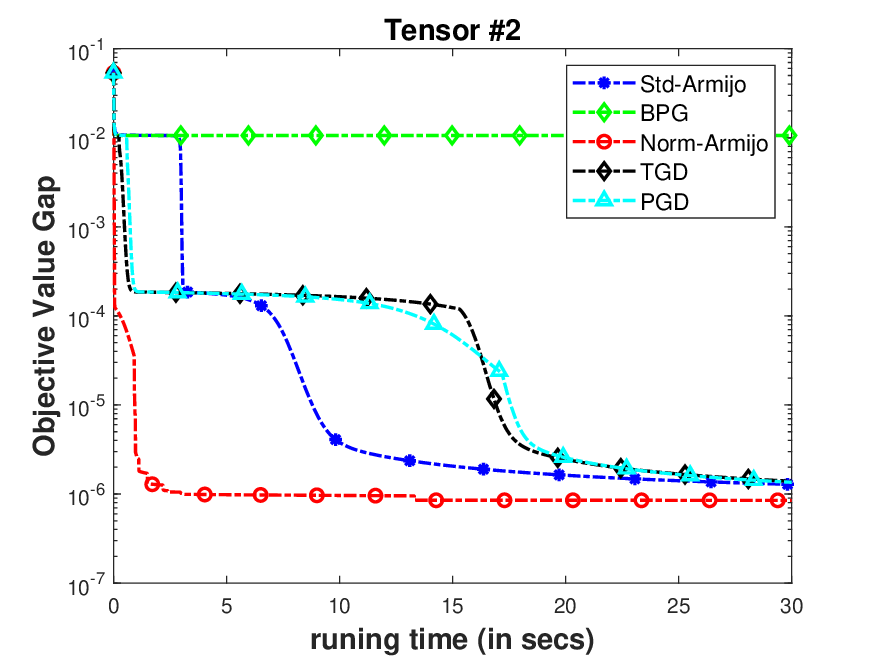}\hspace{-0.18cm}
	\includegraphics[width=0.25\linewidth]{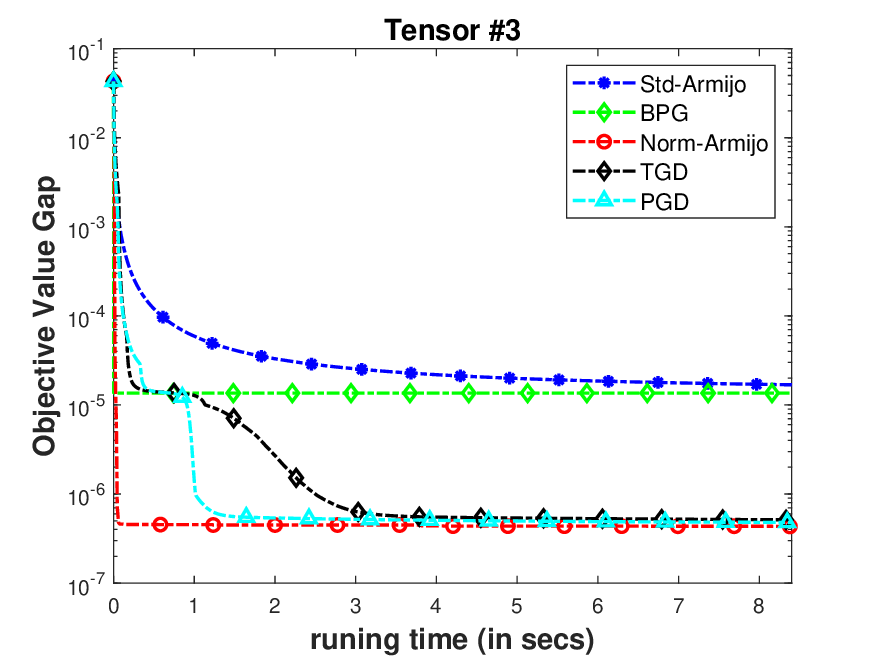}\hspace{-0.18cm}
	\includegraphics[width=0.25\linewidth]{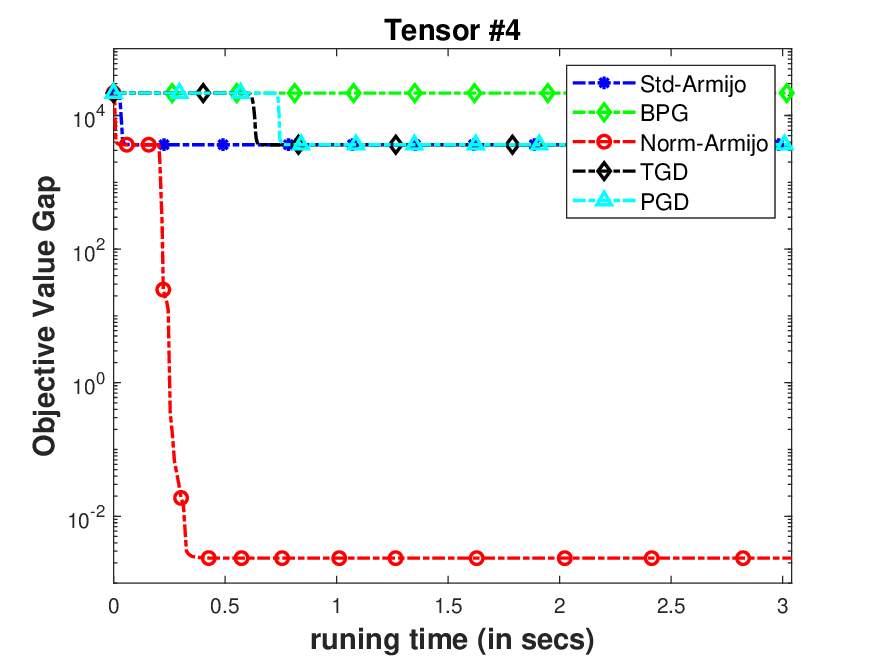}\\[1ex]
	\vfill  
	\includegraphics[width=0.255\linewidth]{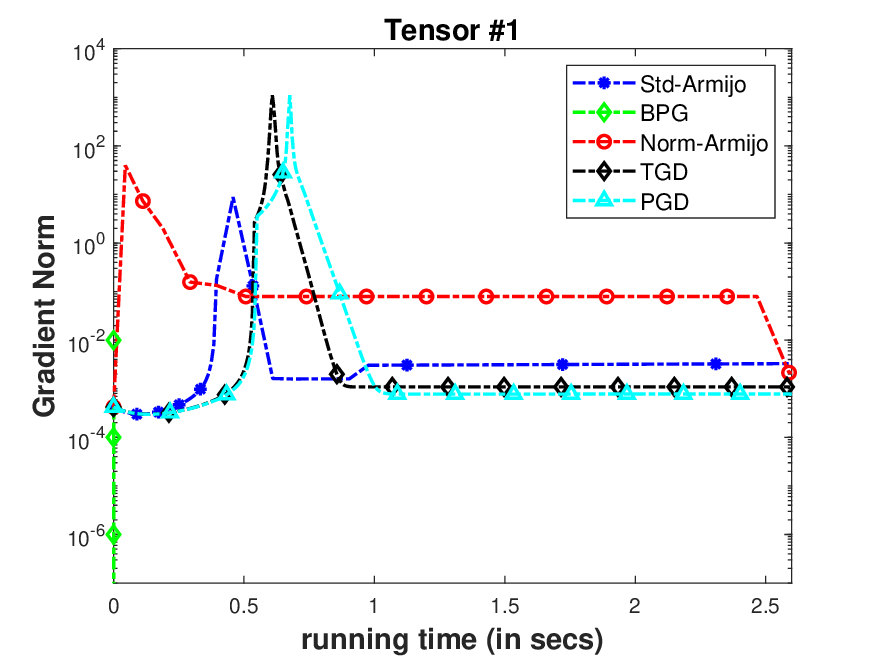} \hspace{-0.39cm}
	\includegraphics[width=0.255\linewidth]{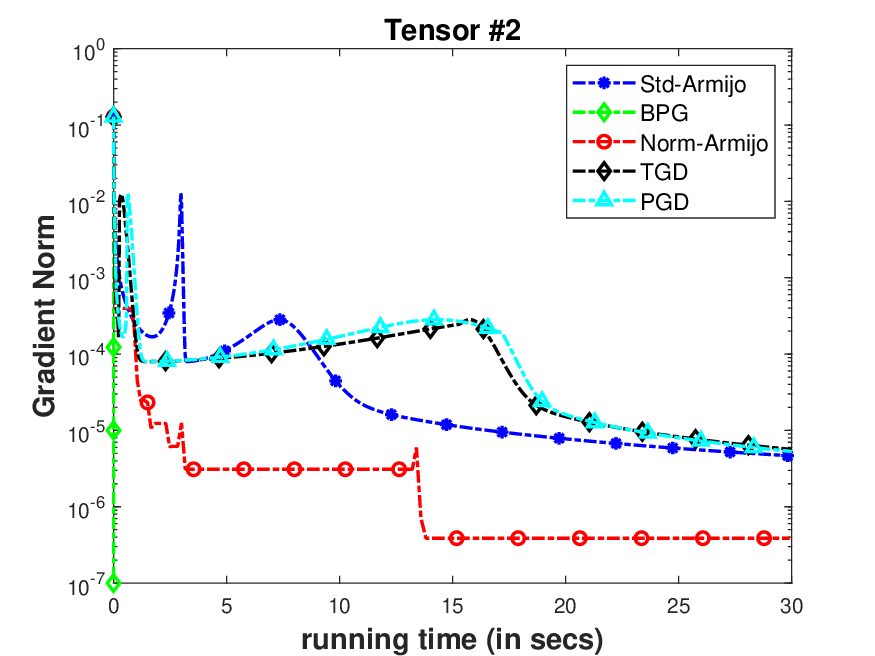}
	\hspace{-0.39cm}
	\includegraphics[width=0.255\linewidth]{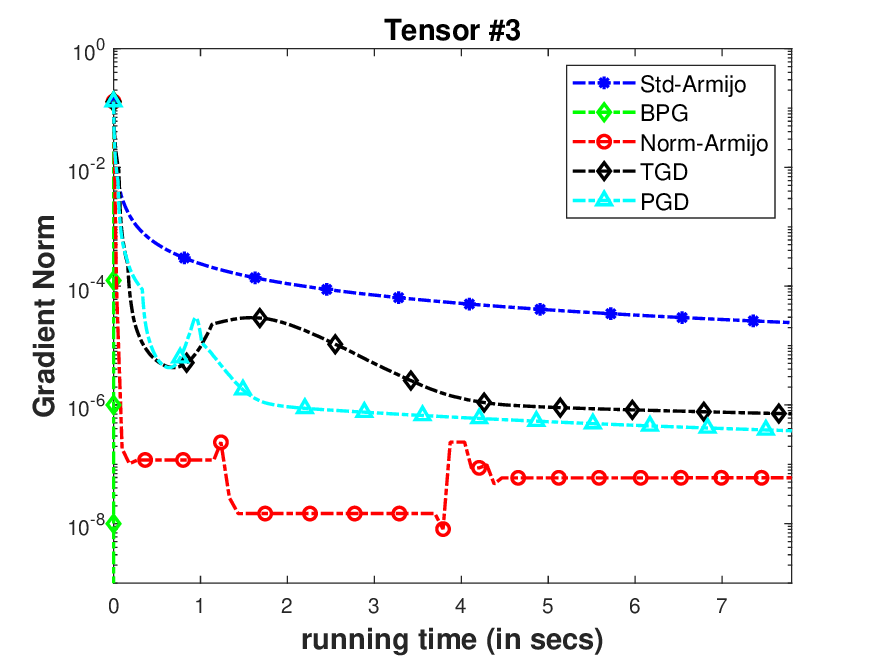}
	\hspace{-0.37cm}
	\includegraphics[width=0.255\linewidth]{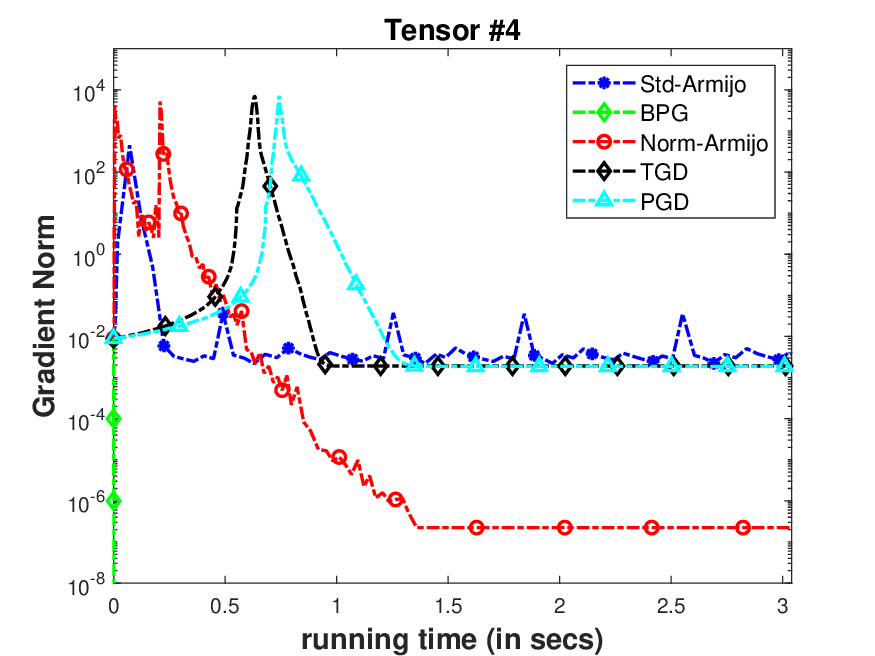}
	\caption{Numerical results on symmetric tensor decomposition, over  representative initial points. }
	\label{fig:SymTD} 
\end{figure*}

For the symmetric tensor decomposition problem \eqref{prob:SymTD}, we let the dimension $k$ to be an odd number so that the variable $\gamma_i$ can be absorbed into the vector $x_i$. Therefore, we solve the problem:
$\mathop{\mathrm{minimize}}_{\bf x} \,\,f({\bf x}) := \big\|\cT - \sum_{i=1}^m \otimes^kx_i\big\|^2,$ where $\otimes^ky$ means the Kronecker product between $k$ $y$'s. The data tensor $\cT$ is a synthesized symmetric tensor. To generate the data $\cT$, we first generate  $\{\hat x_1^*,\hat x_2^*,...,\hat x_m^*\}\subseteq\RR^{d}$ to be a group of $M$ orthonormal bases. Then we multiply each $\hat x^*_i$ by a randomly generated scalar  (which corresponds to $\gamma_i^{1/k}$ in view of \eqref{prob:SymTD}) to create the vectors $\{x_1^*,...,x_m^*\}$. Finally, we set $\cT = \sum_{i=1}^m x_i^*\otimes\cdots\otimes x_i^*$, where ``$\otimes$'' denotes the Kronecker product. In other words, we set $\cT_{j_1,j_2,...,j_k} = \sum_{i=1}^m \Pi_{s=1}^kx_i^*(j_s)$ for any $1\leq j_1,...,j_k\leq d$ where $x_i^*(j_s)$ denotes the $j_s$-th entry of the vector $x_i^*$. In this experiment, we generate a tensor $\cT$ with $d = 8, k = 5, m = 5$. Because $\cT$ has an exact decomposition in the targeted form,  the global optimal value is $f^* = \min_{{\bf x}} f({\bf x}) = 0$. 

\begin{table} 
	\begin{tabular}{ |p{2.5cm}||p{1.3cm}|p{1.3cm}|p{1.45cm}| p{1.4cm}||p{1.3cm}|p{1.3cm}|p{1.5cm}| p{1.45cm}| }
		\hline
		\multicolumn{1}{|c||}{}& \multicolumn{4}{|c||}{Tensor \#1} &  \multicolumn{4}{|c|}{Tensor \#2}\\
		\hline
		\multicolumn{1}{|c||}{}&\multicolumn{2}{|c|}{Gradient norm} &\multicolumn{2}{|c||}{Function value gap} &\multicolumn{2}{|c|}{Gradient norm} &\multicolumn{2}{|c|}{Function value gap}  \\
		\hline 
		& $\,\,\,\,$best & average & $\,\,\,\,$best & average & $\,\,\,\,$best & average & $\,\,\,\,$best & average\\
		\hline
		Std-Armijo         & 5.7e$-$6   & 4.3e$-$5  & 6.5e$-$1 &  5.9e+3  & 3.7e$-$7& 5.7e$-$6&1.1e$-$2&1.1e$-$2\\
		BPG              &   \bf0  &  \bf0 & 2.4e+3 &  2.4e+3 & \bf0 &\bf0 &1.1e$-$2&1.1e$-$2\\
		Norm-Armijo\!\! & 5.6e$-6$ & 7.3e$-$5&  \bf2.2e$-$1 & \bf6.1e$-$1 &3.8e$-\!$8&5.8e$-$7& \bf8.2e$-$7&\bf6.5e$-$4\\
		TGD            & 5.1e$-$6 & 5.9e$-$5   &  6.5e$-$1 & 8.3e+3&2.1e$-$8&9.1e$-$7& 1.9e$-$4&4.3e$-$3\\ 
		PGD            & 5.1e$-$6 & 5.9e$-$5       &  6.5e$-$1 &8.2e+3& 2.7e$-$8&1.3e$-$6&1.9e$-$4&4.3e$-$3\\ 
		\hline
		\hline
		\multicolumn{1}{|c||}{}& \multicolumn{4}{|c||}{Tensor \#3} &  \multicolumn{4}{|c|}{Tensor \#4}\\
		\hline
		\multicolumn{1}{|c||}{}&\multicolumn{2}{|c|}{Gradient norm} &\multicolumn{2}{|c||}{Function value gap} &\multicolumn{2}{|c|}{Gradient norm} &\multicolumn{2}{|c|}{Function value gap}  \\
		\hline
		& $\,\,\,\,$best & average & $\,\,\,\,$best & average & $\,\,\,\,$best & average & $\,\,\,\,$best & average\\
		\hline
		Std-Armijo         & 2.3e$-$\!8  & 1.4e$-$7  & 1.4e$-$5 &  1.4e$-$5  & 5.3e$-$7& 3.1e$-$3 &\bf2.4e$-$3&5.6e+3\\
		BPG              &   \bf0  & \bf 0 & 1.4e$-$5 &  1.4e$-$5 & \bf0 &\bf0 &2.2e+4&2.2e+4\\
		Norm-Armijo & 1.8e$-$10& 1.0e$-$8&  \bf1.2e$-$8 & \bf1.7e$-$7 &2.5e$-$7& 4.3e$-$3& \bf2.4e$-$3&\bf9.1e$+$2\\
		TGD            & 1.7e$-$9 & 3.4e$-$8   &  1.1e$-$7 & 5.6e$-$6 &2.5e$-$6&5.1e$-$3& 3.6e$+$3& 5.8e$+$3\\ 
		PGD            & 1.9e$-$9 & 2.9e$-$8       &  1.2e$-$8 &4.8e$-$6& 2.5e$-$6&5.2e$-$3&3.6e$+$3& 5.8e$+$3\\ \hline
	\end{tabular}
	\caption{Numerical experiment of symmetric tensor decomposition. For each synthetic tensor, we generate the initial solutions elementwisely from a uniform distribution over $[0,0.1]$.  $\qquad\qquad\qquad\,\,$ }
	\label{table:symTD} 
\end{table}

We randomly synthesize 4 tensors. Due to nonconvexity, different initial points may lead to different stationary points, which vary significantly in objective values.  Hence it is not reasonable to plot the averaged curve over multiple random initialization. For each synthetic tensor we plot a representative curve from one single initial solution.  We report the gradient size and objective value gap for all the methods in Figure \ref{fig:SymTD} and Table \ref{table:symTD}, where  Table \ref{table:symTD} contains the best and average result over 20 rounds of all methods. In each round, all methods run for $10$ seconds from the same initialization. As the per-iteration time of TGD, PGD is much smaller than that of BPG and line search methods, we choose the running time instead of iteration number for fair comparison.

In this experiment, we can observe that BPG is the fastest in reducing gradient size. However, BPG also gets trapped at some local solutions, while  Std-Armijo, Norm-Armijo, PGD and TGD find better local solutions. Finally, it is interesting to observe that a sharp increase in gradient size in the above figure corresponds exactly to a sharp decrease in the objective value, which characterizes the process of escaping saddle points or bad local solutions. In addition, it is also worth mentioning that this is only a experiment that aims to validate SLO as a meaningful alternative of BPG method that does not require constructing smooth adaptable functions. Because both the test problem and the experimental settings are of toy sizes, the numerical results cannot be interpreted as showing that SLO with normalized Armijo linesearch has an advantage over BPG in finding stationary points with better objective values in nonconvex optimization.

\section{Conclusion}
In this work, we focus on the nonconvex optimization problems without globally Lipschitz continuous gradient. We propose and analyze a generic framework called the Sequential Local Optimization, which possesses strong theoretical guarantees, while being flexible to incorporate various first-order algorithms as subroutines,  including gradient projection , truncated gradient descent, a novel variant of Armijo linesearch, and a CUPG acceleration. Compared with the existing BPG method, the proposed framework does not require global information of the problem to construct an $L$-smad functions and evaluate Bregman proximal mappings. We believe that the proposed SLO framework can serve as a useful alternative to the BPG-type algorithms when dealing with problems without Lipschitz gradients.

\appendix

\section{Proof of Lemma 4.2}
\label{appdx:lemma:NC-Pair}

\noindent\textbf{Lemma 4.2.}\,\,\emph{Suppose $\hat L_1$ is chosen so that $\nabla \hat f(\cdot)$ is $\hat L_1$-Lipschitz continuous in $X$ and the following inequality holds (even if $\hat x_s\notin X$):}
\begin{align}
	\hat f(\hat y_{s+1}) & \leq \hat f(\hat x_s) + \langle \nabla \hat f(\hat x_s), \hat y_{s+1}-\hat x_s\rangle + \frac{\hat L_1}{2}\|\hat y_{s+1}-\hat x_s\|^2, 	\; s=0, 1, \cdots, t-1.
\end{align}
\emph{When ${\rm Flag} = 2,4$,  then \texttt{Find-NC-Pair}($\cdot$) must be able to find an NC-pair $(u,v)$ such that}
\begin{equation}
	\label{defn:nc-pair}
	\hat f(u)<\hat f(v) + \langle\nabla \hat f(v),u-v\rangle + \frac{\alpha}{2}\|u-v\|^2.
\end{equation}
\emph{Moreover, }
\begin{equation}
	\label{cor:NC-Pair-1}
	\max\{\hat f(\hat y_1),...,\hat f(\hat y_{t-1}),\hat f(u)\} \leq \hat f(\hat y_0).
\end{equation} 

\noindent To prove this lemma, we should first provide the following preliminary result.

\begin{lemma}
	\label{lemma:AGP}
	Let $\{\hat y_0,...,\hat y_t\}$ and $\{\hat x_0,...,\hat x_t\}$ be generated by \texttt{AGP-UPG}. Suppose that for $s = 0,1,...,t-1$, the following conditions hold: 
	\begin{align}
		\hat f(\hat y_s) & \geq \hat f(\hat x_s) + \langle \nabla \hat f(\hat x_s), \hat y_s-\hat x_s\rangle + \frac{\alpha}{2}\|\hat y_s-\hat x_s\|^2, 	\label{lm:AGP-cond-1}  \\
		\hat f(\hat y_{s+1}) & \leq \hat f(\hat x_s) + \langle \nabla \hat f(\hat x_s), \hat y_{s+1}-\hat x_s\rangle + \frac{\hat L_1}{2}\|\hat y_{s+1}-\hat x_s\|^2, 	\label{lm:AGP-cond-1'}  \\
		\hat f(w) & \geq \hat f(\hat x_s) + \langle \nabla \hat f(\hat x_s), w-\hat x_s\rangle + \frac{\alpha}{2}\|w-\hat x_s\|^2. 	\label{lm:AGP-cond-2}
	\end{align}
	Then 
	\begin{equation}
		\label{lm:AGP-main}
		\hat f(\hat y_s) - \hat f(w) \leq \Big(1-\frac{1}{\sqrt{\kappa}}\Big)^s\psi(w)\qquad\mbox{for}\qquad s = 0,1,...,t,
	\end{equation}
	where $\psi(w) = \hat f(\hat y_0) - \hat f(w) + \frac{\alpha}{2}\|w-\hat y_0\|^2.$
\end{lemma}
The form of this lemma is the same as the Proposition 1 of \cite{carmon2017convex}. However, Proposition 1 of \cite{carmon2017convex} is derived for (unconstrained) accelerated gradient descent updates while the Lemma \ref{lemma:AGP} here is derived for the \emph{projected} accelerated gradient descent updates. Therefore, we provide the proof of this lemma in Appendix \ref{appdx:lemma:AGP} for completeness. With this lemma we proceed to the proof of Lemma 4.2.

\begin{proof} 
	{\bf When Flag = 2}: In this case $\hat f(\hat y_t)>\hat f(\hat y_0)$ and $w = \hat y_0$.  If we cannot find an NC-pair, then the conditions \eqref{lm:AGP-cond-1} and \eqref{lm:AGP-cond-2} hold true for Lemma \ref{lemma:AGP}. Since \eqref{lm:AGP-cond-1'} always holds due to our selection of $\hat L_1$, therefore by \eqref{lm:AGP-main}, we will have the following contradiction: 
	$$0<\hat f(\hat y_t) - \hat f(\hat y_0)\leq  \Big(1-\frac{1}{\sqrt{\kappa}}\Big)^t\cdot\psi(\hat y_0) = 0.$$

	{\bf When Flag = 4}: In this case there are two possibilities from the output of \texttt{Certify-Progress}. 
	In the first situation $\hat f(\hat y_t)>\hat f(\hat y_0)$ and $w = \hat y_0$, which is the same as the case where Flag = 2. In the second situation, $w = \hat y_t - \hat L_1^{-1}\cdot\nabla \hat f(\hat y_t)$ and $\|\nabla \hat f(\hat y_t)\|^2>2\hat L_1\psi(w)\cdot\exp\{-\frac{t}{\sqrt{\kappa}}\}$.  If we cannot find an NC-pair, then again the conditions \eqref{lm:AGP-cond-1} and \eqref{lm:AGP-cond-2} hold true. Since \eqref{lm:AGP-cond-1'} always holds, then by Lemma \ref{lemma:AGP}, we have \eqref{lm:AGP-main}, which yields another contradiction that
	$$\exp\Big\{-\frac{t}{\sqrt{\kappa}}\Big\}\cdot \psi(w)<\frac{1}{2\hat L_1}\|\nabla \hat f(\hat y_t)\|^2\leq \hat f(\hat y_t) - \hat f(w)\leq \Big(1-\frac{1}{\sqrt{\kappa}}\Big)^t\cdot\psi(w),$$
	because 
	$\psi(w) = \hat f(\hat y_0) - \hat f(w) + \frac{\alpha}{2}\|w-\hat y_0\|^2\geq \hat f(\hat y_t) - \hat f(w) + \frac{\alpha}{2}\|w-\hat y_0\|^2>0$. Therefore, for both cases with Flag = 2, 4, we can find the NC-pair $(u,v)$.
	
	Consider the proof of the inequality \eqref{cor:NC-Pair-1}. When Flag = 2, we have $\hat f(\hat y_t)>\hat f(\hat y_0)$ and $w = \hat y_0$, hence $u\in\{\hat y_0,\hat y_1,...,\hat y_{t-1}\}$. Because \texttt{AGP-UPG}($\cdot$) does not terminate at $\hat y_1,...,\hat y_{t-1}$, the output of function $\texttt{Certify-Progress}(\cdot)$ is \textbf{null} for $\hat y_1,...,\hat y_{t-1}$. That is, $\hat f(\hat y_j)\leq \hat f(\hat y_0)$ for $j = 1,...,t-1$ and \eqref{cor:NC-Pair-1} is true when Flag = 2.   	When Flag = 4, we have $\hat f(\hat y_t) \leq \hat f(\hat y_0)$ and $w = \hat y_{t} - \frac{1}{\hat L_1}\nabla \hat f(\hat y_t)\in\text{int}(X)$, hence $u\in\{\hat y_1,...,\hat y_{t-1},w\}$. Similar to the argument for Flag = 2, we have $\hat f(\hat y_j)\leq \hat f(\hat y_0)$ for $j = 1,...,t-1$. For $w$, we also have $\hat f(w) \leq \hat f(\hat y_t) - \frac{1}{2\hat L_1}\|\nabla \hat f(\hat y_t)\|^2\leq \hat f(\hat y_t)\leq \hat f(\hat y_0)$. Therefore, \eqref{cor:NC-Pair-1} is also true when Flag = 4. In both cases when Flag = 2,4, the inequality \eqref{cor:NC-Pair-1} holds true. 
\end{proof}

\section{Proof of Lemma \ref{lemma:AGP}}
\label{appdx:lemma:AGP}
At the beginning, we should mention that the form of Lemma \ref{lemma:AGP} is exactly the same as the Proposition 1 of \cite{carmon2017convex}. The main difference is that \cite{carmon2017convex} studies the unconstrained accelerated gradient descent steps, while our SLO scheme is working on a sequence of locally constrained problems. Thus the analysis of our Lemma \ref{lemma:AGP} will have to handle the \emph{projected} gradient descent steps, which is different from the analysis of Proposition 1 of \cite{carmon2017convex}, especially in the construction of the estimation sequence. Before presenting the proof, we first recall the following lemma.  
\begin{lemma}[Theorem 2.2.7, \cite{nesterov2013introductory}]
	\label{lemma:descent-nc}
	Suppose $\hat L_1,\alpha>0$. For a point $\bar x$, define $\bar x_+:= \mathbf{Proj}_X\big\{\bar x - \frac{1}{\hat L_1}\cdot\nabla \hat f(\bar x)\big\}$. Suppose for $\bar x_+$ and some other $x\in X$ the following inequalities are satisfied:
	\begin{equation}
		\label{lm:descent-nc-1}
		\begin{cases}
			\hat f(\bar x_+) \leq \hat f(\bar{x}) + \langle\nabla \hat f(\bar{x}), \bar x_+-\bar x\rangle + \frac{\hat L_1}{2}\|\bar x_+-\bar x\|^2,\\
			\hat f(x) \geq \hat f(\bar{x}) + \langle\nabla \hat f(\bar{x}), x-\bar x\rangle + \,\frac{\alpha}{2}\,\|x-\bar x\|^2.
		\end{cases}
	\end{equation}
	Denote $\cG_{\hat L_1}(\bar x) = \hat L_1(\bar x - \bar x_+)$. Then we have 
	\begin{equation}
		\label{lm:descent-nc-2}
		\hat f(x)\geq \hat f(\bar x_+) + \langle\cG_{\hat L_1}(\bar x),x-\bar x\rangle + \frac{1}{2\hat L_1}\|\cG_{\hat L_1}(\bar x)\|^2 + \frac{\alpha}{2}\|x-\bar x\|^2.
	\end{equation}
\end{lemma} 

\begin{proof}
	Now we start to prove  Lemma \ref{lemma:AGP}. Let us first define an estimate sequence. Define
	\begin{equation}
		\label{lm:AGP-1}
		\phi_0(z): = \hat f(\hat y_0) + \frac{\alpha}{2}\|z-\hat y_0\|^2.
	\end{equation}
	Let $\kappa  = \hat L_1/\alpha$. For $s = 1,...,t$, we recursively define
	\begin{eqnarray}
		\label{lm:AGP-2}
		\phi_{s}(z) &:=& \frac{1}{\sqrt{\kappa}}\cdot\Big(\hat f(\hat y_s) + \frac{1}{2\hat L_1}\|\cG_{\hat L_1}(\hat x_{s-1})\|^2 + \langle\cG_{\hat L_1}(\hat x_{s-1}),z-\hat x_{s-1}\rangle + \frac{\alpha}{2}\|z-\hat x_{s-1}\|^2\Big)\nonumber\\
		& & + \Big(1-\frac{1}{\sqrt{\kappa}}\Big)\cdot\phi_{s-1}(z),
	\end{eqnarray}
	where in \texttt{AGP-UPG}($\cdot$), 
	we have $\hat y_s = \mathbf{Proj}_X\{\hat x_{s-1} - \hat L_1^{-1}\nabla \hat f(\hat x_{s-1})\}$, and the gradient mapping satisfies $\cG_{\hat L_1}(\hat x_{s-1}) = \hat L_1(\hat x_{s-1} - \hat y_s)$. Let us assume that 
	\begin{equation}
		\label{lm:AGP-3}
		\hat f(\hat y_s)\leq \min_z \phi_s(z), \qquad\mbox{for}\qquad s = 0,1,...,t,
	\end{equation} 
	\begin{equation}
		\label{lm:AGP-4}
		\phi_s(w)\leq \hat f(w) + \Big(1-\frac{1}{\sqrt{\kappa}}\Big)^s\psi(w), \qquad\mbox{for}\qquad s = 0,1,...,t.
	\end{equation}
	The proof of inequalities \eqref{lm:AGP-3} and \eqref{lm:AGP-4} are provided in Appendices \ref{appdx:AGP-3} and \ref{appdx:AGP-4} respectively. Combining \eqref{lm:AGP-3} and \eqref{lm:AGP-4} yields
	$$\hat f(\hat y_s)\leq\phi_s(w)\leq \hat f(w) + \Big(1-\frac{1}{\sqrt{\kappa}}\Big)^s\psi(w), \qquad\mbox{for}\qquad s = 0,1,...,t.$$
	Rearranging the terms proves this lemma. 
\end{proof}

\subsection{Proof of inequality \eqref{lm:AGP-3}}
\label{appdx:AGP-3}
\begin{proof}
	We prove this inequality by induction. First, for $s=0$, 
	$$\hat f(\hat y_0) = \min_z \hat f(\hat y_0) + \frac{\alpha}{2}\|z-\hat y_0\|^2 = \min_z\phi_0(z).$$
	Then \eqref{lm:AGP-3} holds for $s=0$. Suppose \eqref{lm:AGP-3} holds for $s=k$ and $k\leq t-1$, now we prove \eqref{lm:AGP-3} for $s = k+1$. Notice that $\{\phi_s\}_{s=0}^t$ is a sequence of quadratic functions, denote $v_k = \argmin_z \phi_k(z)$ and $\phi_k^* = \phi_k(v_k)$, then we have 
	$$\phi_{k}(z) = \phi_k^* + \frac{\alpha}{2}\|z-v_k\|^2.$$
	Due to the recursive definition of $\phi_k$, we have the following relationships:
	\begin{equation}
		\label{lm:AGP-3-0}
		v_{k+1} = \Big(1-\frac{1}{\sqrt{\kappa}}\Big)\cdot v_k + \frac{1}{\sqrt{\kappa}}\cdot \hat x_k - \sqrt{\kappa}\cdot(\hat x_{k} - \hat y_{k+1})
	\end{equation}
	and 
	\begin{eqnarray}
		\label{lm:AGP-3-1}
		\phi_{k+1}^* & = & \frac{1}{\sqrt{\kappa}}\Big(1-\frac{1}{\sqrt{\kappa}}\Big)\cdot\Big(\frac{\alpha}{2}\|\hat x_k-v_k\|^2 + \langle \cG_{\hat L_1}(\hat x_k),v_k-\hat x_k\rangle\Big)-\frac{1}{2\hat L_1}\Big(1-\frac{1}{\sqrt{\kappa}} \Big)\cdot\|\cG_{\hat L_1}(\hat x_{k})\|^2\nonumber\\
		& & + \Big(1-\frac{1}{\sqrt{\kappa}}\Big)\phi^*_k + \frac{1}{\sqrt{\kappa}}\hat f(\hat y_{k+1}).
	\end{eqnarray}
	Due to the condition \eqref{lm:AGP-cond-1} and Lemma \ref{lemma:descent-nc}, we have 
	$$\phi_k^*\geq \hat f(\hat y_k)\geq \hat f(\hat y_{k+1}) + \langle\cG_{\hat L_1}(\hat x_k),\hat y_{k}-\hat x_k\rangle + \frac{1}{2\hat L_1}\|\cG_{\hat L_1}(\hat x_k)\|^2 + \frac{\alpha}{2}\|\hat y_{k+1}-\hat x_k\|^2.$$
	Substitute this inequality into \eqref{lm:AGP-3-1} yields
	\begin{eqnarray}
		\label{lm:AGP-3-2}
		\phi_{k+1}^* & = & \frac{1}{\sqrt{\kappa}}\Big(1-\frac{1}{\sqrt{\kappa}}\Big)\cdot\Big(\frac{\alpha}{2}\|\hat x_k-v_k\|^2 + \langle \cG_{\hat L_1}(\hat x_k),v_k-\hat x_k\rangle\Big)-\frac{1}{2\hat L_1}\Big(1-\frac{1}{\sqrt{\kappa}} \Big)\cdot\|\cG_{\hat L_1}(\hat x_{k})\|^2\nonumber\\
		& & + \Big(1-\frac{1}{\sqrt{\kappa}}\Big)\cdot\Big(\hat f(\hat y_{k+1}) + \langle\cG_{\hat L_1}(\hat x_k),\hat y_{k}-\hat x_k\rangle + \frac{1}{2\hat L_1}\|\cG_{\hat L_1}(\hat x_k)\|^2 + \frac{\alpha}{2}\|\hat y_{k+1}-\hat x_k\|^2\Big) \nonumber\\
		& & + \frac{1}{\sqrt{\kappa}}\hat f(\hat y_{k+1})\nonumber\\
		& = & \hat f(\hat y_{k+1}) + \frac{1}{\sqrt{\kappa}}\Big(1-\frac{1}{\sqrt{\kappa}}\Big)\cdot\Big(\frac{\alpha}{2}\|\hat x_k-v_k\|^2 + \langle \cG_{\hat L_1}(\hat x_k),v_k-\hat x_k\rangle\Big)\\
		& &  + \Big(1-\frac{1}{\sqrt{\kappa}}\Big)\cdot\Big( \langle\cG_{\hat L_1}(\hat x_k),\hat y_{k}-\hat x_k\rangle +  \frac{\alpha}{2}\|\hat y_{k+1}-\hat x_k\|^2\Big)\nonumber\\
		& \geq &  \hat f(\hat y_{k+1}) + \frac{1}{\sqrt{\kappa}}\Big(1-\frac{1}{\sqrt{\kappa}}\Big)\cdot\big\langle \cG_{\hat L_1}(\hat x_k), \sqrt{\kappa}\hat y_{k} +  v_k-(1+\sqrt{\kappa})\hat x_k\big\rangle.\nonumber
	\end{eqnarray}
	Note that by \eqref{lm:AGP-3-0}, we have 
	\begin{eqnarray}
		\label{lm:AGP-3-3}
		& & \sqrt{\kappa}\hat y_{k} +  v_k-(1+\sqrt{\kappa})\hat x_k\\
		& = & \sqrt{\kappa}\hat y_{k} +  v_k-(1+\sqrt{\kappa})\Big(\hat y_{k} + \frac{\sqrt{\kappa}-1}{\sqrt{\kappa}+1}(\hat y_{k} - \hat y_{k-1})\Big)\nonumber\\
		& = & (\sqrt{\kappa}-1)\hat y_{k-1} + v_k - \sqrt{\kappa}\hat y_{k}\nonumber\\
		& = & (\sqrt{\kappa}-1)\hat y_{k-1} + \left((1-\frac{1}{\sqrt{\kappa}})v_{k-1} + \frac{1}{\sqrt{\kappa}}\hat x_{k-1} - \sqrt{\kappa}(\hat x_{k-1}-\hat y_k)\right) - \sqrt{\kappa}\hat y_{k}\nonumber\\
		& = & (\sqrt{\kappa}-1)\hat y_{k-1} + \Big(1-\frac{1}{\sqrt{\kappa}}\Big)v_{k-1} - \Big(\sqrt{\kappa}-\frac{1}{\sqrt{\kappa}}\Big)\hat x_{k-1}\nonumber\\
		& = & \Big(1-\frac{1}{\sqrt{\kappa}}\Big)\cdot\left(\sqrt{\kappa}\hat y_{k-1} +  v_{k-1}-(1+\sqrt{\kappa})\hat x_{k-1}\right)\nonumber\\
		& = & \Big(1-\frac{1}{\sqrt{\kappa}}\Big)^k\cdot\left(\sqrt{\kappa}\hat y_{0} +  v_{0}-(1+\sqrt{\kappa})\hat x_{0}\right)\nonumber\\
		& = &0\nonumber
	\end{eqnarray}
	where the last inequality is because $\hat y_0 = v_0 = \hat x_0$. Substitute \eqref{lm:AGP-3-3} into \eqref{lm:AGP-2} yields 
	$\phi_{k+1}^*\geq \hat f(\hat y_{k+1})$, which proves \eqref{lm:AGP-3}. 
\end{proof}

\subsection{Proof of inequality \eqref{lm:AGP-4}}
\label{appdx:AGP-4}
\begin{proof}
	Let us prove \eqref{lm:AGP-4} by induction. When $s=0$, we have 
	\begin{eqnarray}
		\phi_0(w) & = & \hat f(\hat y_0) + \frac{\alpha}{2}\|w-\hat y_0\|^2\nonumber\\
		& = & \hat f(w) + \Big(1-\frac{1}{\sqrt{\kappa}}\Big)^0\cdot\Big(\hat f(\hat y_0) - \hat f(w) + \frac{\alpha}{2}\|w-\hat y_0\|^2\Big)\nonumber\\
		& = & \hat f(w) + \Big(1-\frac{1}{\sqrt{\kappa}}\Big)^0\cdot\psi(w).\nonumber
	\end{eqnarray}
	That is, \eqref{lm:AGP-4} is true for $s=0$. Now, suppose \eqref{lm:AGP-4} holds for $s=k$, $k\leq t-1$, then let us prove that \eqref{lm:AGP-4} holds for $s = k+1$. By \eqref{lm:AGP-cond-1} and \eqref{lm:AGP-cond-2}, Lemma \ref{lemma:descent-nc} indicates that 
	$$\hat f(w)\geq \hat f(\hat y_{k+1}) + \frac{1}{2\hat L_1}\|\cG_{\hat L_1}(\hat x_k)\|^2 + \langle\cG_{\hat L_1}(\hat x_k),w-\hat x_k\rangle + \frac{\alpha}{2}\|w-\hat x_k\|^2.$$
	Substitute this inequality into the definition of $\phi_{k+1}$ yields
	\begin{eqnarray}
		\phi_{k+1}(w)& = & \frac{1}{2\hat L_1}\|\cG_{\hat L_1}(\hat x_k)\|^2 + \langle\cG_{\hat L_1}(\hat x_k),w-\hat x_k\rangle + \frac{\alpha}{2}\|w-\hat x_k\|^2\nonumber\\
		& & + \Big(1-\frac{1}{\sqrt{\kappa}}\Big)\phi_s(w) + \frac{1}{\sqrt{\kappa}}\hat f(\hat y_{k+1})\nonumber\\
		& \leq & \Big(1-\frac{1}{\sqrt{\kappa}}\Big)\cdot\left(\hat f(w) + \big(1-\frac{1}{\sqrt{\kappa}}\big)^k\psi(w)\right) + \frac{1}{\sqrt{\kappa}}\cdot \hat f(w)\nonumber\\
		& = & \hat f(w) + \Big(1-\frac{1}{\sqrt{\kappa}}\Big)^{k+1}\psi(w).\nonumber
	\end{eqnarray}
	Hence we complete the proof. 
\end{proof}

\section{Proof of Lemma 4.4}
\label{appdx:condition-AGP}
\noindent\textbf{Lemma 4.4} \emph{Suppose Assumptions 2.1 and 2.2 hold. Let $\{x^\tau_0, x^\tau_1,...,x^\tau_{K_\tau}\}$ be generated by Algorithm 1 during the $\tau$-th epoch, with $D, d, L_1^\tau, L_2^\tau$ defined according to}
\begin{equation}
	\label{eqn:para-AGP-UPG}
	d = 2\epsilon^{\frac{1}{4}},\qquad D \geq 6 \epsilon^{\frac{1}{4}},\qquad L_1^\tau = \cL_1(B(x^\tau_0,3D)),\qquad L_2^\tau = \cL_2(B(x^\tau_0,3D)). 
\end{equation}
\emph{and the subroutine $\cA$ is chosen as with the subroutine  $x^\tau_{k+1} =  \cA_{AGP}\left(f,x^\tau_k; L_1^\tau, L_2^\tau; x^\tau_0, D, \epsilon\right)$.  In addition, if $\|\nabla f(x^\tau_k)\|>\sqrt{\epsilon}$ for $0\leq k\leq K_\tau-1,$ then Condition 2.1 holds with}
\begin{align}\label{eq:C1C2:AGP}
	C_1\left(L^{\tau}_1,L_2^\tau,\epsilon\right) = \frac{\sqrt{L_2^\tau}}{24}\cdot \epsilon^{\frac{1}{4}} \qquad\mbox{and}\qquad C_2(L^{\tau}_1,L_2^\tau,\epsilon) = \frac{\epsilon^\frac{3}{4}}{10\sqrt{L_2^\tau}}.
\end{align}

\begin{proof} Consider the $k$-th iteration of the $\tau$-th epoch:
	$x^\tau_{k}  = \cA_{AGP}\left(f,x^\tau_{k-1}; L_1^\tau, L_2^\tau; \tau D; \epsilon\right).$ 	We prove this lemma in the following two parts. For all the discussion on the Flag values, please refer to Remark 4.1.\\
	\textbf{Part I. } Proving the first, inequality of Condition 2.1 for $k = 1,2,...,K_\tau$. Note that the property of the output $x^\tau_k$ is dependent on the Flag values. We prove the result in the following 4 cases. \vspace{0.2cm}\\ 
	\textbf{Case 1. } Flag = 5.	In this case, we have $x^\tau_{k} = \hat y_t$, and 
	the output of $\texttt{Certify-Progress}(\cdot)$ is \textbf{null} for the point $x^\tau_k$, meaning that $\hat f(x^\tau_{k}) \leq \hat f(x^\tau_{k-1})$. Consequently, 
	\begin{eqnarray}
		\label{lm:acc-descent-1-1}
		f(x^\tau_k)-f(x^\tau_{k-1}) & = & \hat f(x^\tau_k) - \hat f(x^\tau_{k-1}) - \alpha\|x^\tau_k - x^\tau_{k-1}\|^2\nonumber\\
		& \leq &- \alpha\|x^\tau_k - x^\tau_{k-1}\|^2 \\
		& \leq & -2\sqrt{L_2^\tau}\epsilon^{\frac{1}{4}}\cdot\|x^\tau_k - x^\tau_{k-1}\|^2\nonumber,
	\end{eqnarray} 
	where in the first equality we used the definition of $\hat{f}$, and in the last inequality we use the fact that $\alpha = 2\sqrt{L_2^\tau}\cdot\epsilon^{\frac{1}{4}}$, defined in  Algorithm 3.\vspace{0.2cm}\\
	\textbf{Case 2.} Flag = 2,4, and $x^\tau_k = b^{(1)}$. In this case, $(u,v)\neq (\textbf{null},\textbf{null})$. Note that $b^{(1)}\in\{y_1,...,y_{t-1},u\}$, then Lemma 4.2 indicates that $\hat f(x^\tau_{k}) = \hat f(b^{(1)})\leq \hat f(x^\tau_{k-1})$, i.e., \eqref{lm:acc-descent-1-1} still holds.\vspace{0.2cm}\\
	\textbf{Case 3.} Flag = 2,4, and $x^\tau_k = b^{(2)}$. To proceed to the next step, we will need the Lemma 3 and Lemma 1 of \cite{carmon2017convex}. For completeness, we present these lemmas in the Appendix \ref{appdx:lms from proven guilty}. Note that, by Algorithm 3, we know $b^{(2)}$ is selected if $$f(b^{(1)})- f(x^\tau_{k-1})>-\frac{\alpha^3}{64(L_2^\tau)^2}.$$ 
	By Lemma 3 of \cite{carmon2017convex} (with $\hat y_0 = x^\tau_{k-1}$ in this case), the above inequality indicates that  $\|u-v\| \leq \frac{\alpha}{2L_2^\tau}$  and $\|u-x^\tau_{k-1}\|\leq\frac{\alpha}{8L_2^\tau}$.  
	
	On the other hand, the NC-pair $(u,v)\neq (\textbf{null},\textbf{null})$ in this case and \eqref{defn:nc-pair} holds with $\alpha = 2\sqrt{L_2^\tau}\epsilon^{\frac{1}{4}}$. Therefore, 
	$$f(v) + \langle\nabla f(v),u-v\rangle -\frac{\alpha}{2}\|u-v\|^2 - f(u) \overset{(i)}{=} \hat f(v) + \langle\nabla \hat f(v),u-v\rangle + \frac{\alpha}{2}\|u-v\|^2 - \hat f(u)\overset{(ii)}{>}0,$$
	where (i) is due to the definition of $\hat f$ and (ii) is due to NC-pair condition \eqref{defn:nc-pair}. Consequently, the following holds
	$$f(u)< f(v) + \langle\nabla f(v),u-v\rangle -\frac{\alpha}{2}\|u-v\|^2.$$
	Together with the fact that $\|u-v\|\leq \frac{\alpha}{2L_2^\tau}$, the above inequality indicates that 
	$$ f(b^{(2)}) \leq f(u) - \frac{\alpha^3}{12(L_2^\tau)^2}$$
	due to the Lemma 1 of \cite{carmon2017convex}. Then Lemma 4.2 indicates that 
	\begin{eqnarray}
		\label{lm:acc-descent-1-4}
		f(x^\tau_k) &\!=\!& f(b^{(2)}) \!\leq\! f(u) - \frac{\alpha^3}{12(L_2^\tau)^2}\!\nonumber\\
		&\overset{(i)}{\leq}&\! \hat f(u)- \frac{\alpha^3}{12(L_2^\tau)^2}\!\overset{(ii)}{\leq}\! \hat f(x^\tau_{k-1})- \frac{\alpha^3}{12(L_2^\tau)^2}\\
		& \!\overset{(iii)}{=}\! & f(x^\tau_{k-1})- \frac{\alpha^3}{12(L_2^\tau)^2}.\nonumber
	\end{eqnarray}
	Where (i) is because $f(u)\leq f(u) + \alpha\|u-x^\tau_{k-1}\|^2 = \hat f(u)$, (ii) is because $\hat f(u)\leq \hat f(\hat y_0) = \hat f(x^\tau_{k-1})$ due to Lemma 4.2, and (iii) is because $\hat f(x^\tau_{k-1}) = f(x^\tau_{k-1}) + \alpha\|x^\tau_{k-1}-x^\tau_{k-1}\|^2 = f(x^\tau_{k-1}).$ Also note that 
	$$\|x^\tau_{k} - x^\tau_{k-1}\|  = \|b^{(2)} - x^\tau_{k-1}\| \leq \|b^{(2)} - u\| + \|u - x^\tau_{k-1}\|\leq \frac{\alpha}{L_2^\tau} + \frac{\alpha}{8L_2^\tau}\leq \frac{2\alpha}{L_2^\tau}.$$
	Combining the above inequality with \eqref{lm:acc-descent-1-4} yields
	\begin{equation}
		\label{lm:acc-descent-1-5}
		f(x^\tau_{k}) - f(x^\tau_{k-1}) \leq  - \frac{\alpha^3}{12(L_2^\tau)^2}\leq - \frac{\alpha}{48}\|x^\tau_{k} - x^\tau_{k-1}\|^2 = -\frac{\sqrt{L_2^\tau}}{24}\epsilon^{\frac{1}{4}}\cdot\|x^\tau_{k} - x^\tau_{k-1}\|^2.
	\end{equation}
	\textbf{Case 4.} Flag = 1,3, the output $x^\tau_k$ is on the boundary of $X$, i.e., $\|x^\tau_k-x^\tau_0\| =  D-2\epsilon^\frac{1}{4} = D-d$. Therefore, this case only happens at the end of the epoch, i.e., $k = K_\tau$. Due to the discussion of Remark 3.1, $x^\tau_k$ satisfies $\hat f(x^\tau_k) \leq \hat f(x^\tau_{k-1})$. Then  \eqref{lm:acc-descent-1-1} is still true. 
	\vspace{0.1cm}\\
	Combining \textbf{Case 1-4} proves first inequality of Condition 2.1 with $C_1(L_1^\tau,L_2^\tau,\epsilon) = \frac{\sqrt{L_2^\tau}}{24}\cdot \epsilon^\frac{1}{4}$.\vspace{0.3cm}\\
	\textbf{Part II. } Proving the second inequality of Condition 2.1 for $k = 1,2,...,K_\tau-1$. Note that $Flag = 1,3$ only happens at the last iterate $x^\tau_{K_\tau}$, we only need to consider  three cases. \vspace{0.2cm}\\ 
	\textbf{Case 1. } Flag = 5.	In this case, we have $x^{\tau}_k = \hat y_t$, and $\|\nabla \hat f(x^\tau_{k})\| \leq \sqrt{\hat \epsilon} =  \sqrt{\frac{\epsilon}{100}}$. Note that $\|\nabla f(x^\tau_{k})\|>\sqrt{\epsilon}$ and $\nabla \hat f(x^\tau_{k}) = \nabla f(x^\tau_{k}) + 2\alpha(x^\tau_k - x^\tau_{k-1})$. Consequently, we have 
	$$2\alpha\|x^\tau_k - x^\tau_{k-1}\|\geq\|\nabla f(x^\tau_{k})\| - \|\nabla \hat f(x^\tau_{k})\| \geq \frac{9\sqrt{\epsilon}}{10}.$$
	Combined with \eqref{lm:acc-descent-1-1},  we have 
	\begin{eqnarray*}
		\label{lm:acc-descent-1-2}
		f(x^\tau_k)-f(x^\tau_{k-1}) \leq -\alpha\|x^\tau_k - x^\tau_{k-1}\|^2\leq -\alpha\left(\frac{9\sqrt{\epsilon}}{20\alpha}\right)^2 \leq -\frac{\epsilon}{5\alpha} = -\frac{\epsilon^\frac{3}{4}}{10\sqrt{L_2^\tau}}.
	\end{eqnarray*}
	\\
	\textbf{Case 2.} Flag = 2,4, and $x^\tau_k = b^{(1)}$. Then   Algorithm 3  indicates that 
	\begin{eqnarray*} 
		f(x^\tau_k)-f(x^\tau_{k-1})  = f(b^{(1)})-f(x^\tau_{k-1})\leq -\frac{\alpha^3}{64(L_2^\tau)^2} = -\frac{\epsilon^\frac{3}{4}}{8\sqrt{L_2^\tau}}.
	\end{eqnarray*}
	\textbf{Case 3.} Flag = 2,4, and $x^\tau_k = b^{(2)}$. Then \eqref{lm:acc-descent-1-4} indicates that 
	\begin{eqnarray*} 
		f(x^\tau_k)-f(x^\tau_{k\-1})  = f(b^{(2)})-f(x^\tau_{k-1})\leq -\frac{\alpha^3}{12(L_2^\tau)^2} = -\frac{2\epsilon^\frac{3}{4}}{3\sqrt{L_2^\tau}}.
	\end{eqnarray*}
	Combining \textbf{Case 1-3} proves second inequality of Condition 2.1 with $C_2(L_1^\tau,L_2^\tau,\epsilon) = \frac{\epsilon^\frac{3}{4}}{10\sqrt{L_2^\tau}}$.
\end{proof} 
\subsection{Supporting lemmas from \cite{carmon2017convex}}
\label{appdx:lms from proven guilty}
\begin{lemma}[Lemma 1 of \cite{carmon2017convex}]
	If the function $f(\cdot)$ has $L_2$-Lipschitz Hessian (in $B(u,\eta)$). Let $\alpha>0$ and let $u,v$ satisfy 
	$$f(u) < f(v) + \langle\nabla f(v),u-v\rangle - \frac{\alpha}{2}\|u-v\|^2. $$ 
	If $\|u-v\|\leq \frac{\alpha}{2L_2}$ and $\eta\leq \frac{\alpha}{L_2}$, then 
	$$\min\left\{f\left(u - \eta\cdot\frac{u-v}{\|u-v\|}\right), f\left(u + \eta\cdot\frac{u-v}{\|u-v\|}\right)\right\}\leq f(u) - \frac{\alpha\eta^2}{12}.$$
\end{lemma}
As a remark, in our case, we only need $f$ to have locally $L_2$-Lipschitz Hessian within the ball $B(u,\eta)$, which contains the line segment $\left[u - \eta\cdot\frac{u-v}{\|u-v\|}, u + \eta\cdot\frac{u-v}{\|u-v\|}\right]$. In our case, the line segment is contained in $B(x^\tau_0,3D)$.

\begin{lemma}[Lemma 3 of \cite{carmon2017convex}]
	Consider the $\cA_{AGP}$ subroutine, let $f$ have $L_1$-Lipschitz gradient (in $B(\bar x_0,3D)$). If $(u,v)\neq(\textbf{null},\textbf{null})$, and $f(b^{(1)})\geq f(\hat y_0)-\alpha\nu^2$ for some $\nu\geq0$, where $\hat y_0 = \bar x$ is the input to $\cA_{AGP}$. Then for $\forall i$,
	$$\|\hat y_i-\hat y_0\|\leq \nu,\qquad\|u-\hat y_0\|\leq \nu,\qquad\|\hat x_i - \hat y_0\|\leq 3\nu,\qquad\|u-v\|\leq 4\nu.$$ 
\end{lemma}
One remark is that $\|u-\hat y_0\|\leq \nu$ is an intermediate result in the proof of Lemma 3 of \cite{carmon2017convex}, we present this explicitly in the lemma for our need. Another remark is that although our $\cA_{AGP}$ differs from its counterpart in \cite{carmon2017convex} by using the projection steps and a few other adaptations, these changes are irrelevant to the proof of the lemma. Therefore, the proof of Lemma 3 \cite{carmon2017convex} is still valid for our $\cA_{AGP}$.

\section{Proof of Lemma 5.1}
First, we provide a detailed version of Lemma 5.1 as follows.\vspace{0.2cm}\\ 
\textbf{Lemma 5.1} \emph{Suppose $f(\cdot)$ is $L$-smooth adaptable to 
	$h(x) \!=\! \alpha\|x\|^d \!+\! \beta\|x\|^2$ for some $\alpha,\beta\!>\!0$ and $d\geq4$. Define $D_h(y,x) \!=\! h(y)-\nabla f(x)^{\!\top}\!(y-x)-h(x)$ and let $x^{k+1}\!:=\!\argmin_x \!\nabla f(x^k)^{\!\top}\!(x\!-\!x^k) \!+\! LD_h(x^k)$ be a BPG update.  Denote $\delta_k\!:=\!x^{k+1}\!-\!x^{k}$ and assume w.l.o.g. that $\|x^k\|\geq 1$, then for sufficiently small $\|\delta_k\|\ll 1$, we have
	\begin{equation}
		\label{0-1}
		\delta_k +\cO\left(\|\delta_k\|^2\right) =  \left(\frac{\alpha d(d-2)C_0^{-1}}{1+C_0^{-1}\alpha d(d-2)}\!\cdot\! \frac{x^k(x^k)^{\!\top}}{\|x^k\|^2}  - I\right)\!\!\cdot\!  \frac{\nabla f(x^k)}{LC_0\|x^k\|^{d-2}}
	\end{equation}
	where $C_0 = \alpha d + \frac{2\beta}{\|x^k\|^{d-2}}\in [\alpha d,\alpha d+2\beta] $. As a result, when $\|\delta_k\|$ is small enough s.t. $\|\delta_k\|\leq\frac{\|x^k\|}{d-2}$ and the norm of the $\mathcal{O}(\|\delta_k\|^2)$ residual vector is less than $\frac{\|\delta_k\|}{2}$, we have 
	\begin{equation}
		\label{0-1.5}
		\frac{L\alpha d}{4}\leq\frac{\|\nabla f(x^k)\|}{\|x^k\|^{d-2}}\big/\|\delta_k\|\leq\frac{3L(d-1)(\alpha d+2\beta)}{2}
	\end{equation}
	and 
	\begin{equation}
		\label{0-2}
		\frac{L^2\alpha^2d^2}{16(2\beta+3\alpha d(d-1))}\leq\frac{\|\nabla f(x^k)\|^2}{\|x^k\|^{d-2}}\big/D_h(x^{k+1},x^k)\leq\frac{36L^2(d-1)^2(\alpha d+2\beta)^2}{8\beta+\alpha d}.
\end{equation} }
We remark that the assumption $\|\delta_k\|\ll 1 \leq \|x^k\|$ is quite natural. In fact, if one expects BPG to converge, then $\delta_k\to0$ as $k\to+\infty$. 
\begin{proof}
	First, by direct computation, we have 
	$$\nabla h(x) = \alpha d\|x\|^{d-2}\cdot x + 2\beta x,$$
	$$\nabla^2 h(x) = \alpha d(d-2)\|x\|^{d-4}\cdot xx^\top + (2\beta +\alpha d\|x\|^{d-2})\cdot I.$$ 
	Note that BPG subproblem can be rewritten as 
	$$\delta_k = \argmin_\delta\,\, L^{-1}\nabla f(x^k)^\top\delta + h(x^k+\delta) - \nabla h(x^k)^\top \delta - h(x^k).$$
	The KKT condition of this problem gives 
	$$L^{-1}\nabla f(x^k) + \alpha d\|x^k+\delta_k\|^{d-2}\cdot (x^k+\delta_k) + 2\beta (x^k+\delta_k) - \alpha d\|x^k\|^{d-2}\cdot x^k - 2\beta x^k=0.$$
	Dividing the equation by $\|x^k\|^{d-2}$ gives 
	\begin{equation}
		\label{1}
		\alpha d\left\|\frac{x^k}{\|x^k\|}+\frac{\delta_k}{\|x^k\|}\right\|^{d-2}\cdot (x^k+\delta_k) - \alpha d\cdot x^k + \frac{2\beta \delta_k}{\|x^k\|^{d-2}}=  \frac{-\nabla f(x^k)}{L\cdot\|x^k\|^{d-2}}. 
	\end{equation}
	Consider the function $\psi(x):=\|x\|^{d-2}$, then we have $\nabla \psi(x) = (d-2)\|x\|^{d-4}\cdot x$. Define $C_2=\sup_{x:\|x\|\leq 2} \|\nabla^2\psi(x)\|$. Then we have 
	$$E_k:=\psi(y) - \nabla\psi(x)^\top(y-x)-\psi(x),\quad |E_k|\leq \frac{C_2}{2}\|y-x\|^2, \quad \forall \|x\|\leq 2, \|y\|\leq2. $$  
	Substitute $x = \frac{x^k}{\|x^k\|}$ and $y=\frac{x^k}{\|x^k\|}+\frac{\delta_k}{\|x^k\|}$ into the above inequality yields 
	\begin{eqnarray}
		\left\|\frac{x^k}{\|x^k\|}+\frac{\delta_k}{\|x^k\|}\right\|^{d-2}\cdot (x^k+\delta_k) & = & \left(\left\|\frac{x^k}{\|x^k\|}\right\|^{d-2} + (d-2)\left\|\frac{x^k}{\|x^k\|}\right\|^{d-4} \cdot \frac{\delta_k^\top x^k}{\|x^k\|^2}+E_k\right)\cdot (x^k+\delta_k)\nonumber\\
		& = & x^k+\delta_k + (d-2)\frac{x^k(x^k)^\top}{\|x^k\|^2}\delta_k + \cO(\|\delta_k\|^2).\nonumber
	\end{eqnarray} 
	Substituting the above equation to \eqref{1} yields
	$$\left(\left(\alpha d + \frac{2\beta}{\|x^k\|^{d-2}}\right)\cdot I + \frac{\alpha d(d-2)}{\|x^k\|^2}\cdot x^k(x^k)^\top\right)\delta_k + \cO(\|\delta_k\|^2) = \frac{-\nabla f(x^k)}{L\|x^k\|^{d-2}}$$
	Then applying Sherman-Morrison formula to invert the matrix provides \eqref{0-1}. Note that 
	$$\frac{1}{d-1}\leq \left\|\left(\frac{\alpha d(d-2)C_0^{-1}}{1+C_0^{-1}\alpha d(d-2)}\!\cdot\! \frac{x^k(x^k)^{\!\top}}{\|x^k\|^2}  - I\right)\right\|\leq 2,$$
	when $\|\delta_k\|$ is small enough so that the norm of the $\mathcal{O}(\|\delta_k\|^2)$ residual vector is less than $\frac{\|\delta_k\|}{2}$, \eqref{0-1} indicates that
	$$\frac{\|\delta_k\|}{2}\leq \frac{2\|\nabla f(x^k)\|}{LC_0\|x^k\|^{d-2}}\qquad\mbox{and}\qquad\frac{3\|\delta_k\|}{2}\geq \frac{\|\nabla f(x^k)\|}{(d-1)LC_0\|x^k\|^{d-2}}.$$
	Then further applying the fact that $C_0=\alpha d+\frac{2\beta}{\|x^k\|^{d-2}}\in[\alpha d, \alpha d+2\beta]$ when $\|x^k\|\geq 1$ proves \eqref{0-1.5}.
	Next, we analyze the Bregman distance that measures the convergence of BPG method. By Taylor's expansion, for some $\theta\in[0,1]$ we have 
	\begin{eqnarray}
		D_h(x^{k+1},x^k) & = & h(x^{k+1})-h(x^k)-\nabla h(x^k)^\top(x^{k+1}-x^k) \nonumber\\
		& = & \frac{1}{2}\delta_k^\top\nabla^2 h(x^k + \theta\delta_k)\delta_k\nonumber.
	\end{eqnarray} 
	Denote $\tilde{x} = x^k+\theta\delta_k$ for simplicity, then we have $$\nabla^2 h(\tilde{x}) = (2\beta + \alpha d\|\tilde{x}\|^{d-2}) I + \alpha d(d-2)\|\tilde{x}\|^{d-4}\cdot\tilde{x}\tilde{x}^\top.$$
	As we assume $\delta_k$ is small enough s.t. $\|\delta_k\| \leq \frac{\|x_k\|}{d-2}$ and $\theta\in[0,1]$, we have 
	$$\|\nabla^2h(\tilde{x})\|\leq 2\beta + \left(\alpha d+\alpha d(d-2)\right)\Big(1+\frac{1}{d-2}\Big)^{d-2}\|x^k\|^{d-2} \leq 2\beta + 3\alpha d(d-1)\|x^k\|^{d-2} $$
	and, given $d\geq4$, we have
	$$\nabla^2h(\tilde{x})\succeq \left(2\beta + \alpha d\left(1-\frac{1}{d-2}\right)^{d-2}\|x^k\|^{d-2}\right) I \succeq \left(2\beta + \frac{\alpha d}{4}\|x^k\|^{d-2}\right) I.$$
	Combining the bounds on $\nabla^2 h(\tilde{x})$ and \eqref{0-1.5}, we have 
	$$\frac{L^2\alpha^2d^2}{16(2\beta+3\alpha d(d-1))}\leq\frac{\|\nabla f(x^k)\|^2}{\|x^k\|^{d-2}}\big/D_h(x^{k+1},x^k)\leq\frac{36L^2(d-1)^2(\alpha d+2\beta)^2}{8\beta+\alpha d},$$
	which completes the proof. 
\end{proof}

\section{Implementation of the BPG subproblem.}
\label{appdx:BPG}
For an objective function $f$,  the BPG method solves the following subproblem in each iteration:
\begin{eqnarray*}
	x_{t+1} & = &  \argmin_{x} \langle \nabla f(x_t), x-x_t\rangle + L\cdot (h(x) - \langle \nabla h(x_t),x-x_t\rangle - h(x_t))\\
	& = & \argmin_{x} \langle \nabla f(x_t) + L\cdot \nabla h(x_t), x\rangle + \frac{L}{n}\cdot \|x\|^n + \frac{L}{2}\cdot \|x\|^2.
\end{eqnarray*}
Let $v_t = \nabla f(x_t) + L\cdot \nabla h(x_t)$. Then the optimal solution must be 
$$x_{t+1}=-\rho_t\cdot v_t$$ for some $\rho_t>0$ where $\rho_t$ is the positive root of the following equation
\begin{equation*}
	\rho_t = \argmin_{\rho>0} p(\rho):=-\|v_t\|^2\cdot\rho + \frac{L}{n}\cdot \|v_t\|^n\cdot\rho^n + \frac{L}{2}\cdot \|v_t\|^2\cdot\rho^2.
\end{equation*}
Equivalently, we can solve the nonlinear equation:
\begin{equation}
	\label{prob:BPG-sub}
	p'(\rho):=-\|v_t\|^2 + L\cdot \|v_t\|^n_2\cdot\rho^{n-1} + L\cdot \|v_t\|^2\cdot\rho = 0, \qquad \rho>0.
\end{equation}
It is straightforward that $p'(0) = -\|v_t\|^2\leq0$. Then we apply the following simple binary search to solve the above equation. So the running time of solving a single BPG subproblem is $\cO(\log(1/\epsilon_0))$. 

\begin{algorithm2e}[H]
	\DontPrintSemicolon  
	\caption{Solving \eqref{prob:BPG-sub}}
	\label{alg:BPG-sub}
	\textbf{Input:} An accuracy $\epsilon_0$.\\
	\textbf{Initialize:} Constants $\rho_{0} = 0$, $\rho_{1} = 1$.\\
	\While{$p'(\rho_1)<0$}{
		$\rho_0 = \rho_1, \; \rho_1 = 2\rho_1$.
	} 
	\While{$\rho_1-\rho_0>\epsilon_0$}{
		\eIf{$p'(\frac{\rho_0+\rho_1}{2})>0$}{
			$\rho_1 = \frac{\rho_0+\rho_1}{2}$.}{
			$\rho_0 = \frac{\rho_0+\rho_1}{2}$.}
	}
	\textbf{Output:} $\frac{\rho_0+\rho_1}{2}$.
\end{algorithm2e}

\section{Additional experiments}
In this appendix, we present some additional numerical experiments on the unsupervised linear linear autoencoder training problem and the supervised linear neural network training problem. 
\subsection{Unsupervised linear autoencoder training}
Consider the linear autoencoder training problem 
$$\mathop{\mathrm{minimize}}_{W_1,...,W_m}\,\, \Big\|X - \sigma\big(W_m\cdots\sigma\big(W_2\,\sigma\big(W_1 X\big)\big)\cdots\big)\Big\|^2,$$
where the activation function $\sigma(\cdot)$ is the identity mapping. We randomly pick $6,000$ data points from the ijcnn1 dataset\footnote{\url{https://www.csie.ntu.edu.tw/~cjlin/libsvmtools/datasets/binary.html}}. Each data point has $22$ features and the whole data matrix has size $X\in\RR^{22\times6000}$. We set the number of layers to be $m = 5$, with weights:
$W_5\in\RR^{22\times10}, W_4\in\RR^{10\times4}, W_3\in\RR^{4\times10}, W_2\in\RR^{10\times20},  W_1\in\RR^{20\times22}.$

\begin{figure*}[h]
	\centering 
	\includegraphics[width=0.25\linewidth]{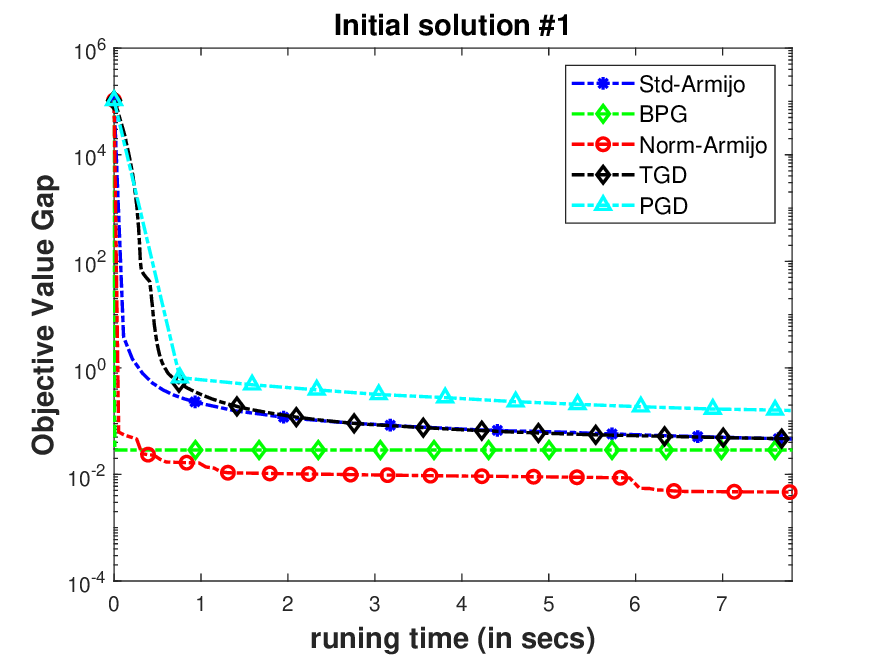}\hspace{-0.18cm}
	\includegraphics[width=0.25\linewidth]{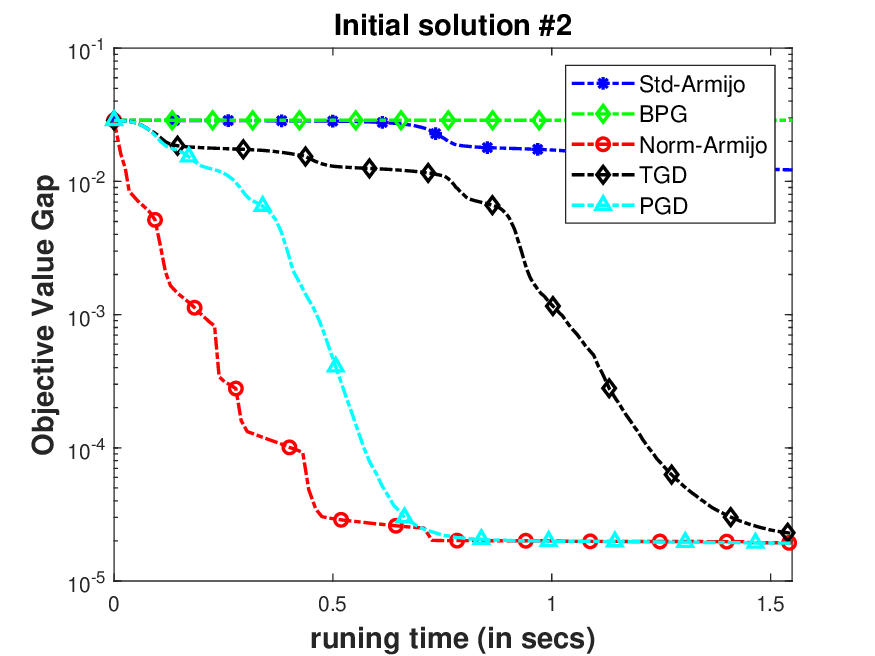}\hspace{-0.18cm}
	\includegraphics[width=0.25\linewidth]{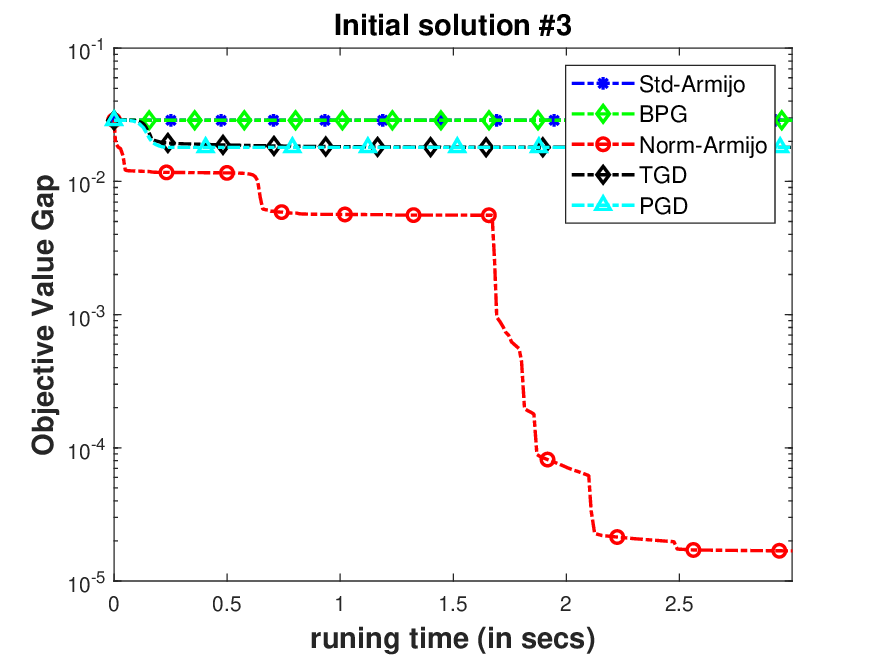}\hspace{-0.18cm}
	\includegraphics[width=0.25\linewidth]{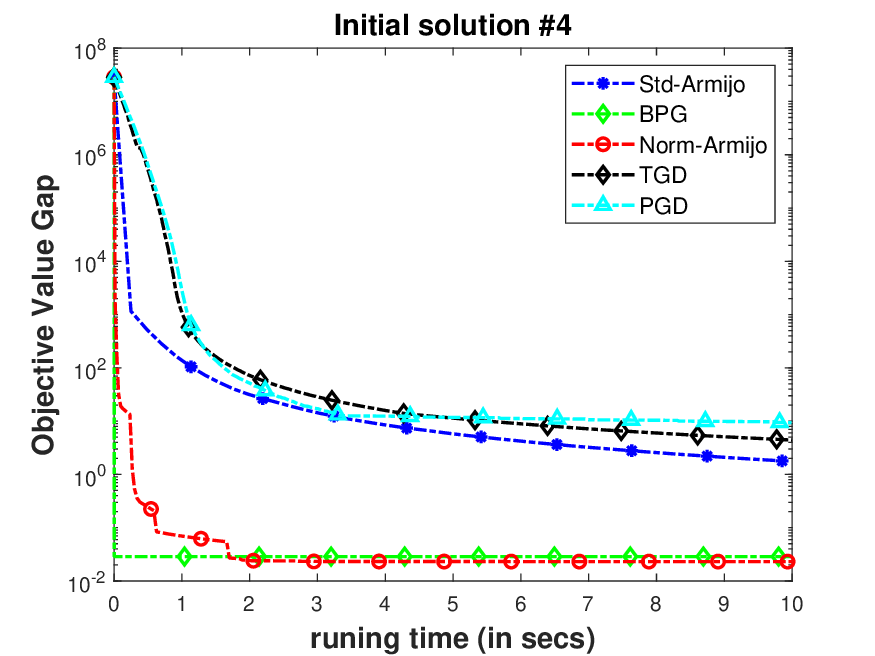}\\[1ex]
	\vfill  
	\includegraphics[width=0.255\linewidth]{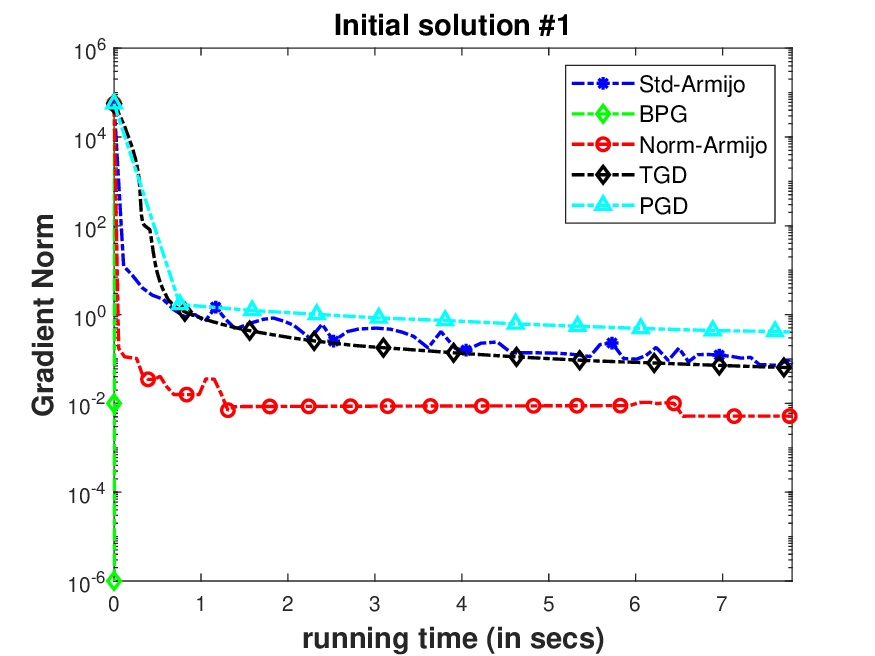} \hspace{-0.39cm}
	\includegraphics[width=0.255\linewidth]{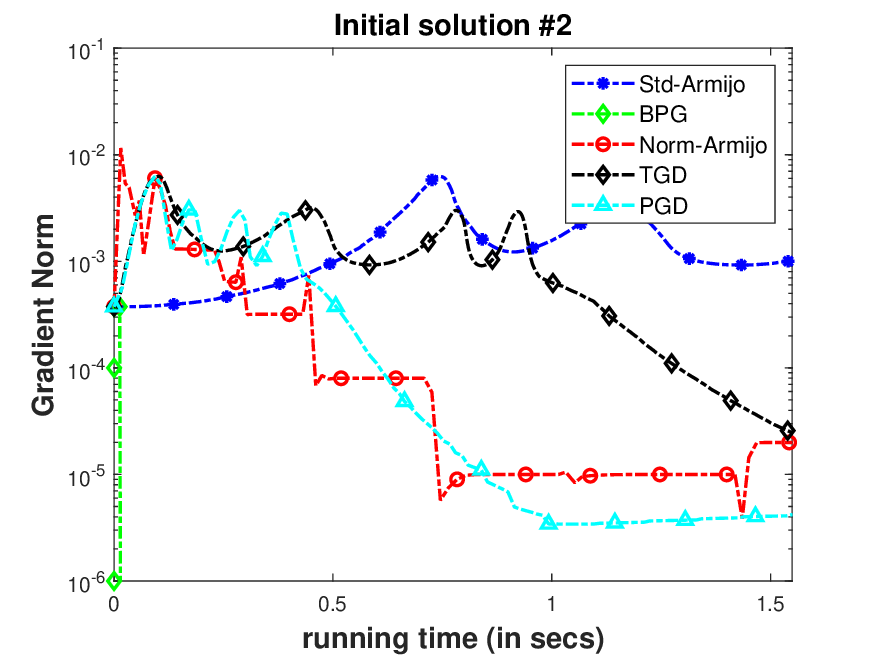}
	\hspace{-0.39cm}
	\includegraphics[width=0.255\linewidth]{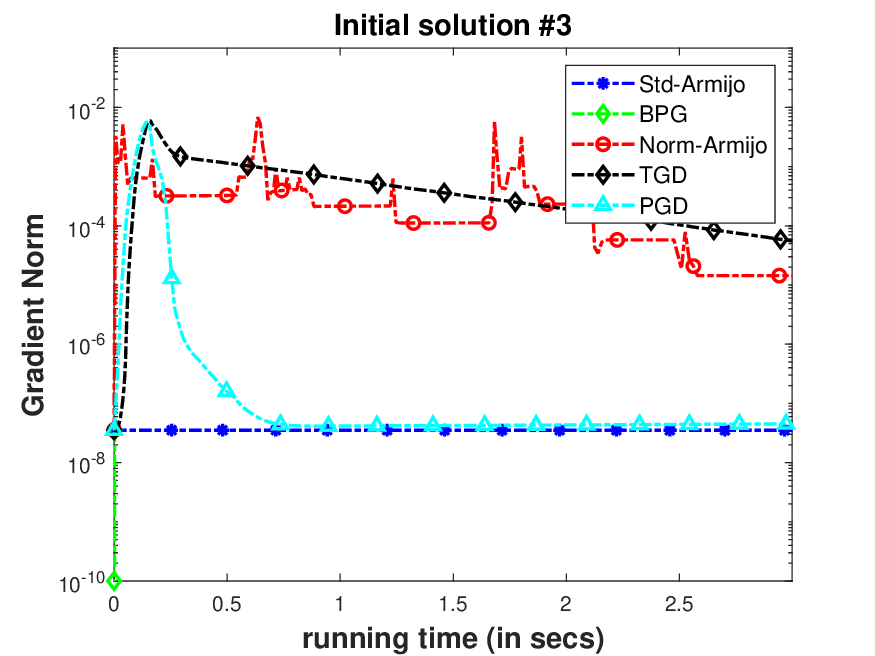}
	\hspace{-0.37cm}
	\includegraphics[width=0.255\linewidth]{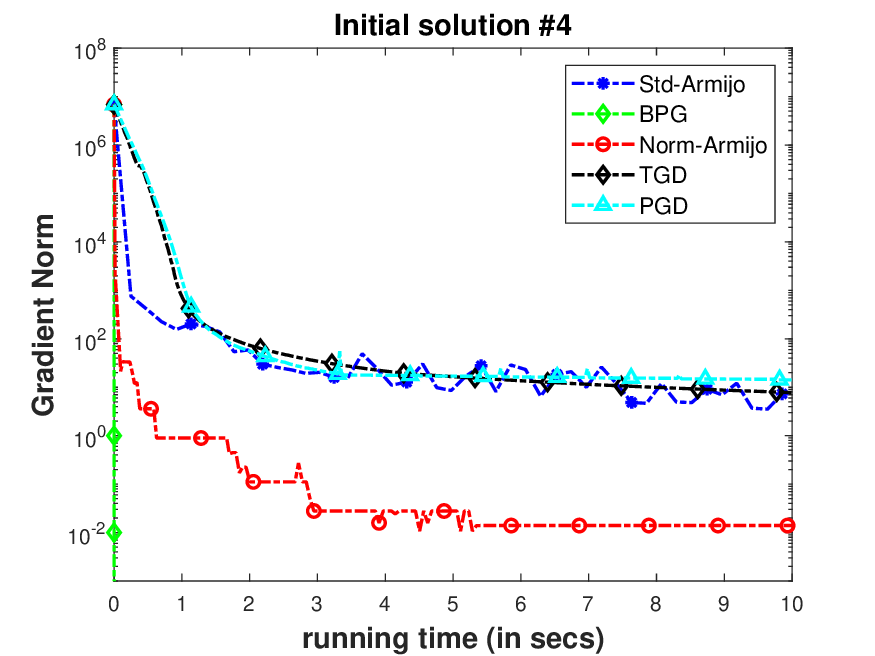}
	\caption{Numerical experiments on unsupervised training of linear autoencoder.}
	\label{fig:autoencoder}
\end{figure*}

\begin{table} 
	\begin{tabular}{ |p{2.5cm}||p{1.35cm}|p{1.35cm}|p{1.5cm}| p{1.4cm}||p{1.35cm}|p{1.35cm}|p{1.35cm}| p{1.2cm}| }
		\hline
		\multicolumn{1}{|c||}{}& \multicolumn{4}{|c||}{Initial solution entrywise$\sim$$\mathcal{U}$[0,0.001]} &  \multicolumn{4}{|c|}{Initial solution entrywise$\sim$$\mathcal{U}$[0,0.01]}\\
		\hline
		\multicolumn{1}{|c||}{}&\multicolumn{2}{|c|}{Gradient norm} &\multicolumn{2}{|c||}{Function value gap} &\multicolumn{2}{|c|}{Gradient norm} &\multicolumn{2}{|c|}{Function value gap}  \\
		\hline
		& $\,\,\,\,$best & average & $\,\,\,\,$best & average & $\,\,\,\,$best & average & $\,\,\,\,$best & average\\
		\hline
		Std-Armijo         & 8.1e$-$12   & 1.1e$-$11  & 2.9e$-$2 &  2.9e$-$2  & 5.4e$-$4   & 7.3e$-$3  & 8.2e$-$4 &  1.4e$-$2 \\
		BPG              &   \bf0  &  \bf0 & 2.9e$-$2 &  2.9e$-$2 & \bf0  &  \bf0 & 2.9e$-$2 &  2.9e$-$2 \\
		Norm-Armijo& 8.2e$-$12 & 1.1e$-$11&  \bf5.6e$-$3 & \bf1.4e$-$2 &8.1e$-$4 & 2.7e$-$3&  \bf4.4e$-$4 & \bf7.1e$-$4\\
		TGD            & 8.2e$-$12 & 1.1e$-$11   &  2.9e$-$2 & 2.9e$-$2  & 5.3e$-$3 & 8.6e$-$2   &  5.0e$-$2 & 1.4e$-$2\\ 
		PGD            & 1.5e$-$14 & 6.9e$-$14   &  1.8e$-$2 &1.8e$-$2 & 9.3e$-$\!3 & 7.4e$-$2       &  2.1e$-$2 &5.1e$-$2\\ 
		\hline
		\hline
		\multicolumn{1}{|c||}{}& \multicolumn{4}{|c||}{Initial solution entrywise$\sim$$\mathcal{U}$[0,0.1]} &  \multicolumn{4}{|c|}{Initial solution entrywise$\sim$$\mathcal{U}$[0,1]}\\
		\hline
		\multicolumn{1}{|c||}{}&\multicolumn{2}{|c|}{Gradient norm} &\multicolumn{2}{|c||}{Function value gap} &\multicolumn{2}{|c|}{Gradient norm} &\multicolumn{2}{|c|}{Function value gap}  \\
		\hline
		& $\,\,\,\,$best & average & $\,\,\,\,$best & average & $\,\,\,\,$best & average & $\,\,\,\,$best & average\\
		\hline
		Std-Armijo  & 3.2e$-$6   & 6.2e$-$6  & 1.8e$-$2 &  1.8e$-$2  & 4.0e$-$4   & 1.3e$-$2  & \bf4.7e$-$4 &  1.8e$-$2\\
		BPG         &   \bf0  &  \bf0 & 2.8e$-$2 &  2.8e$-$2 & \bf0  & \bf 0 & 2.9e$-$2 &  2.9e$-$2\\
		Norm-Armijo & 4.3e$-$7 & 2.6e$-$6&  \bf6.3e$-$\!9 & \bf2.3e$-$6 &  1.3e$-$3 & 4.2e$-$3&  1.5e$-$3 & \bf7.0e$-$3\\
		TGD            & 1.3e$-$6 & 6.0e$-$6   &  5.6e$-$3 & 9.9e$-$3 &1.1e$-$2 & 1.3e$-$1   &  2.6e$-$2 & 6.6e$-$2\\ 
		PGD            & 4.3e$-$7 & 6.2e$-$6       &  2.8e$-$8 & 2.1e$-$3 & 1.5e$-$2 & 1.0e$-$1       &  2.4e$-$2 &5.8e$-$2\\ 
		\hline
	\end{tabular}
	\caption{Numerical experiment of  training  linear autoencoder. We generate the initial solution elementwisely under 4 uniform distributions, i.e. $\mathcal{U}$[0,0.01], $\mathcal{U}$[0,0.1], $\mathcal{U}$[0,0.5] and $\mathcal{U}$[0,1]. $\qquad\qquad\,\,$$\qquad\qquad\qquad\,\,$}
	\label{table:autoencoder} \vspace{-0.2cm}
\end{table}

In this case, the optimal  value is unknown. To compute the objective value gap, we randomly initialize and run all $5$ methods for multiple times and enough number of iterations. Then we pick the minimal objective value obtained among all the observed iterations from all algorithms.  Figure \ref{fig:autoencoder} shows the curves of the tested algorithms from $4$ representative initial solutions. Table \ref{table:autoencoder} shows the results for initial solutions generated from 4 different distributions. The table contains the best and average result over 20 rounds of all methods. All the methods start from a same random initial point and run for 10 seconds per round. 

\subsection{Supervised linear neural network training}
For the third experiment, we consider the supervised training of deep linear neural network, where the activation function $\sigma(\cdot)$ is set to be the identity mapping. 
In this experiment, we use the same data matrix $X$ that has been used in the last section, which is randomly subsampled from ijcnn1 dataset. We choose the number of layers to be $m = 4$, with weights:  $W_4\in\RR^{1\times5}, W_3\in\RR^{5\times10}, W_2\in\RR^{10\times15}, W_1\in\RR^{15\times22}.$
We construct the random label $Y$ by first randomly generating the matrices $W_1^*,...,W_4^*$ and then setting $Y = W_4^*\cdots W_1^*X$. Thus the optimal value is 0. 

The way we present the numerical result is the same as that of the linear autoencoder example, thus the details are omitted. The results are summarized in Fig. \ref{fig:dnn} and Table \ref{table:dnn}.
\begin{figure*}
	\centering 
	\includegraphics[width=0.25\linewidth]{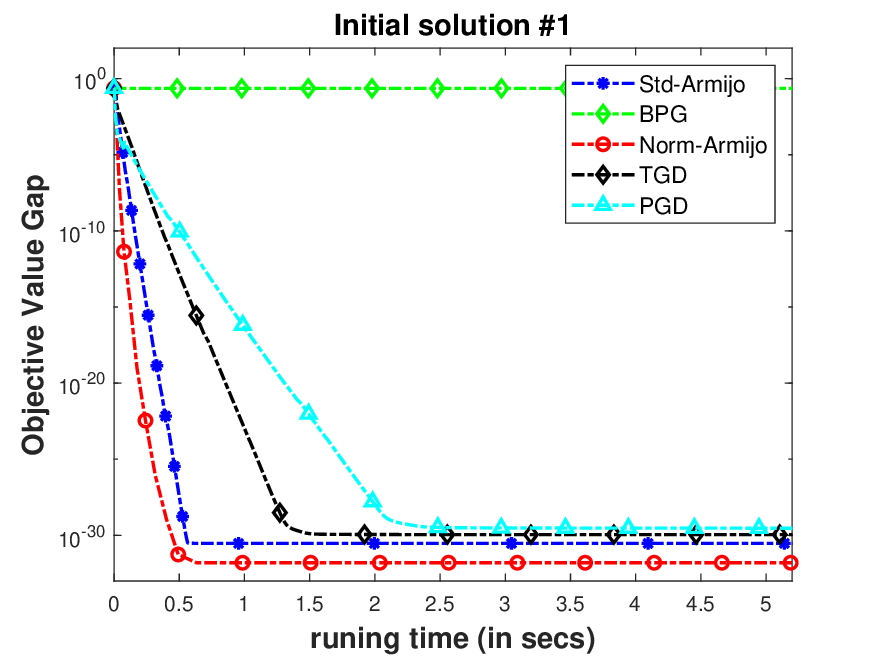}\hspace{-0.18cm}
	\includegraphics[width=0.25\linewidth]{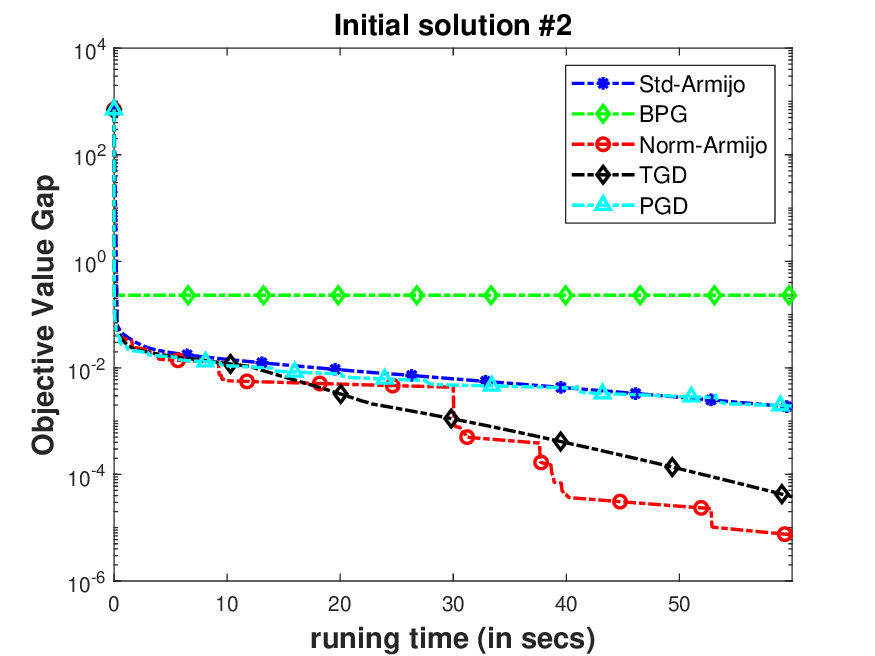}\hspace{-0.18cm}
	\includegraphics[width=0.25\linewidth]{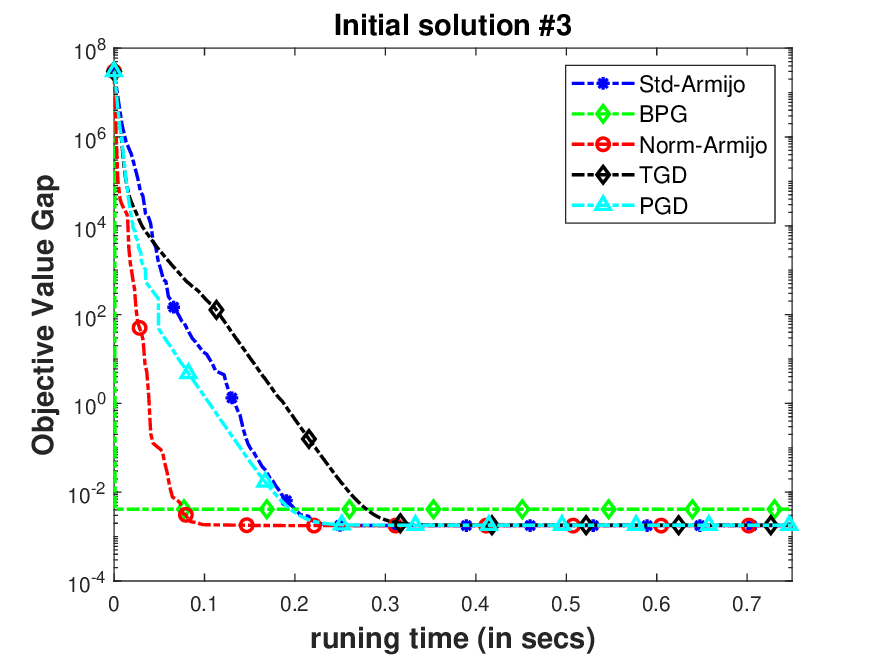}\hspace{-0.18cm}
	\includegraphics[width=0.25\linewidth]{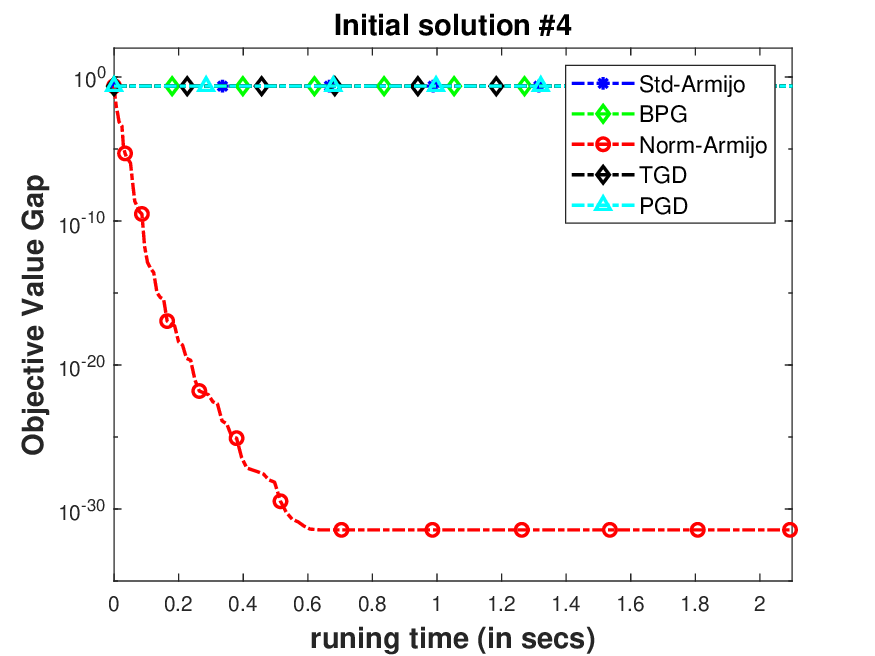}\\[1ex]
	\vfill  
	\includegraphics[width=0.255\linewidth]{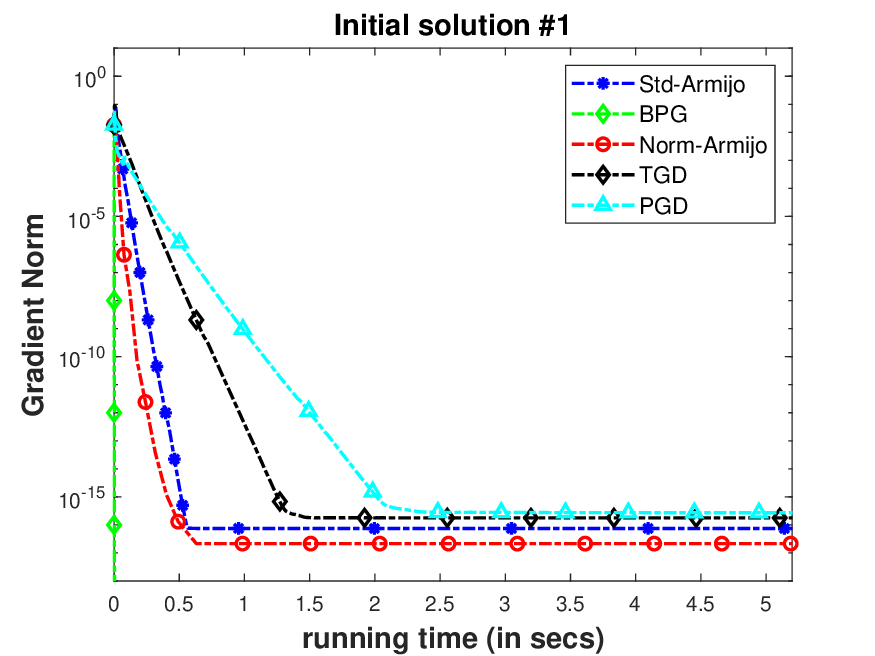} \hspace{-0.39cm}
	\includegraphics[width=0.255\linewidth]{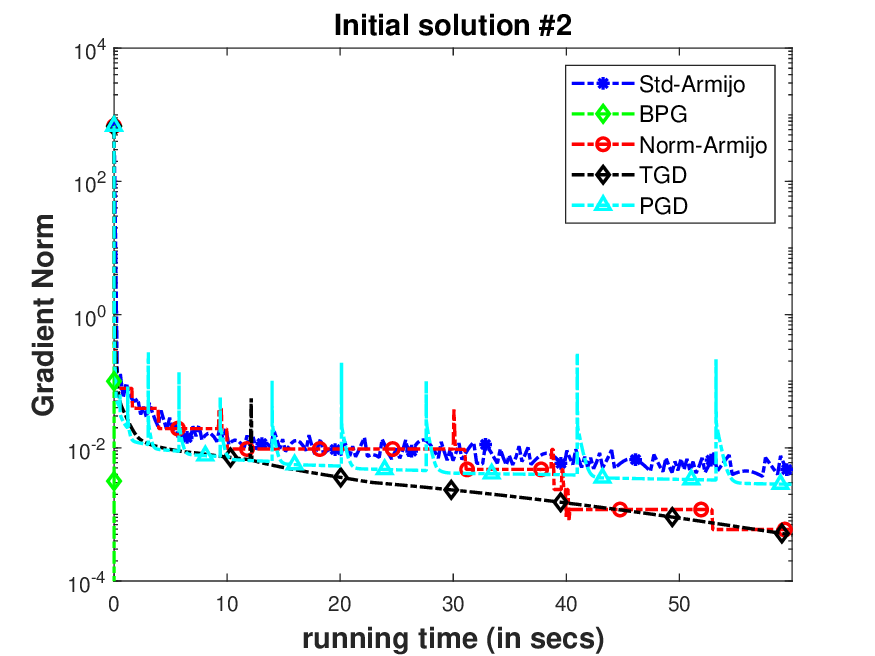}
	\hspace{-0.39cm}
	\includegraphics[width=0.255\linewidth]{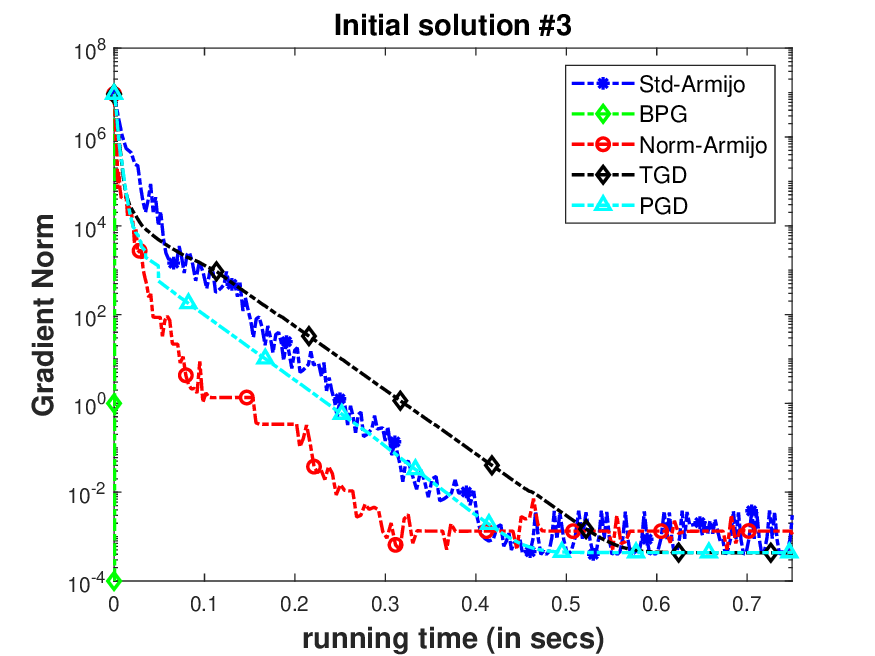}
	\hspace{-0.37cm}
	\includegraphics[width=0.255\linewidth]{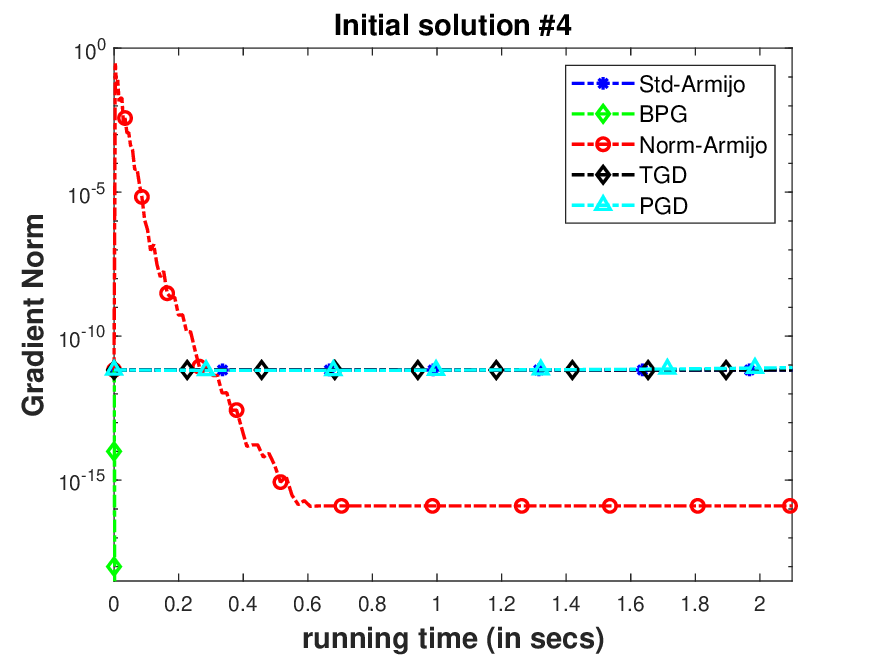}
	\caption{Numerical experiments on supervised training of linear neural networks.}
	\label{fig:dnn}
\end{figure*}

\begin{table}
	\begin{tabular}{ |p{2.5cm}||p{1.3cm}|p{1.3cm}|p{1.45cm}| p{1.4cm}||p{1.3cm}|p{1.3cm}|p{1.5cm}| p{1.45cm}| }
		\hline
		\multicolumn{1}{|c||}{}& \multicolumn{4}{|c||}{Initial solution entrywise$\sim$$\mathcal{U}$[0,0.001]} &  \multicolumn{4}{|c|}{Initial solution entrywise$\sim$$\mathcal{U}$[0,0.01]}\\
		\hline
		\multicolumn{1}{|c||}{}&\multicolumn{2}{|c|}{Gradient norm} &\multicolumn{2}{|c||}{Function value gap} &\multicolumn{2}{|c|}{Gradient norm} &\multicolumn{2}{|c|}{Function value gap}  \\
		\hline
		& $\,\,\,\,$best & average & $\,\,\,\,$best & average & $\,\,\,\,$best & average & $\,\,\,\,$best & average\\
		\hline
		Std-Armijo        & 3.0e$-$16   & 5.3e$-$8  & \bf1.3e$-$30 &  8.2e+0  & 5.5e$-$17 &  5.8e$-$17  & 1.1e$-$31 &  1.3e$-$31 \\
		BPG              &   \bf0  &  \bf0 & 1.8e$+$1 &  1.8e$+$1 & \bf0  &  \bf0 & 2.3e$-$1 &  2.3e$-$1\\
		Norm-Armijo & 3.5e$-$16 & 1.2e$-$15&  \bf1.3e$-$30 & \bf3.6e$-$30 & 8.7e$-$18 & 2.6e$-$17&  \bf1.3e$-$32 & \bf3.1e$-$32 \\
		TGD            & 2.3e$-$15 & 1.6e$-$8   &  5.1e$-$30& 2.7e+0  & 9.3e$-$16 & 1.1e$-$15   &  2.2e$-$29 & 3.3e$-$29\\ 
		PGD            & 1.9e$-$14 & 2.7e$-$14      &  2.7e$-$28 &5.4e$-$28& 2.1e$-$16 & 2.6e$-$16      &  1.1e$-$30 & 2.0e$-$30\\ 
		\hline
		\hline
		\multicolumn{1}{|c||}{}& \multicolumn{4}{|c||}{Initial solution entrywise$\sim$$\mathcal{U}$[0,0.1]} &  \multicolumn{4}{|c|}{Initial solution entrywise$\sim$$\mathcal{U}$[0,1]}\\
		\hline
		\multicolumn{1}{|c||}{}&\multicolumn{2}{|c|}{Gradient norm} &\multicolumn{2}{|c||}{Function value gap} &\multicolumn{2}{|c|}{Gradient norm} &\multicolumn{2}{|c|}{Function value gap}  \\
		\hline
		& $\,\,\,\,$best & average & $\,\,\,\,$best & average & $\,\,\,\,$best & average & $\,\,\,\,$best & average\\
		\hline
		Std-Armijo         & 9.6e$-$17  & 2.2e$-$14  & 1.5e$-$29 &  1.2e$-$22  & 3.8e$-$15   & 3.9e$-$2  & \bf5.6e$-$28 &  \bf3.5e$-$2\\
		BPG              &  \bf0  &  \bf0 & 3.4e$-$3 &  3.7e$-$3 &  \bf0  &  \bf0 & 2.3e$-$1 &  2.3e$-$1\\
		Norm-Armijo & 2.0e$-$19 & 6.7e$-$19 &  \bf3.6e$-$34 & \bf8.4e$-$34 &  9.8e$-$15 & 7.0e$-$2 &  2.7e$-$27 & 5.0e$-$2\\
		TGD            & 7.8e$-$19 & 1.3e$-$18   &  9.7e$-$34 & 2.0e$-$33 & 5.0e$-$2 & 8.9e$-$2   &  3.4e$-$2 & 1.3e$-$1\\ 
		PGD            & 2.6e$-$18 & 7.8e$-$18       &  7.5e$-$33 & 7.5e$-$32 & 3.8e$-$2 & 7.7e$-$2       &  5.1e$-$2 &1.1e$-$1\\ 
		\hline
	\end{tabular}
	\caption{Numerical experiment of supervised training of linear neural networks. We generate the initial solution elementwisely under 4 uniform distributions, i.e. $\mathcal{U}$[0,0.001], $\mathcal{U}$[0,0.01], $\mathcal{U}$[0,0.1] and $\mathcal{U}$[0,1].  }
	\label{table:dnn}
\end{table}

In this experiment, the behavior of BPG is similar to the previous examples. It often gets stuck at some poor stationary point while the other methods can find stationary points with lower objective values. In fact, in most cases, the best results returned by Std-Armijo and Norm-Armijo reach the machine precision, while TGD and PGD only fail to find a high precision solution for the $\mathcal{U}$[0,1] initialization. It is also worth noting that the two line search methods behave comparably in this experiment except for the average function value gap with the  $\mathcal{U}$[0,0.001] initialization. It can be observed that our Norm-Armijo approach behaves much better than the Std-Armijo under this initialization. An interesting illustration of cases where Std-Armijo fails can be found in the Figure \ref{fig:dnn} with initial point \# 4. In this figure, the initial solution is almost a stationary point (gradient norm less than $10^{-10}$). BPG quickly converges to the exact stationary point nearby; Std-Armijo, TGD and PGD almost do not move, while only Norm-Armijo quickly escapes this point and converges to the global optimal solution. The behavior of Norm-Armijo in this case is mainly attributed to the normalized search region, which not only serves the role of preventing too aggressive update when the gradient is large, but also allows bigger steps when the gradient is very small.

\bibliographystyle{plain}
\bibliography{ref_final}

\end{document}